\documentclass[3p,10pt,nopreprintline]{elsarticle}

\usepackage{amsmath}
\usepackage{amsthm}
\usepackage{amssymb}
\usepackage{graphicx}
\usepackage{epsfig}
\usepackage[normal]{subfigure}
\usepackage{float}
\usepackage{paralist}
\usepackage{caption2}
\usepackage{hyperref}
\usepackage{aliascnt}
\usepackage[T1]{fontenc}
\usepackage{natbib}
\usepackage{CJK}
\usepackage{color}
\usepackage[table]{xcolor}
\usepackage{colortbl,booktabs}

\newtheoremstyle{theorem}{6pt}{6pt}{\itshape}{}{\bfseries}{.}{.5em}{}
\newtheoremstyle{definition}{6pt}{6pt}{\upshape}{}{\bfseries}{.}{.5em}{}

\theoremstyle{theorem}
\newtheorem{theorem}{Theorem}[section]

\newaliascnt{corollary}{theorem}

\aliascntresetthe{corollary}

\newaliascnt{lemma}{theorem}
\newtheorem{lemma}[lemma]{Lemma}
\aliascntresetthe{lemma}

\newaliascnt{sublemma}{theorem}

\aliascntresetthe{sublemma}

\theoremstyle{definition}

\newtheorem{remark}{Remark}[section]
\newaliascnt{proposition}{theorem}
\newtheorem{proposition}[proposition]{Proposition}
\aliascntresetthe{proposition}

\renewcommand{\u}{{\bf u}}

\newcommand{\R}{{\mathbb R}}

\newcommand{\dif}{{\mathrm d}}

\newcommand{\bn}{\begin{eqnarray}}
\newcommand{\en}{\end{eqnarray}}
\newcommand{\bnn}{\begin{eqnarray*}}
\newcommand{\enn}{\end{eqnarray*}}
\renewcommand{\div}{ {\rm div }  }

\newcommand{\N}{\mathbb{N}}
\newcommand{\rL}{\mathrm{L}}
\newcommand{\T}{\mathbb{T}}

\newcommand{\al}{\alpha}

\newcommand{\frD}{\mathfrak{D}}

\newcommand{\calE}{\mathcal{E}}
\newcommand{\pt}{\partial_t}

\newcommand{\D}{\nabla}

\newcommand{\vr}{\varrho}

\allowdisplaybreaks

\numberwithin{equation}{section}

\usepackage{scalerel,stackengine}
\stackMath
\newcommand\reallywidehat[1]{%
\savestack{\tmpbox}{\stretchto{%
  \scaleto{%
    \scalerel*[\widthof{\ensuremath{#1}}]{\kern-.6pt\bigwedge\kern-.6pt}%
    {\rule[-\textheight/2]{1ex}{\textheight}}%WIDTH-LIMITED BIG WEDGE
  }{\textheight}% 
}{0.5ex}}%
\stackon[1pt]{#1}{\tmpbox}%
}

\begin{document}

\begin{frontmatter}

\title{Global Existence and Asymptotic Behavior of Large Strong Solutions \\ to the 3D Full Compressible Navier-Stokes Equations \\ with Density-dependent Viscosities}

\author[label1]{Yachun Li}
\address[label1]{School of Mathematical Sciences, CMA-Shanghai, MOE-LSC, and SHL-MAC,\\ Shanghai Jiao Tong University, Shanghai 200240, P.R.China;}
\ead{ycli@sjtu.edu.cn}

\author[label2]{Peng Lu}
\address[label2]{School of Mathematical Sciences, Shanghai Jiao Tong University, Shanghai 200240, P.R.China;}
\ead{lp95@sjtu.edu.cn}

\author[label3]{Zhaoyang Shang\corref{cor2}}
\address[label3]{School of Finance, Shanghai Lixin University of  Accounting and Finance, Shanghai 201209, P.R.China;}
%\address[label3]{School of Mathematical Sciences, Shanghai Jiao Tong University, Shanghai 200240, P.R.China;}
\cortext[cor2]{Corresponding author. }
\ead{shangzhaoyang@sjtu.edu.cn}

\author[label2]{Shaojun Yu}
%\address[label3]{School of Mathematical Sciences, Shanghai Jiao Tong University, Shanghai 200240, P.R.China;}
%\address[label3]{School of Mathematical Sciences, Shanghai Jiao Tong University, Shanghai 200240, P.R.China;}
%\cortext[cor2]{Corresponding author. }
\ead{edwardsmith123@sjtu.edu.cn}

\begin{abstract}
The purpose of this work is to investigate the Cauchy problem of global-in-time existence of large strong solutions to the Navier–Stokes equations for compressible viscous  and heat conducting fluids. A class of density-dependent viscosity is considered. By introducing the modified  effective viscous flux and using the bootstrap argument, we establish the global existence of large strong solution when the initial density is linearly equivalent to a large
constant state. It is worthy of mentioning that, different from the work of Matsumura and Nishida (J. Math. Kyoto Univ., 1980) with small initial perturbation and the work of Huang and Li (Arch. Ration. Mech. Anal., 2018) with small energy but possibly large oscillations, our global large strong solution is uniform-in-time in $H^2$ Sobolev space and the uniform-in-time bounds of both density and temperature are obtained without any
restrictions on the size of initial velocity and initial temperature. In addition, when the initial data belongs to $L^{p_0}\cap H^2$ with $p_0\in[1,2]$, we establish the convergence of the solution to its associated equilibrium with an explicit
decay rate whether the initial data close to or far away from the equilibrium in the whole space. As a result, we give a specific large strong solution in Sobolev space satisfying the global existence assumptions proposed by  Villani (Mem. Amer. Math. Soc., 2009), He, Huang, and  Wang (Arch. Ration. Mech. Anal., 2019), Zhang and Zi (Ann. Inst. H. Poincaré Anal. Non Linéaire, 2020) in studying the asymptotic behavior of solution, and extend the second above result to the nonisentropic case. This paper considers for the first time the application of Fourier splitting method to Navier-Stokes equations with variable viscosity.

 \end{abstract}


\end{frontmatter}



\section{Introduction}
The motion of a compressible viscous, heat-conductive, and Newtonian polytropic fluid is governed by the following full
compressible Navier-Stokes equations:
\begin{equation}\label{FCNS}
\begin{cases}
\pt\rho+\div(\rho u)=0,\\[6pt]
\pt(\rho u)+\div(\rho u\otimes u)+\D P=\div\T,\\[6pt]
\pt(\rho\calE)+\div(\rho\calE u+Pu)=\div(u\T)+\div(\kappa\D\theta),
\end{cases}
\end{equation}
where $x=(x_1,x_2,x_3)^\top\in\R^3$ and $t\geq 0$ denote the spatial coordinate and time coordinate, $\rho$ the mass density, $u = (u^1, u^2, u^3)^\top$ the fluid velocity, $P$ the pressure of the fluid, $\theta$ the absolute temperature, $\calE=e+\frac{1}{2}|u|^2$ the specific total energy, $e$ the specific internal energy. The equation of state for polytropic gas satisfies
\begin{equation}\label{Pecv}
P=R\rho\theta, \quad e=c_v\theta,
\end{equation}
where $R>0$ is the gas constant, and $c_v>0$ is the specific heat at
constant volume. $\T$ is the viscosity stress tensor given by $\T=2\tilde{\mu}(\rho)\frD(u)+\tilde{\lambda}(\rho
)\div u\mathbb{I}_3$, where 
$$\frD(u) =\frac{\D u+(\D u)^\top}{2}$$
is the deformation tensor, $\mathbb{I}_3$ is the $3 \times 3$ identity matrix, $\tilde{\mu}$ is the shear viscosity coefficient, and $\tilde{\lambda}+\frac{2}{3}\tilde{\mu}$ is the bulk viscosity coefficient. $\kappa$ denotes the coefficient of heat conductivity.
% In the theory of gas dynamics, the compressible Navier-Stokes equations can be derived from the Boltzmann equations through the Chapman–Enskog expansion, see Chapman–Cowling [] and Li–Qin []. Under some proper physical assumptions, the viscosity coefficients $(\tilde{\mu}, \tilde{\lambda})$ and the coefficient of thermal conductivity $\tilde{\kappa}$ are not constants but functions of the density $\rho$ and the absolute temperature $\theta$, such as
% \begin{equation}
% \tilde{\mu}(\theta)=r_1\theta^{\frac{1}{2}}F(\theta), ~ \tilde{\lambda}=r_2\theta^{\frac{1}{2}}F(\theta), ~ \tilde{\kappa}=r_3\theta^{\frac{1}{2}}F(\theta),
% \end{equation}
% for some constants $r_i ~ (i = 1, 2, 3)$. 
In this paper, we consider the following case:
\begin{equation}\label{temperature-dependent}
\tilde{\mu}(\rho)=\mu\rho^\al, \quad \tilde{\lambda}(\rho)=\lambda\rho^\al, \quad \kappa=constant>0,
\end{equation}
where $\mu, \lambda$ and $\al$ are constants satisfying
\begin{equation}\label{constants}
\mu>0, \quad 2\mu+3\lambda\geq 0, \quad \al>0.
\end{equation}

There have been extensive literature on the well-posedness for the  compressible Navier-Stokes equations \eqref{FCNS}  in multi-dimensional space.  The local existence and uniqueness of strong and classical solutions in bounded and unbounded domain were established in \cite{MR0283426, MR0149094, MR0106646, ta}. 
For the case of global well-posedness in three dimensional space, when the viscosity and heat conductivity coefficients are constants, in 1980s, Matsumura and Nishida \cite{MR555060,MR0564670} first demonstrated the global strong solutions with the initial data close to a non-vacuum equilibrium in Sobolev space. Later, Hoff \cite{MR1480244} extended their results to the cases with discontinuous initial data, provided that the initial density remains strictly positive. In 1996, Jiang \cite{MR1389908} considered the equations in the domain exterior to a ball and proved the global existence of spherically symmetric smooth solutions for large initial data with spherical symmetry. In 1998, Xin \cite{MR1488513} presented a sufficient condition on the blowup of smooth solution in arbitrary space dimensions with initial density of compact support. It is shown that any smooth solution for polytropic fluids in the absence of heat conduction will blow up in finite time as long as the initial density has compact support.
In 2004, Hoff and Jenssen \cite{MR2091508} proved the global existence of weak solutions with initial data and external forces which are large, discontinuous, and spherically or cylindrically symmetric. In 2013, Xin and Yan \cite{MR3063918} improved the blowup results of \cite{MR1488513} by removing the crucial assumptions that the initial density has compact support and the smooth solution has finite total energy. In 2017,  Wen and Zhu \cite{MR3597161} investigated the Cauchy problem and showed that the strong solution exists globally in time if the initial mass is small for the fixed coefficients of viscosity  $\mu$ and heat conductivity $\kappa$, and can be large for the large coefficients of viscosity and heat conductivity.  In the next year, Huang and Li \cite{MR3744381} proved the global existence of classical solution to the Cauchy problem with smooth initial data which are of small energy but possibly large oscillations. Recently, Chen, Huang and Zhu \cite{MR4813232} established the global-in-time existence of solution of the Cauchy problem with initial data that are large, discontinuous, spherically symmetric, and away from the vacuum. The solution they obtained is of global finite total energy and on any region strictly away from the possible vacuum, the velocity and specific internal energy are Hölder continuous, and the density has a uniform upper bound. For the case of global well-posedness in 1D and 2D spaces, we refer to 
\cite{MR637519, MR651877, MR4346514, MR4039142, MR4491875, MR4740640, MR4008547}, etc., and the references cited therein.

However, viscosity and the heat conductivity coeﬃcients may not be constants, they may depend on density or temperature, or both. At present, there are much fewer results on global well-posedness to the 3D full compressible Navier-Stokes equations with variable viscosities. First, we recall some known results of isentropic case. In two dimensional space, when initial vacuum is not allowed, in 1995, Vaigant and Kazhikhov \cite{vk} studied the  initial-boundary value problem with density-dependent viscosity, that is, 
\begin{equation}\label{1.30-2-1}
\tilde{\mu}=constant >0,\quad\tilde{\lambda}(\rho)=\rho^\beta,\quad\beta>3,%\quad \Omega=[0,1]\times[0,1],
\end{equation}
and proved the existence of global large strong solution. Later, in 2016, Huang and Li \cite{MR3505779} relaxed the power index $\beta$ to be $\beta>\frac{4}{3}$, and the related Cauchy problem can be found in \cite{MR4429416}. When $0<\beta\leq 1$, Fang and Guo \cite{liguo} established the global well-posedness of the strong solution to the Cauchy problem provided that the initial data are of small total energy, which was further improved by Ding, Huang and Liu \cite{MR3844919} to a classical solution in 2018. When the initial vacuum is allowed, we refer \cite{MR4444073, MR4732164, MR3200422, MR3053470} to the interested readers. In three dimensional space, a remarkable discovery of a new mathematical entropy function was made by Bresch and Desjard\^ins \cite{Bre2003Exi}, in 2003,
for $\lambda(\rho)$ and $\mu(\rho)$ satisfying the relationship
\begin{equation}
\tilde{\lambda}(\rho)=2(\tilde{\mu}^\prime(\rho)\rho-\tilde{\mu}(\rho)),
\end{equation}
which provides additional regularity on some derivative of the density. This observation was applied widely in proving the global existence of weak solutions with vacuum for isentropic Navier-Stokes equations and some related models. Based on the methods in \cite{Bre2003Exi}, in 2016, Vasseur-Yu \cite{MR3573976} proved the existence of global weak solutions for 3D
compressible Navier–Stokes equations with degenerate viscosities for $1<\gamma<3$ (The result holds for $\gamma>1$ in two dimensional space). For more results about global existence of weak solutions, we refer to  \cite{Bresch2022Glo, MR2254008, liweek, Mellet2007On}, etc.  In 2021, Xin and Zhu \cite{Xin2021Glo}  considered viscosities satisfying 
\begin{equation}\label{temperature-dependent-1}
\tilde{\mu}(\rho)=\mu\rho^\al, \quad \tilde{\lambda}(\rho)=\lambda\rho^\al,
\end{equation}
with $\alpha>1$ and established the global existence of classical solution for the Cauchy problem with vacuum under some initial smallness assumptions on density in homogeneous Sobolev space.
% Later, Cao-Li-Zhu \cite{Cao2022Glo} proved the global existence of 1D classical solution with large initial data and vacuum for $0<\alpha\leq 1$.
In 2023, Yu \cite{MR4589926} considered the Cauchy problem with viscosities as in (\ref{temperature-dependent-1}) where the parameters satisfying 
\begin{equation}
\frac{4}{3}\leq\gamma\leq\alpha\leq\frac{5}{3}, \quad\alpha+4\gamma>7,\quad \alpha+\gamma \leq 3.
\end{equation}
Under the condition that the lower bound of initial density is sufficiently large, he proved
that strong solution exists globally in time for large initial data. Recently, Cao, Li and Zhu \cite{MR4811168} considered the initial-boundary value problem in the domain exterior to a ball in $\R^d ~ (d=2, 3)$. When $\alpha=1$ in \eqref{temperature-dependent-1}, they proved the global existence of the unique spherically symmetric classical solution for large initial data with spherical symmetry and far field vacuum. For the results of isentropic Navier-Stokes equations in one dimensional space, we refer to  \cite{Burtea2020New, Cao2022Glo,Con2020Com, Haspot2018Exi, Kang2020Glo, Mellet2007Exi}, and the references therein.

Next, we turn to the results of nonisentropic compressible Navier-Stokes equations with variable viscosity, and the issues become much more complicated. In 1985, Kawohl \cite{MR791841} proved the global existence of large classical solution to initial-boundary value problem in one dimensional space when the viscosity and the heat conductivity satisfy 
\begin{equation}\label{temperature-dependent-2}
0<\mu_0\leq\mu(v)\leq \mu_1,\quad \kappa_0(1+\theta^q)\leq \kappa(v,\theta)\leq \kappa_1(1+\theta^q), \quad |\kappa_v(v,\theta)|+|\kappa_{vv}(v,\theta)|\leq \kappa_1(1+\theta^q),
\end{equation}
where $v$ is specific volume and $q\geq2$.
In 2014, Liu, Yang, Zhao and Zou \cite{MR3225502} investigated the temperature-dependent transport coefficients and established the Nishida–Smoller type global smooth solution to Cauchy problem for the large data under the assumption that $(\gamma-1)\|(v_0-1, u_0, \frac{\theta_0}{\sqrt{\gamma-1}})\|_{H^3}< C $ for some positive constant $C$ which does not depend on $\gamma$. 
In 2016, Wang \cite{MR3461630} studied the initial and initial-boundary value problems  for the
$p-$th power Newtonian fluid  and established the existence and uniqueness of global smooth non-vacuum solutions when the transport coefficients $\mu$ and $\kappa$ are given as 
\begin{equation}\label{temperature-dependent-4}
\tilde{\mu}(\rho)=\rho^\al, \quad \tilde{\kappa}(\theta)=\theta^\beta,
\end{equation}
with some positive parameters $\alpha$ and $\beta$.
In the same year, Wang and Zhao \cite{MR3564590} considered the Cauchy problem when the viscosity $\mu$ and the heat conductivity $\kappa$ depend on the density $\rho$ and the temperature $\theta$ and are both proportional to $h(\rho)\theta^\al$ for certain non-degenerate smooth function $h$. They proved the existence and uniqueness of a global-in-time non-vacuum solution under certain assumptions on the parameter $\al$ and initial data, which imply that the initial data can be large if $|\al|$ is sufficiently small. It is the first global existence result for general adiabatic exponent and large initial data when the viscosity coefficient depends on both density and temperature. In 2017, Duan, Guo and Zhu \cite{MR3603273} considered the initial-boundary value problem with the stress-free and heat insulated boundary condition, and established the global existence of strong solution  with  the initial density is away from vacuum and the following density-dependent viscosity and temperature-dependent heat conductivity 
\begin{equation}\label{temperature-dependent-3}
\mu(\rho)=1+\rho^\alpha, \quad\kappa(\theta)=\theta^\beta,\quad \alpha\geq 0,\quad \beta>0.
\end{equation}
For the three dimensional case, in 2007, Bresch and Desjard\^ins \cite{MR2297248} considered the Cauchy problem or periodic problem and proved that the weak solutions exist globally in time under the condition that
$$\lambda(\rho)=2(\mu'(\rho)\rho-\mu(\rho)),\quad \kappa(\rho,\theta)=\kappa_0(\rho,\theta)(1+\rho)(1+\theta^\al),\quad \al\geq 2.$$
In 2020, when the coefficients of viscosity depend on density and temperature, Yu and Zhang \cite{MR4079010} studied the Dirichlet boundary value problem and proved the strong solution exists globally in time provided that $\|\nabla u_0\|_{L^2}^2+\|\nabla \theta_0\|_{L^2}^2$ is suitably small. 
%Some other interesting results and discussions can be seen in \cite{}(Zhu and Xin nonisentropic )
From the above mentioned known results,
as far as our concerned, there are few results for the case of 3D full compressible Navier-Stokes equations with variable viscosity and large initial data. In this paper, we are interested in investigating the global well-posedness  and  asymptotic behavior of strong solution to the system \eqref{FCNS} with density-dependent viscosities and large initial data.  Without loss of generality, we assume the constant $c_v=1$ in \eqref{Pecv}. It follows from \eqref{Pecv}--\eqref{temperature-dependent} that \eqref{FCNS} can be rewritten as
\begin{equation}\label{FCNS:2}
\begin{cases}
\pt\rho+\div(\rho u)=0,\\[6pt]
\rho(\pt u+(u\cdot\D)u)+\D P=2\mu\div(\rho^\al\frD(u)) +\lambda\D(\rho^\al\div u),\\[6pt]
\rho(\pt\theta+u\cdot\D\theta)+P\div u=2\mu\rho^\al |\frD(u)|^2+\lambda\rho^\al(\div u)^2+\kappa\Delta\theta.
\end{cases}
\end{equation}
When system \eqref{FCNS:2} is supplemented with the following initial data and far field behavior
\begin{equation}\label{data and far}
\begin{aligned}
& (\rho, u, \theta)(0, x)=(\rho_0, u_0, \theta_0)(x) \quad \text{for} \quad x\in\R^3, \\
& (\rho, u, \theta)(t, x) \to (\vr, 0, 1) \quad \text{as} \quad |x|\to\infty \quad \text{for} \quad t \geq 0,
\end{aligned}
\end{equation}
where $\vr$ is a positive constant, we give our first result concerning the global existence of strong solution to Cauchy problem \eqref{FCNS:2}--\eqref{data and far} with large initial data.
\begin{theorem}\label{th} Let the constants $\vr>1$ and $\al>4$. Suppose the initial data $(\rho_0,u_0,\theta_0)$ satisfy 
\begin{equation}\label{ic-1}
\frac{3}{4}\vr\leq\rho_0(x)\leq\frac{5}{4}\vr, \quad \underline{\theta}\leq\theta_0(x)\leq\bar{\theta}, \quad \left(\rho_0(x)-\vr, u_0(x), \theta_0(x)-1\right)\in H^2,
\end{equation}
\begin{equation}\label{compatibility condition}
\|\rho_0-\vr\|_{H^2}+\|\D\rho_0\|_{L^4}\leq C_0, \quad 
\end{equation}
with positive constants $C_0$, $\bar{\theta}$ and $\underline{\theta}$ independent of $\vr$, then there exists a constant $\rL>0$ independent of $\vr$, such that if $\vr\geq\rL$, Cauchy problem \eqref{FCNS:2}--\eqref{data and far} admits a global strong solution $(\rho,u,\theta)$ satisfying
\begin{equation}\label{sp}
\begin{cases}
\displaystyle
\left(\rho-\vr, u, \theta-1\right)\in L^\infty(0,\infty; H^2), \quad u\in L^2(0,\infty;W^{2,4}),\\[6pt]
\displaystyle
\displaystyle
( { u}_t,\theta_t)\in L^2(0,\infty;H^1)\cap L^\infty(0,\infty;L^2), \quad \rho_t\in L^\infty(0,\infty;H^1).
\end{cases}
\end{equation}
and
\begin{equation}\label{bc}
\frac{2}{3}\vr\leq\rho(x,t)\leq\frac{4}{3}\vr, \quad \underline{\Theta}\leq\theta(x,t)\leq\overline{\Theta}, \quad \forall ~ (x,t)\in\R^3\times(0,\infty),
\end{equation}
with positive constants $\underline{\Theta}$ and $\overline{\Theta}$ depending on $\bar{\theta}$, $\underline{\theta}$, $\mu$, $\lambda$, $\kappa$, $R$, $\vr$ and the initial data. Moreover, the following large time behavior holds
\begin{equation}\label{1010}
    \lim_{t\rightarrow \infty} \left(\|\D\rho\|_{H^1}+\|\nabla\theta\|_{H^1}+\|\nabla u\|_{H^1}\right)=0.
\end{equation}
\end{theorem}

\begin{remark}
    We would like to mention that Theorem \ref{th} is the first result concerning the global existence of strong solutions  for the Cauchy problem to the 3D full Navier-Stokes equations with variable viscosities and large initial data. Our global strong solution is uniform-in-time in solution space (\ref{sp}). Moreover, for high dimensional case, the uniform-in-time lower and upper bounds of the density and temperature are also given in (\ref{bc}).
\end{remark}
% \begin{remark}
%     The condition that $\vr\geq L$  implies the viscosities are  large enough to ensure the global existence of large solution.
% \end{remark}

\begin{remark}
    The choice of $\alpha>4$ is a technical requirement, which is used to dealing with the leading terms in the a priori estimates. It should be mentioned that it seems that $\alpha>4$ is not a sharp assumption. Therefore, it would be interesting to study the problem \eqref{FCNS:2}--\eqref{data and far} when $\alpha\leq4$. This is left for interested readers. 
\end{remark}

\begin{remark}
In Theorem \ref{th}, we consider the coefficients of viscosity depend on the density and  the coefficient of heat conductivity is a constant. This is certainly a restriction which is not physically motivated, experimental evidence pointed out that the coefficients of viscosity and heat
conductivity  usually depend on both density and temperature. In particular for high temperature case, the transport coefficients increase such as 
\begin{equation}\label{t}
\tilde{\mu}(\theta)={\mu}\theta^{\alpha}, \quad \tilde{\lambda}(\theta)={\lambda}\theta^{\alpha},\quad \tilde{\kappa}(\theta)={\kappa}\theta^{\beta}, \quad \alpha>0,\quad \beta>0.
\end{equation}
   % In fact, temperature plays a dominant role in the effect of the transport coefficients compared to density. 
   Under this physical background, the Cauchy problem \eqref{FCNS} with  temperature-dependent viscosities (\ref{t}) is taken into account. The initial data and far field behavior are given as follows:
\begin{equation}\label{data and far2}
\begin{aligned}
& (\rho, u, \theta)(0,x)=(\rho_0, u_0, \theta_0)(x) \quad \text{for} \quad x\in\R^3, \\
& (\rho , u, \theta)(t, x) \to (1, 0, \bar{\theta}) \quad \text{as} \quad |x|\to\infty \quad \text{for} \quad t \geq 0,
\end{aligned}
\end{equation}
where $\bar{\theta}$ is a positive constant. In our forthcoming paper, the global well-posedness of large strong solution is established provided that the initial temperature is
linearly equivalent to a large constant state.
\end{remark}

\begin{remark}
We notice that there are some global existence results of nonisentropic Navier-Stokes equations with density-dependent viscosity in one dimensional space, see \cite{MR3603273,MR3461630}. It should be pointed out that, by using the similar method of this paper, our results can be proved in a bounded domain. From mathematical point of view, 
we extend their results to the 3D case. Moreover, our method can be applied to the constant viscosities case with small initial energy which was studied in \cite{MR3744381} where there are some restrictions on the size of initial velocity and temperature, while in this paper, we only pose restriction on the size of initial density.

% Moreover,  we give a specific large strong solution in Sobolev space satisfying the global existence assumption proposed by  Villani \cite{MR2562709}, He, Huang, and  Wang \cite{He2019Glo}, Zhang and Zi \cite{Zhang2020Con} in studying the asymptotic behavior of large solution, and extend the result in \cite{He2019Glo} to the nonisentropic case with variable viscosities.
\end{remark}

After obtaining the global existence of large strong solution, motivated by the above mentioned results in the whole space, it is natural to investigate the asymptotic behavior of the solution which far away from  the equilibrium initially. Next, we recall some related works.  Regarding to the long time behavior for solutions, a lot of results are under the close-to-equilibrium setting. In 1980s, Matsumura and Nishida \cite{MR555060, MR0564670} first obtained the optimal decay rates of $L^2$-norm 
\begin{equation}\label{MN:decay}
\| (\rho-\vr, u, \theta-1)(t)\|_{L^{2}} \leq C(1+t)^{-\frac{3}{4}}.
\end{equation}
The general $L^{p}$ decay rates had been established by Ponce \cite{MR0785713}, in 1985,
\begin{equation}
    \|\nabla^{k}(\rho-\vr, u, \theta-1)(t)\|_{L^{p}} \leq C(1+t)^{-\frac{d}{2}\left(1-\frac{1}{p}\right)-\frac{k}{2}},
\end{equation}
where $p\geq 2$, $0\leq k\leq 2$ and $d = 2, 3$ is the dimension of space. 
% The optimal time decay rates of the solutions with negative Sobolev norms are given by Guo and Wang \cite{Guo2012Decay} in 2012. In 2017, Danchin and Xu \cite{Danchin2017Opi} generalized the estimates in \cite{MR555060} to the general $L^p$ critical spaces by making full use of the refined time weighted inequalities in the Besov space. 
% When the external potential force is taken into consideration, we refer to \cite{Deck1993Decay, Duan2007Optimal, Duan2007Opt, Shibata2003On,Shibata2007Rate,Ukai2006Con} for details.
In 2009, Villani \cite{MR2562709} studied the asymptotic behavior of global solution which belongs to $L^\infty(0,\infty;C^k(\R^3))$, $\forall k\in \N$.  Under the assumption that the density and the temperature have uniform lower bounds, he proved that 
\begin{equation}\label{v}
\|(\rho-1,u,\theta-1)(t)\|_{C^{k}}=O(t^{-\infty}),  \quad \text{for} ~ k\in\mathbb{N}.
\end{equation}
When the initial data is far away from the equilibrium, there are fewer results on the large-time behavior. For the isentropic Navier-Stokes equations, in 2019, He, Huang and Wang \cite{He2019Glo} established the global stability of large smooth solution in $\mathbb{R}^3$ under the condition that the density is uniformly bounded in $C^\alpha$ for some $\alpha\in(0,1)$ in the whole space, more precisely, they obtained the decay rate 
\begin{equation}\label{1.23}
\|\rho(t)-1\|_{H^{1}}+\|u(t)\|_{H^{1}} \leq C(1+t)^{-\frac{3}{4}\left(\frac{2}{p}-1\right)},
\end{equation}
where $(\rho_0 -1, u_0 ) \in L^{p} \cap H^2$ with $p\in[1,2]$. In the next year,  Gao, Wei and Yao \cite{Gao2020The} refined the results in \cite{He2019Glo}, and considered the following second-order derivative decay rate
\begin{equation}
\|\D\rho(t)\|_{H^{1}}+\|\D u(t)\|_{H^{1}}+\|\rho_t(t)\|_{L^2}+\|u_t(t)\|_{L^2} \leq C(1+t)^{-\frac{3}{4}\left(\frac{2}{p}-1\right)-\frac{1}{2}}.
\end{equation}
%and gave a rigorous proof of optimal $L^2$ time decay rate. 
% Under the assumption that the initial perturbation is bounded in some negative Besov norm, in 2021, Xin and Xu \cite{Xin2021Opi} proved the optimal decay rates of solutions by employing the Lyapunov energy argument. 
For the nonisentropic case, in 2020, Zhang and Zi \cite{Zhang2020Con} studied the convergence to equilibrium on the torus under the conditions that both the density and the temperature possess uniform in time positive lower and upper bounds, and obtained the exponential stability of the $C^\infty$ solution
\begin{equation}\label{ZZ:decay}
\| (\rho-\vr, u, \theta-1)(t)\|_{C^{k}} \leq Ce^{-\frac{Ct}{2k+4}}, \quad \text{for} ~ k\in\mathbb{N},
\end{equation}
which improved the previous result (\ref{v}) in \cite{MR2562709}. Recently, He, Huang and Wang \cite{MR4389852} proved the global-in-time stability of large solution
\begin{equation}\label{HHW:decay}
\| (\rho-\vr, u, \theta-1)(t)\|_{H^{1}} \leq C(1+t)^{-\frac{3}{4}}
\end{equation}
under the assumptions that $(\rho_0-\vr, u_0, \theta_0-1)\in L^1(\R^3)$, the density is uniformly bounded in $C^\alpha$ for some $\alpha\in(0,1)$, and the temperature is uniformly bounded in $\R^3\times\R_+$. Gao, Wei and Yao \cite{Gao2021De} extended the results in \cite{MR4389852} to the second-order spatial derivative of the solution
\begin{equation}\label{GWY:decay}
\|\D(\rho-\vr, u, \theta-1)(t)\|_{H^{1}} \leq C(1+t)^{-\frac{5}{4}},
\end{equation}
which was improved by Luo and Zhang \cite{MR4493879} to the rate $(1+t)^{-\frac{7}{4}}$. Moreover, the authors in \cite{Gao2021De} obtained 
\begin{equation}\label{GWY:decay:2}
\|(\rho-\vr, u, \theta-1)(t)\|_{H^{1}} \leq C(1+t)^{-\frac{s}{2}}, \quad \|\D(\rho-\vr, u, \theta-1)(t)\|_{H^{1}} \leq C(1+t)^{-\frac{1+s}{2}},
\end{equation}
when the initial perturbation belongs to the negative Sobolev space $\dot{H}^{-s}$ for $s\in(0,\frac{3}{2})$.

\bigskip

In the following, we give our second result concerning the algebraic decay rate of the global solution obtained in Theorem \ref{th}.
\begin{theorem}\label{th2} Let $(\rho, u, \theta)$ be the global strong solution of (\ref{FCNS})  obtained in Theorem \ref{th}. Moreover, if $(\rho_0-\vr, u_0, \theta_0-1)\in L^{p_0}(\R^3)$ with $p_0\in[1,2]$, then we have
 \begin{equation}\label{decay-0}
\|\rho-\vr\|_{H^1}+\|u\|_{H^1} + \|\theta-1\|_{H^1}+\|\dot{u}\|_{L^2}+\|\dot{\theta}\|_{L^2}\leq \bar{C}(1+t)^{-\frac 34\left(\frac 2{p_0}-1\right)},
\end{equation}
where $\dot{u}$ and $\dot{\theta}$ denote the material derivative of $u$ and $\theta$, and the constant $\bar{C}$ depends only on $\vr$, $\mu$, $\lambda$, $\kappa$, $R$ and initial data.
\end{theorem}

\begin{remark}
In previous works \cite{MR555060, MR0564670}, the authors established the global smooth solution under the assumption that the initial perturbation is sufficiently small, moreover, they obtained the decay rate \eqref{MN:decay}. When the initial perturbation is large, there are rich literature concerning the large time behaviour of the solution, see \cite{Gao2020The,Gao2021De,He2019Glo,MR4389852,MR4493879,MR2562709,Zhang2020Con}.  However, as far as we know, the global existence of the large strong or smooth solution in Sobolev space are not given in the above references. In this paper, we establish the global large strong solution in Theorem \ref{th}, based on which, we show the asymptotic behaviour of the solution in Theorem \ref{th2}. It is worth mentioning that, the regularity assumptions
\begin{equation}\label{condition:HHW}
\sup\limits_{t\geq 0}\|\rho-\vr\|_{C^\al}\leq M, \quad \sup\limits_{t\geq 0}\|\theta-1\|_{L^\infty}\leq M,
\end{equation}
in \cite{Gao2020The,Gao2021De,He2019Glo,MR4389852,MR4493879} are satisfied in Theorem \ref{th}. This is the first result dedicated to the asymptotic behavior of the large strong solution to 3D full compressible Navier-Stokes equations with variable viscosities in the whole space. %Compared to the previous works \cite{MR0564670, MR0785713}, etc., which are restricted to the small perturbation framework, we consider the case of large initial data. Motivated by \cite{Sch1985De,He2019Glo}, where the case of incompressible and compressible isentropic Navier-Stokes equations are considered respectively, we study the full compressible Navier-Stokes equations with density-dependent viscosities in this paper, which possess stronger nonlinear interaction in the momentum equation and the temperature equation, bringing great difficulties to resolve.
\end{remark}

\begin{remark}
Combining \eqref{1010} and the proof of Theorem \ref{th2}, we can obtain the decay rate of the second-order spatial derivative of density. Recently, we notice that in \cite{Gao2021De,MR4493879} the authors obtained the optimal decay rates of the higher-order spatial derivative of solution when the viscosities are constants. However, for the variable viscosities case, whether the optimal decay rates still hold remain unsolved.
% , and the density and temperature satisfy the following assumptions,
% \begin{equation}
% \sup\limits_{t\geq 0}\|\rho-\vr\|_{C^\al}\leq M, \quad \sup\limits_{t\geq 0}\|\theta-1\|_{L^\infty}\leq M,
% \end{equation}
% with $\al\in(0,1)$. However, the decay rates of the second-order derivatives of the solution are not given. In Theorem \ref{th2}, we give the decay rates of the second-order derivatives of the velocity and the temperature. Moreover, when the initial belong to $L^{p_0}$ for $p_0\in[1,2]$, we give rigorous proof the decay rates of the solution. In \cite{Gao2021De}, the authors established the optimal decay rates of the strong solution under the assumptions \eqref{condition:HHW}.
\end{remark}

%, and some new difficulties arise

We now make some comments on the analysis of this paper. The local
existence and uniqueness of strong solutions to \eqref{FCNS:2}--\eqref{data and far} can be found in \cite{MR0283426, MR0149094, MR0106646, ta, MR3210747}. However, 
% after the celebrated work \cite{vk} given by Vaigant and Kazhikhov, in 1995, for global existence of large strong solutions to 2D isentropic Naiver-Stokes equations with density-dependent viscosity,
there are few results concerning the global existence of large strong or classical solutions to nonisentropic Naiver-Stokes equations in Sobolev space. In this paper, we consider the global existence of large strong 
solutions to the Cauchy problem \eqref{FCNS:2}--\eqref{data and far} when initial data satisfy \eqref{ic-1}–-\eqref{compatibility condition}. We hope to use the bootstrap argument to extend the local strong solution with strictly positive initial density and temperature globally in time just under the condition that the initial density is
linearly equivalent to a large constant state. Under this purpose, we need to establish the global a priori estimates in the whole space for the power law of density-dependent viscosities case. It turns out that the key difficulties in this paper are to derive both the time-independent lower-order and higher-order
estimates of the strong solution $(\rho, u, \theta)$, and to find the range of the parameter $\alpha$ to close our a priori estimates. One may think that the larger $\alpha$, the better dissipation property, and then global well-posedness problem can be solved provided $\alpha$ is large. Nevertheless, it is a nontrivial problem. In fact, using  the
bootstrap argument to close the estimate  and finding a suitable interval of the parameter $\alpha$ to control the terms $A_i(T) ~ (i=1,\cdots,8)$, $A_{5,j}(T) ~(j=1,2,3)$ (see Proposition \ref{p4.1}) at the same time is not obvious.  More precisely,
\begin{itemize}
    \item Global solution of large strong solution.
    
    In the lower-order estimates, we give an entropy-type inequality when viscosities depend on density, and bounds on the $L^2$-norm of the solution are obtained. Here, the bounds are uniform-in-time  and depend on the far field value of density $\vr$, which will play a crucial role in the analysis of this paper. In the case of density-dependent viscosities in the higher-dimensional space, the first main difficulty lies in the estimates on the first-order spatial derivatives
of both the velocity and the temperature. To overcome this difficulty, different 
from the constant viscosity case \cite{MR3744381}, we introduce the modified effective viscous flux and give some assumptions on the first-order spatial derivatives
of density (see $A_7(T)$, $A_8(T)$), by which, after some 
careful analysis on the standard $L^p$ elliptic estimate, we succeed in deriving some new uniform estimates on the key terms of $\|\D G\|_{L^p}$ and $\|\D u\|_{L^p}$, $p\in[2,6]$.
The second main difficulty comes from the bound
of $\|\sqrt{\rho}\dot u\|_{L^2}$. When viscosities depend on density and $\vr$ is large, by the standard energy method, we can see that $\|\sqrt{\rho}\dot u\|_{L^2}$ is a large term due to large initial data. In fact, the upper bound
of $\|\sqrt{\rho}\dot u\|_{L^2}$ plays an important role in the higher-order estimates of the velocity, which further will be used in the time-independent estimates of the first-order spatial derivatives of density. When $\|\sqrt{\rho}\dot u\|_{L^2}$ is too large, we cannot use bootstrap argument to close our a priori assumptions directly. For this reason, we divide our estimates with respect to time into two parts $(0,\sigma(T))$ and $(\sigma(T),T)$, and some new difficulties come out. In the time interval $(0,\sigma(T))$, estimates of $A_i(\sigma(T))(i=1,\cdots,4)$, with elaborate analysis on the power of $\vr$, can be obtained provided that the $\vr$ is suitably large, which imply the smallness of $\nabla u$ in $L^2(0,\sigma(T);L^2)$. In order to derive the the smallness information of $\nabla u$ in $L^2(\sigma(T),T;L^2)$,  different from the isentropic case, we need to prove the uniform upper bound of the temperature as we can see from the entropy-type inequality. Next, in the time interval $(\sigma(T), T)$,  we firstly assume that the upper bound of temperature can be controlled by a constant which depends on $\vr$, by which we have $\nabla u$ is small in $L^2(0,T;L^2)$ when $\vr$ is suitably large, and it plays a fundamental role  throughout this paper. Then the time-weighted higher-order estimates of $A_5(T)$ and $A_6(T)$ can be proved by using a linear combination of $A_{5,j}(T)(j=1,\cdots,3)$, and the estimates on $A_5(T)$ and $A_6(T)$ are better than $A_i(\sigma(T))(i=1,\cdots,4)$ due to the influence of the lack of initial values. With the estimates on the first-order and the second-order spatial derivatives of temperature at hand, we show that, by using interpolation, the temperature is indeed uniform bounded from above. At the end of our a priori estimates, we are in position to estimate the first-order spatial derivatives of density, $A_7(T)$ and $A_8(T)$, which is the third main difficulty in proving the global existence of strong solution. After careful observation, we find that time-weighted estimates on $A_7(T)$ and $A_8(T)$ can be established in time interval $(0,\sigma(T))$ by using Gronwall inequality due to the bound of $\sigma(T)$. It is worth mentioning that due to the interaction between density and velocity, $\|\nabla^2 u\|_{L^4}$ plays an important role in our analysis. When $t\in(\sigma(T),T)$, although the time interval is unbounded, thanks to the time-weighted estimates, the upper bound of $\|\nabla^2 u\|_{L^4}$ is smaller than that of the $(0,\sigma(T))$ case. Then, by using standard energy estimates, together with the modified effective viscous flux, $A_7(T)$ and $A_8(T)$ can be bounded, by which, together with $L^2$-norm estimate of $\rho-\vr$ and interpolation inequality, the lower and upper bounds of the density can be obtained. Based on the analysis mentioned above, we can close the a priori assumptions in Proposition \ref{p4.1}. At last, in Lemma \ref{Lem:10}, different from the previous results \cite{MR555060, MR0564670,MR3744381}, the $\vr$-dependent uniform upper bound of the strong solution is given, and hence, global existence of large strong solution is established. In addition, the uniform-in-time lower and upper bounds of both density and temperature are given in the higher dimension space.

%It should be pointed out that the terms $A_i(\sigma(T))(i=1,\cdots,4)$ may be large and .

%where the authors considered the large time behavior of global solutions under some regularity assumptions
 \item Asymptotic behavior  of large strong solution.
 
Motivated by results in \cite{Gao2020The,Gao2021De,MR4389852,MR4493879,MR2562709,Xin2021Opi,Zhang2020Con} , we  consider the asymptotic behavior  of large strong solution to the full compressible Navier-Stokes equations with density-dependent viscosities and large initial data. The most cited known results are given by Matsumura and Nishida \cite{MR555060,MR0564670} in the small perturbation framework. However, the asymptotic behaviour of global large solution under the case of variable viscosities are open as far as we know. In Theorem \ref{th}, we establish the global-in-time existence of large strong solutions to the Cauchy problem  when the initial density is linearly equivalent to a large constant state. 
Then, in Theorem \ref{th2}, we turn to study asymptotic behavior of large strong solution. Here, we apply the method in \cite{He2019Glo} which concerning the global stability of large solution to isentropic Navier-Stokes equations, to study the nonisentropic case with variable viscosities. First, in Lemma \ref{Lem4.1:6-0}, we derive a dissipation inequality of the global solution $(\rho,u,\theta)$, where the temperature of fluid is considered additionally. Second, in Lemma \ref{Lem4.2}, by using Fourier transform, we give some decay estimates on the low frequency part of the solution provided the initial data belong to $L^{p_0}$ space for $1\leq {p_0}\leq 2$, while \cite{MR4389852,MR4493879} only considered the case that the initial data belong to $L^1$ space. It should be noted that, compared with the constant viscosities case \cite{Gao2020The,Gao2021De,MR4389852,MR4493879,MR2562709,Xin2021Opi,Zhang2020Con}, estimates in frequency space become much more complicated when we consider the nonisentropic Navier-Stokes equations with variable viscosities. 
In this paper, we find some cancellation property about the temperature related nonlinear terms which cannot be controlled by the present initial assumptions. When viscosities depend on density, based on properties of  Fourier transform and the necessary regularity of density established in \S \ref{Section 3}, we succeed in deriving decay estimates on the low frequency part of viscosity related terms, viscosity stress tensor $\T$ for example. At last, in Lemma \ref{decayprop}, inspired by the work \cite{Sch1985De}, by using Fourier splitting method, together with the high frequency estimates and ordinary differential equation with respect to time, after some iterations of time decay rates, we establish the asymptotic behavior of large strong solutions in whole space. As a result,  we extend the works in \cite{Gao2020The,Gao2021De,MR4389852,MR4493879,MR2562709,Xin2021Opi,Zhang2020Con} to the nonisentropic, variable viscosity case. To the best of authors' knowledge, this is the first result of asymptotic behavior of large strong solution to the variable viscosity fluid model. 
\end{itemize}

The rest of this paper is organized as follows.  In \S \ref{prelim}, we give some notations and preliminary lemmas which will be used frequently in our following proof. In \S \ref{Section 3}, we devote to establishing the global existence of large strong solution by using bootstrap argument. More precisely, in \S \ref{sec3.1}, we give the lower-order uniform estimates of solution, and the higher-order time-weighted uniform estimates are given in \S \ref{sec3.2}. In \S \ref{sec3.3}, we give the uniform estimates on the first-order spatial derivatives of density, by which, lower and upper bounds of density are also obtained.
Combining the estimates mentioned above, in \S \ref{sec3.4}, we show the higher-order uniform estimates, then we can close our the priori estimates and extend the local solution to a global one. At the end of \S \ref{Section 3}, we prove that the first-order spatial derivatives of the global large solution trends to zero in $H^1$ Sobolev space as time trends to infinity, by which we obtain the lower bound of temperature. Finally, in \S \ref{sec4}, with additional assumptions on the initial data, we establish the convergence of the solution to its associated equilibrium with an explicit decay rate when the initial data far away from the equilibrium in the whole space.
%.  based on the a priori estimates and continuation argument,

\section{Preliminaries}\label{prelim}
	
\subsection{Notations and Inequalities}
	
For later purpose, we introduce the following notations. For any $r\in[1,\infty]$ and integer $k\geq 0$, we denote
\begin{equation*}
\begin{aligned}
& L^r=L^r(\R^3), \quad W^{k,r}=W^{k,r}(\R^3), \quad H^k=W^{k,2}, \\
& D^{k,r}=\left\{u\in L^1_{loc}(\R^3) : ~\D^ku\in L^r\right\}.
\end{aligned}
\end{equation*}
For matrices $A$ and $B$, we denote $A:B=\sum\limits_{i,j=1}^nA_{ij}B_{ij}$, where $A_{ij}$ is the $(i,j)$-element of $A$. For any function $f(x,t)$, we denote
\begin{equation*}
\int f\dif x=\int_{\R^3}f\dif x,
\end{equation*}
and $\dot{f}=f_t+u\cdot\D f$ is the material derivative of $f$. Moreover, $\|\cdot\|_{L^p}$, $\|\cdot\|_{L^\infty}$ and $\|\cdot\|_{H^s}$ stand for the norms of $L^p$, $L^{\infty}$ and $H^s$, respectively. $\widehat{f}(\xi)$ denotes the Fourier transformation of a function $f(x)$. Throughout this paper, we use $C$ to denote a generic positive constant that may depend on $\mu$, $\lambda$, $\kappa$, $R$, $C_0$, but is independent of $T$, $\vr$, $K_i$, and will change in different places.

% Then we denote
% \begin{equation}\label{notations}
% \begin{cases}
% \dot{f}=f_t+u\cdot\D f, \\
% G=(2\mu+\lambda)\div u-R(1-\al)^{-1}(\rho^{1-\al}\theta-\vr^{1-\al}),
% \end{cases}
% \end{equation}
% which are the material derivative of a function $f$ and the effective viscous flux, respectively. We also denote
% \begin{equation}\label{H}
% H=\rho^{-\al}\left(-\rho\dot{u}+\frac{R\al}{\al-1}\rho\D\theta +2\mu\frD(u)\cdot\D(\rho^\al)+\lambda\div u\D(\rho^\al)\right).
% \end{equation}
	
The following local existence result can be similarly obtained to \cite{MR3210747}.
\begin{proposition}\label{local}
Assume that the initial data $(\rho_0, u_0, \theta_0)$ satisfies
\begin{equation}\label{data:condition}
\frac{3}{4}\vr\leq\rho_0(x)\leq\frac{5}{4}\vr, \quad \rho_0(x)-\vr\in H^1\cap D^{1,4}, \quad (u_0(x), ~ \theta_0(x)-1)\in H^2,
\end{equation}
then there exist $T_*> 0$ and a strong solution $(\rho,u,\theta)$ to \eqref{FCNS:2}--\eqref{data and far} in $\R^3 \times [0, T_*]$.
\end{proposition}
	
The following well-known Gagliardo–Nirenberg inequality will be used (see \cite{Ladyzhenskaya1968}).
\begin{lemma}
For $p\in[2,6]$, $q\in(1,\infty)$ and $r\in(3,\infty)$, there exists a constant $C>0$ which may depend on $q,r$  such that for $f\in H^1$ and $g\in L^q \cap D^{1,r}$, it holds
\begin{equation}
\|f\|_{L^p}\leq C\|f\|_{L^2}^{\frac{6-p}{2p}}\|\D f\|_{L^2}^{\frac{3p-6}{2p}},
\end{equation}
\begin{equation}
\|g\|_{L^\infty}\leq C\|g\|_{L^q}^{\frac{q(r-3)}{3r+q(r-3)}}\|\D g\|_{L^r}^{\frac{3r}{3r+q(r-3)}}.
\end{equation}
\end{lemma}

From $(\ref{FCNS:2})_2$, in this paper, we denote  $G$ satisfies
\begin{equation}\label{elliptic}
-\Delta G=\div H,
\end{equation}
where 
\begin{equation}\label{G}
    G=(2\mu+\lambda)\div u+R(\rho^{1-\al}\theta-\vr^{1-\al}),
\end{equation}
is the modified effective viscous flux and 
\begin{equation}\label{H}
H=\rho^{-\al}\left(-\rho\dot{u}+R\al\rho\D\theta +2\mu\frD(u)\cdot\D(\rho^\al)+\lambda\div u\D(\rho^\al)\right).
\end{equation}

Then we have the following elementary estimates:
\begin{lemma}
Let $(\rho,u,\theta)$ be a strong solution of \eqref{FCNS:2}--\eqref{data and far}, and the assumptions of Theorem \ref{th} hold. Then there exists a generic positive constant $C$ depending only on $\mu$, $\lambda$, $\kappa$, and $R$ such that, for any $p\in[2, 6]$, it holds
\begin{equation}\label{2.3:2}
\|\D G\|_{L^p}\leq C\|H\|_{L^p}\leq C\vr^{-\al}\left(\|\rho\dot{u}\|_{L^p}+\vr\|\D\theta\|_{L^p}+\vr^{\al-1}\|\D u\cdot\D\rho\|_{L^p}\right),
\end{equation}
\begin{equation}\label{2.3:3}
\|G\|_{L^p}\leq C\|H\|_{L^2}^{\frac{3p-6}{2p}} \left(\|\D u\|_{L^2}+\|\rho^{1-\al}\theta-\vr^{1-\al}\|_{L^2}\right)^{\frac{6-p}{2p}},
\end{equation}
\begin{equation}\label{2.3:4}
\|\D u\|_{L^p}\leq C\|\D u\|_{L^2}^{\frac{6-p}{2p}}\left( \|H\|_{L^2}+\|\rho^{1-\al}\theta-\vr^{1-\al}\|_{L^6}\right)^{\frac{3p-6}{2p}}.
\end{equation}
\end{lemma}
\begin{proof}
Applying standard $L^p$-estimate to \eqref{elliptic} and noting the definition of $H$ in \eqref{H}, we obtain \eqref{2.3:2}. By the interpolation inequalities of $L^p$ spaces, we have
\begin{align}\label{2.3:?}
\|G\|_{L^p}\leq \|\D G\|_{L^2}^{\frac{3p-6}{2p}}\|G\|_{L^2}^{\frac{6-p}{2p}},
\end{align}
then \eqref{2.3:3} is followed by \eqref{2.3:2} and \eqref{2.3:?}. Noting that $-\Delta u=-\D\div u+\D\times\D\times u$,
we get
\begin{align*}
-\D u=-\D(-\Delta)^{-1}\D\div u+\D(-\Delta)^{-1}\D\times\D\times u.
\end{align*}
Therefore, standard $L^p$-estimate implies
\begin{align}\label{2.3:??}
\|\D u\|_{L^p}\leq C\left(\|\div u\|_{L^p}+\|\D\times u\|_{L^p}\right).
\end{align}
By the definition of $G$ and \eqref{2.3:??}, we obtain
\begin{equation}\label{2.3:1}
\|\D u\|_{L^p}\leq C\left(\|G\|_{L^p}+\|\D\times u\|_{L^p}+\|\rho^{1-\al} \theta-\vr^{1-\al}\|_{L^p}\right).
\end{equation}
Noting that $\mu\Delta(\D\times u)=\D\times H$, we have
\begin{equation}\label{2.3:???}
\|\D\times u\|_{L^6}\leq C\|\D(\D\times u)\|_{L^2}\leq C\|H\|_{L^2},
\end{equation}
then \eqref{2.3:4} is followed by \eqref{2.3:1}, \eqref{2.3:3}, \eqref{2.3:???}, and the interpolation inequalities of $L^p$ spaces.
\end{proof}	

\section{Global Existence of Large Strong Solution}\label{Section 3}

In this section, we will establish a priori bounds for the local-in-time strong solution to \eqref{FCNS:2}--\eqref{data and far} obtained in Proposition \ref{local}. We thus fix a strong solution
$(\rho,u,\theta)$ of \eqref{FCNS:2}--\eqref{data and far} on $\R^3 \times (0,T]$ for some time $T>0$, with initial data
$(\rho_0,u_0,\theta_0)$ satisfying \eqref{data:condition}.
We denote $A_i(T)$ as follows:
\begin{equation}\label{A1}
A_1(T)=\sup\limits_{t\in[0,T]}\left(\frac{\vr}{2}\right)^\al\|\D u\|_{L^2}^2+\int_0^T\|\sqrt{\rho}\dot{u}\|_{L^2}^2\dif t,
\quad 
A_2(T)=\kappa\sup\limits_{t\in[0,T]}\|\D \theta\|_{L^2}^2+\int_0^T\|\sqrt{\rho}\dot{\theta}\|_{L^2}^2\dif t,
\end{equation}
\begin{equation}\label{A3}
A_3(T)=\sup\limits_{t\in[0,T]}\|\sqrt{\rho}\dot{u}\|_{L^2}^2+\left(\frac{\vr}{2}\right)^\al\int_0^T\|\D \dot{u}\|_{L^2}^2\dif t,
\end{equation}
%\begin{equation}\label{A4}
%A_4(T)=\sup\limits_{t\in[0,T]}\int\rho|\dot{\theta}|^2\dif x+\kappa\int_0^T\int|\D \dot{\theta}|^2\dif x\dif t,
%\end{equation}
\begin{equation}\label{A4}
A_4(T)=\sup\limits_{t\in[0,T]}\|\sqrt{\rho}(\theta-1) \|_{L^2}^2+\int_0^T\left(\left(\frac{\vr}{2}\right)^\al\|\D u\|_{L^2}^2+\kappa\|\D\theta\|_{L^2}^2 \right)\dif t,
\end{equation}
\begin{equation}\label{B1}
A_5(T)=M\vr A_{5,1}(T)+M^2\vr A_{5,2}(T)+A_{5,3}(T),
\quad
A_6(T)=\sup\limits_{t\in[0,T]}\sigma^6\|\sqrt{\rho}\dot{\theta}\|_{L^2}^2+\kappa\int_0^T\sigma^6\|\D\dot{\theta}\|_{L^2}^2\dif t,
\end{equation}
%\begin{equation}\label{A6}
%A_8(T)=\sup_{t\in[0,T]}\|\rho-\vr\|_{L^6}^6+\vr^{1-\al}\int_0^T\|\rho-\vr\|_{L^6}^6\dif t,
%\end{equation}
%\begin{equation}\label{A8}
%A_7(T)=\sup_{t\in[0,T]}\|\D \rho\|_{L^3}^2+\vr^{1-\al}\int_0^T\|\D \rho\|_{L^3}^2\dif t,
%\end{equation}
\begin{equation}\label{A7}
A_7(T)=\sup_{t\in[0,T]}\|\D \rho\|_{L^2}^2+\vr^{1-\al}\int_0^T\|\D \rho\|_{L^2}^2\dif t,
\quad
A_8(T)=\sup_{t\in[0,T]}\|\D \rho\|_{L^4}^2+\vr^{1-\al}\int_0^T\|\D \rho\|_{L^4}^2\dif t,
\end{equation}
where $\sigma(t) := \min\{1, t\}$ and $M>1$ is a constant to be determined later.
\begin{equation}\label{B1,1}
A_{5,1}(T)=\sup\limits_{t\in[0,T]}\sigma^2\left(\frac{\vr}{2}\right)^\al\|\D u\|_{L^2}^2+\int_0^T\sigma^2\|\sqrt{\rho}\dot{u}\|_{L^2}^2\dif t,
\end{equation}
\begin{equation}\label{B1,2}
A_{5,2}(T)=\kappa\sup\limits_{t\in[0,T]}\sigma^3\|\D\theta\|_{L^2}^2+\int_0^T\sigma^3\|\sqrt{\rho}\dot{\theta}\|_{L^2}^2\dif t,
\end{equation}
\begin{equation}\label{B1,3}
A_{5,3}(T)=\sup\limits_{t\in[0,T]}\sigma^3\|\sqrt{\rho}\dot{u}\|_{L^2}^2+\left(\frac{\vr}{2}\right)^\al\int_0^T\sigma^3\|\D \dot{u}\|_{L^2}^2\dif t.
\end{equation}

We then have the following key a priori estimates on $(\rho,u,\theta)$.

\begin{proposition}\label{p4.1}
For given $\vr>1$, assume that $(\rho_0,u_0,\theta_0)$ satisfies \eqref{ic-1}--\eqref{compatibility condition}, then there exist positive constants $K_i$ and $\rL$ depending on $\mu$, $\lambda$, $\kappa$, $R$, $C_0$, $u_0$, and $\theta_0$, such that if $(\rho,u,\theta)$ is a strong solution of \eqref{FCNS:2}--\eqref{data and far} on $\R^3\times(0,T]$, for any $T>0$, satisfying
\begin{equation}\label{assume:1}
A_1(\sigma(T))\leq 2K_1\vr^\al, ~ A_2(\sigma(T))\leq 2K_2\vr^\al, ~ A_3(\sigma(T))\leq 2K_3\vr^{2\al-1}, ~ A_4(\sigma(T))\leq 2K_4\vr, 
\end{equation}
\begin{equation}\label{assume:2}
A_5(T)\leq 2K_6\vr^2, \quad A_6(T)\leq 2K_7\vr^2,
\end{equation}
\begin{equation}\label{assume:1.75}
\|\theta-1\|_{L^\infty}^2\leq 2K_8\vr^2, \qquad\text{for} ~ t\in[\sigma(T),T],
\end{equation}
\begin{equation}\label{assume:3}
A_7(T)\leq 2K_9\vr^2, \quad A_8(T)\leq 2\vr^{\frac{3}{2}}, \quad \frac{1}{2}\vr\leq\rho\leq\frac{3}{2}\vr,
\end{equation}
we have
\begin{equation}\label{goal:1}
A_1(\sigma(T))\leq K_1\vr^\al, ~ A_2(\sigma(T))\leq K_2\vr^\al, ~ A_3(\sigma(T))\leq K_3\vr^{2\al-1}, ~ A_4(\sigma(T))\leq K_4\vr,
\end{equation}
\begin{equation}\label{goal:2}
A_5(T)\leq K_6\vr^2, \quad A_6(T)\leq K_7\vr^2,
\end{equation}
\begin{equation}\label{goal:1.75}
\|\theta-1\|_{L^\infty}^2\leq K_8\vr^2, \qquad\text{for} ~ t\in[\sigma(T),T],
\end{equation}
\begin{equation}\label{goal:3}
A_7(T)\leq K_9\vr^2, \quad A_8(T)\leq \vr^{\frac{3}{2}}, \quad \frac{2}{3}\vr\leq\rho\leq\frac{4}{3}\vr,
\end{equation}
provided $\vr\geq\rL$.
\end{proposition}

Next, we give some lemmas to prove Proposition \ref{p4.1}.
\subsection{Lower-order uniform estimates of $(\rho,u,\theta)$ on $[0,T]$}\label{sub:3.1}

In this subsection, we use $C(K)$ to denote a constant that may depend on $K_i$. The constant $C_i$ ($i\in\mathbb{N}_+$) denotes the coefficient of leading terms with respect to $\vr$. We begin with the following standard energy estimate of $(\rho,u,\theta)$.

\begin{lemma}\label{Lem3.1}
Under the conditions of Proposition \ref{p4.1}, there exists a positive constant $C$ depending on $\mu$, $\lambda$, $\kappa$, $C_0$, and $R$ such that if $(\rho,u,\theta)$ is a strong solution of \eqref{FCNS:2}--\eqref{data and far} on $\R^3\times(0,T]$, the following estimates hold for all $t\in(0,T]:$
\begin{equation}\label{lem3.1:g1}
\|u\|_{L^2}\leq C,
\end{equation}
\begin{equation}\label{lem3.1:g2}
\|\rho-\vr\|_{L^2}\leq C\vr,
\end{equation}
\begin{equation}\label{lem3.1:g3}
\|\theta-1\|_{L^2}\leq C\left(\|\D\theta\|_{L^2}+1\right).
\end{equation}
\end{lemma}
\begin{proof}
Adding $(\ref{FCNS:2})_2$ multiplied by $u$ to $(\ref{FCNS:2})_3$ multiplied by $1-\theta^{-1}$, we obtain after integrating the resulting equality over $\R^3$ and using $(\ref{FCNS:2})_1$ that
\begin{align}\label{Lem3.1:1}
&\frac{\dif}{\dif t}\int\left(\frac{1}{2}\rho|u|^2+R(\rho\ln\rho-\rho-\rho\ln\vr+\vr)+\rho(\theta-\ln\theta-1)\right)\dif x\nonumber\\
=~&-\int\left(\lambda\rho^\al(\div u)^2+2\mu\rho^\al |\frD(u)|^2\right)\dif x-\int\frac{\kappa}{\theta^2} |\D\theta|^2\dif x+\int\left(1-\frac{1}{\theta}\right)\left(\lambda\rho^\al(\div u)^2+2\mu\rho^\al|\frD(u)|^2\right)\dif x\nonumber\\
=~&-\int\left(\frac{1}{\theta}\left(\lambda\rho^\al(\div u)^2+2\mu\rho^\al|\frD(u)|^2\right)+\frac{\kappa}{\theta^2} |\D\theta|^2\right)\dif x.
\end{align}

It follows from the identity
\begin{align*}
\rho\ln\rho-\rho-\rho\ln\vr+\vr=\vr\left(\frac{\rho}{\vr}-1\right)^2\int_0^1\frac{1-y}{y\left(\frac{\rho}{\vr}-1\right)+1}\dif y
\end{align*}
that
\begin{align}\label{Lem3.1:2}
\frac{1}{3\vr}(\rho-\vr)^2\leq\rho\ln\rho-\rho-\rho\ln\vr+\vr\leq\frac{2}{\vr}(\rho-\vr)^2.
\end{align}

Similarly, it follows from the identity
\begin{align*}
\theta-\ln\theta-1=(\theta-1)^2\int_0^1\frac{y}{y(\theta-1)+1}\dif y
\end{align*}
that
\begin{align}\label{Lem3.1:3}
\frac{1}{8}(\theta-1)\chi_{\{\theta(t)>2\}}+\frac{1}{12}(\theta-1)^2\chi_{\{\theta(t)<3\}}\leq\theta-\ln\theta-1\leq(1+\theta^{-1})(\theta-1)^2,
\end{align}
where $\chi_A$ is the characteristic function on the set $A$, and
\begin{align*}
\{\theta(t)>2\}:=&\left\{x\in\R^3|\theta(x,t)>2\right\},\\
\{\theta(t)<3\}:=&\left\{x\in\R^3|\theta(x,t)<3\right\}.
\end{align*}

Integrating \eqref{Lem3.1:1} with respect to t over $(0, T)$ yields
\begin{align}\label{Lem3.1:3-1}
&\sup\limits_{t\in[0,T]}\int\left(\frac{1}{2}\rho|u|^2+R (\rho\ln\rho-\rho-\rho\ln\vr+\vr)+\rho(\theta-\ln\theta-1)\right)\dif x\nonumber\\
&+\int_0^T\int\left(\frac{1}{\theta}\left(\lambda\rho^\al(\div u)^2+2\mu\rho^\al|\frD(u)|^2\right)+\frac{\kappa}{\theta^2} |\D\theta|^2\right)\dif x\dif t\nonumber\\
\leq~&\int\left(\frac{1}{2}\rho_0|u_0|^2+R (\rho_0\ln\rho_0-\rho_0-\rho_0\ln\vr+\vr)+\rho_0(\theta_0-\ln\theta_0-1)\right)\dif x,
\end{align}
which together with \eqref{Lem3.1:2} and \eqref{Lem3.1:3} leads to %\footnote{We need to verify $\theta>0$.}
\begin{align}\label{Lem3.1:3-2}
&\sup\limits_{t\in[0,T]}\int\left(\rho|u|^2+\vr^{-1}(\rho-\vr)^2\right)\dif x+\sup\limits_{t\in[0,T]}\int\left(\rho(\theta-1)\chi_{\{\theta(t)>2\}}+\rho(\theta-1)^2\chi_{\{\theta(t)<3\}}\right)\dif x\nonumber\\
&+\int_0^T\int\left(\frac{1}{\theta}\left(\lambda\rho^\al(\div u)^2+2\mu\rho^\al|\frD(u)|^2\right)+\frac{\kappa}{\theta^2} |\D\theta|^2\right)\dif x\dif t\nonumber\\
\leq~&\vr\|u_0\|_{L^2}^2+2R\vr^{-1}\|\rho_0-\vr\|_{L^2}^2+\left(1+\underline{\theta}\right)\vr\|\theta_0-1\|_{L^2}^2\nonumber\\
\leq ~& C\vr.
\end{align}
This directly gives \eqref{lem3.1:g1} and \eqref{lem3.1:g2}.

Next, we shall prove \eqref{lem3.1:g3}. Direct computation shows that
\begin{equation}\label{Lem3.1:4}
\|\theta-1\|_{L^2(\theta(t)<3)}\leq C, \quad \|\theta-1\|_{L^1(\theta(t)>2)}\leq C,
\end{equation}
which together with H\"{o}lder’s inequality and Sobolev embedding theorem leads to
\begin{align}\label{Lem3.1:5}
\|\theta-1\|_{L^2(\theta>2)}\leq\|\theta-1\|_{L^1(\theta>2) }^{\frac{2}{5}}\|\theta-1\|_{L^6}^{\frac{3}{5}}\leq C\|\D\theta\|_{L^2}^{\frac{3}{5}}\leq C\left(\|\D\theta\|_{L^2}+1\right).
\end{align}
Combining \eqref{Lem3.1:4} and \eqref{Lem3.1:5} yields \eqref{lem3.1:g3} directly. The proof of Lemma \ref{Lem3.1} is finished.
\end{proof}

\subsubsection{Uniform estimates of velocity and temperature on $[0,\sigma(T)]$}\label{sec3.1}

In the following Lemma \ref{Lem33}--\ref{Lem35}, we give some uniform estimates on $A_i(\sigma(T)), i=1,\cdots,4$. First, the first-order estimates of velocity and temperature are given in Lemma \ref{Lem33}.
\begin{lemma}\label{Lem33}
Under the conditions of Proposition \ref{p4.1}, there exist positive constants $K_1$, $K_2$ and $\rL_1$ depending on $\mu$, $\lambda$, $\kappa$, $R$, $C_0$, $u_0$, and $\theta_0$, such that if $(\rho,u,\theta)$ is a strong solution of \eqref{FCNS:2}--\eqref{data and far} on $\R^3\times(0,T]$, the following estimate holds
\begin{equation}\label{g:3.3:1}
A_1(\sigma(T))=\sup\limits_{t\in[0,\sigma(T)]}\left(\frac{\vr}{2}\right)^\al\|\D u\|_{L^2}^2+\int_0^{\sigma(T)}\int \rho|\dot{u}|^2\dif x\dif t\leq K_1\vr^\al,
\end{equation}
and
\begin{equation}\label{g:3.3:2}
A_2(\sigma(T))=\kappa\sup\limits_{t\in[0,\sigma(T)]}\|\D \theta\|_{L^2}^2+\int_0^{\sigma(T)}\int \rho|\dot{\theta}|^2\dif x\dif t\leq K_2\vr^\al,
\end{equation}
provided that $\vr\geq \rL_1$.
\end{lemma}
\begin{proof}
First, multiplying $(\ref{FCNS:2})_2$ by $2u_t$ and integrating the resulting equality over $\R^3$, we obtain after integration by parts that
\begin{align}\label{3.3:1}
&\frac{\dif}{\dif t}\int\left(2\mu\rho^\al|\frD(u)|^2+\lambda \rho^\al(\div u)^2 \right)\dif x+\int\rho|u_t|^2\dif x\nonumber\\
\leq~&-2\int\D P\cdot u_t\dif x+\int\rho|u\cdot\D u|^2\dif x-\al\int\rho^{\al-1}\div(\rho u)\left(2\mu|\frD(u)|^2+\lambda (\div u)^2\right)\dif x\nonumber\\
=~&2R\frac{\dif}{\dif t}\int(\rho\theta-\vr)\div u\dif x-2\int P_t\div u\dif x+\int\rho|u\cdot\D u|^2\dif x\nonumber\\
&-\al\int\rho^{\al-1}\div(\rho u)\left(2\mu|\frD(u)|^2+\lambda (\div u)^2\right)\dif x\nonumber\\
=~&2R\frac{\dif}{\dif t}\int(\rho\theta-\vr)\div u\dif x-\frac{2}{2\mu+\lambda}\int P_tG\dif x+\frac{2R}{(2\mu+\lambda)(\al-1)}\int P_t(\rho^{1-\al}\theta-\vr^{1-\al})\dif x\nonumber\\
&+\int\rho|u\cdot\D u|^2\dif x-\al\int\rho^{\al-1}\div(\rho u)\left(2\mu|\frD(u)|^2+\lambda (\div u)^2\right)\dif x\nonumber\\
:=~&2R\frac{\dif}{\dif t}\int(\rho\theta-\vr)\div u\dif x+\sum\limits_{i=1}^4I_i.
\end{align}

Before we give estimates on the right-hand side of (\ref{3.3:1}), first we show some estimates about $\|\D u\|_{L^6}$, $\|G\|_{L^2}$ and $\|\D G\|_{L^2}$, which will be used frequently in the following proof. It follows from \eqref{2.3:2},  \eqref{2.3:4}, and \eqref{assume:3} that
\begin{align}
\|\D u\|_{L^6}\leq~& C\left(\|H\|_{L^2}+\|\rho^{1-\al}\theta-\vr^{1-\al}\|_{L^6}\right)\nonumber\\
\leq~& C\left(\vr^{\frac{1}{2}-\al}\|\sqrt{\rho}\dot{u}\|_{L^2}+\vr^{1-\al}\|\D\theta\|_{L^2}+\vr^{-1}\|\D\rho\|_{L^4}\|\D u\|_{L^4}+\|\rho^{1-\al}\theta-\vr^{1-\al}\|_{L^6}\right)\nonumber\\
\leq~& C\left(\vr^{\frac{1}{2}-\al}\|\sqrt{\rho}\dot{u}\|_{L^2}+\vr^{1-\al}\|\D\theta\|_{L^2}+\vr^{1-\al}\|\theta-1\|_{L^6}+\vr^{-\al}\|\rho-\vr\|_{L^6}\right)\nonumber\\
&+C\vr^{-1}\|\D\rho\|_{L^4}\|\D u\|_{L^2}^{\frac{1}{4}}\|\D u\|_{L^6}^{\frac{3}{4}}\nonumber\\
\leq~& C\left(\vr^{\frac{1}{2}-\al}\|\sqrt{\rho}\dot{u}\|_{L^2}+\vr^{1-\al}\|\D\theta\|_{L^2}+\vr^{-\al}\|\nabla\rho\|_{L^2}+\vr^{-4}\|\D\rho\|_{L^4}^4\|\D u\|_{L^2}\right)+\frac{1}{2}\|\D u\|_{L^6}.
\end{align}
Then, we get
\begin{align}\label{estimate:DuL6}
\|\D u\|_{L^6}\leq~& C\left(\vr^{\frac{1}{2}-\al}\|\sqrt{\rho}\dot{u}\|_{L^2}+\vr^{1-\al}\|\D\theta\|_{L^2}+\vr^{-\al}\|\nabla\rho\|_{L^2}+\vr^{-4}\|\D\rho\|_{L^4}^4\|\D u\|_{L^2}\right).
\end{align}
By \eqref{assume:1} and \eqref{assume:3}, we have
\begin{align}\label{estimate:DuL6:2}
\|\D u\|_{L^6}^2\leq CK_3, \quad \text{for} ~ t\in[0,\sigma(T)],
\end{align}
and
\begin{align}\label{estimate:DuL6:3}
\int_0^{\sigma(T)}\|\D u\|_{L^6}^2\dif t\leq~&C\vr^{1-2\al}\int_0^{\sigma(T)}\|\sqrt{\rho}\dot{u}\|_{L^2}^2\dif t+C\vr^{2-2\al}\int_0^{\sigma(T)}\|\D\theta\|_{L^2}^2\dif t\nonumber\\
&+C\vr^{-2\al}\int_0^{\sigma(T)}\|\D\rho\|_{L^2}^2\dif t+C\vr^{-2}\int_0^{\sigma(T)}\|\D u\|_{L^2}^2\dif t\nonumber\\
\leq~&C(K_1+K_9)\vr^{1-\al},
\end{align}
provided that $\alpha\geq 2$. 

Next, from \eqref{2.3:2}, \eqref{2.3:3}, \eqref{lem3.1:g2} and \eqref{lem3.1:g3}, we get
\begin{align}\label{GL2}
\|G\|_{L^2}\leq~&C\left(\|\D u\|_{L^2}+\|\rho^{1-\al}\theta-\vr^{1-\al}\|_{L^2}\right)\nonumber\\
\leq~&C\left(\|\D u\|_{L^2}+\vr^{1-\al}\|\theta-1\|_{L^2}+\vr^{-\al}\|\rho-\vr\|_{L^2}\right)\nonumber\\
\leq~&C\left(\|\D u\|_{L^2}+\vr^{1-\al}+\vr^{1-\al}\|\D\theta\|_{L^2}\right),
\end{align}
and
\begin{align}\label{DGL2}
\|\D G\|_{L^2}^2\leq~&C\vr^{-2\al}\left(\vr\|\sqrt{\rho}\dot{u}\|_{L^2}^2+\vr^2\|\D\theta\|_{L^2}^2+\vr^{2\al-2}\|\D\rho\|_{L^4}^2\|\D u\|_{L^4}^2\right)\nonumber\\
\leq~&C\vr^{-2\al}\left(\vr\|\sqrt{\rho}\dot{u}\|_{L^2}^2+\vr^2\|\D\theta\|_{L^2}^2+\vr^{2\al-2}\|\D\rho\|_{L^4}^2\|\D u\|_{L^2}^{\frac{1}{2}}\|\D u\|_{L^6}^{\frac{3}{2}}\right).
\end{align}
Hence, we have
\begin{align}\label{DGL2:2}
\|\D G\|_{L^2}^2\leq CK_3, \quad \text{for} ~ t\in[0,\sigma(T)],
\end{align}
and
\begin{align}\label{DGL2:3}
\int_0^{\sigma(T)}\|\D G\|_{L^2}^2\dif t\leq CK_1\vr^{1-\al}.
\end{align}

Now, we estimate the terms on the right hand side of \eqref{3.3:1}. 
\begin{align}\label{3.3:4}
I_1=~&-\frac{2}{2\mu+\lambda}\int P_tG\dif x=\frac{2}{2\mu+\lambda}\int \left(\div(Pu)-R\rho\dot{\theta}\right)G\dif x\nonumber\\
\leq~& C\int P|u||\D G|\dif x+C\int\rho|\dot \theta||G|\dif x\nonumber\\
\leq~& C\vr\|\theta-1\|_{L^6}\|u\|_{L^3}\|\D G\|_{L^2}+C\vr\|u\|_{L^2}\|\D G\|_{L^2}+C\vr^\frac 12\|\sqrt{\rho}\dot\theta\|_{L^2}\| G\|_{L^2}\nonumber\\
\leq~&C\vr\|\D \theta\|_{L^2}\|u\|_{L^2}^\frac{1}{2}\|\nabla u\|_{L^2}^\frac{1}{2}\|\D G\|_{L^2}+C\vr\|u\|_{L^2}\|\D G\|_{L^2}+C\vr^\frac 12\|\sqrt{\rho}\dot\theta\|_{L^2}\| G\|_{L^2}\nonumber\\
\leq~&C(K)\left(\vr^{1+\frac{\alpha}{2}}+\vr^\frac 12\|\sqrt{\rho}\dot\theta\|_{L^2}\right),
\end{align}
where we have used integration by parts, \eqref{2.3:3}, \eqref{assume:1}, \eqref{lem3.1:g1}, \eqref{GL2}, Lemma \ref{Lem3.1}, and
\begin{equation}\label{eqn:P}
P_t=-\div(Pu)+R\rho\dot{\theta}.
\end{equation}
due to \eqref{FCNS:2}.
\begin{align}\label{3.3:7}
I_2=~&\frac{2R}{(2\mu+\lambda)(\al-1)}\int P_t(\rho^{1-\al}\theta-\vr^{1-\al})\dif x\nonumber\\
\leq~&  C\int|\D\rho\cdot u|\left|\rho^{1-\al}\theta-\vr^{1-\al}\right|\dif x+C\int\rho|\D u|\left|\rho^{1-\al}\theta-\vr^{1-\al}\right|\dif x+C\int\rho|\dot\theta|\left|\rho^{1-\al}\theta-\vr^{1-\al}\right|\dif x\nonumber\\
\leq~&  C\|\D\rho\|_{L^4}\|u\|_{L^4}\left\|\rho^{1-\al}\theta-\vr^{1-\al}\right\|_{L^2}+C\vr\|\D u\|_{L^2}\left\|\rho^{1-\al}\theta-\vr^{1-\al}\right\|_{L^2}+C\vr^{\frac{1}{2}}\|\sqrt{\rho}\dot\theta\|_{L^2}\left\|\rho^{1-\al}\theta-\vr^{1-\al}\right\|_{L^2}\nonumber\\
\leq~&  C(K)\left(\vr^\frac{3}{4}\left\|\rho^{1-\al}\theta-\vr^{1-\al}\right\|_{L^2}+\vr\left\|\rho^{1-\al}\theta-\vr^{1-\al}\right\|_{L^2}+\vr^{\frac{1}{2}}\|\sqrt{\rho}\dot\theta\|_{L^2}\left\|\rho^{1-\al}\theta-\vr^{1-\al}\right\|_{L^2}\right)\nonumber\\
\leq~&  C(K)\left(\vr^{2-\frac{\alpha}{2}}+\vr^{\frac{3}{2}-\frac{\alpha}{2}}\|\sqrt{\rho}\dot\theta\|_{L^2}\right),
\end{align}
where we have used \eqref{eqn:P}, \eqref{assume:3} Lemma \ref{Lem3.1} and 
\begin{align}
\left\|\rho^{1-\al}\theta-\vr^{1-\al}\right\|_{L^2} 
\leq~&
\left\|\rho^{1-\al}\left(\theta-1\right)\right\|_{L^2}+\left\|\rho^{1-\al}-\vr^{1-\al}\right\|_{L^2}\nonumber\\
\leq~&\vr^{1-\al}\left\|\theta-1\right\|_{L^2}+C\vr^{-\al}\left\|\rho-\vr\right\|_{L^2}\nonumber\\
\leq~& C\vr^{1-\al}\left(1+\left\|\D\theta\right\|_{L^2}\right).
\end{align}
\begin{align}\label{3.3:5}
I_3=\int\rho|u\cdot\D u|^2\dif x\leq C\vr\|u\|_{L^6}^2\|\D u\|_{L^2}\|\D u\|_{L^6}\leq C\vr\|\D u\|_{L^2}^3\|\D u\|_{L^6}\leq C(K)\vr.
%\leq~&C\vr^{1-\al}\|\D u\|_{L^2}^3\left(\vr^{\frac{1}{2}}\|\sqrt{\rho}\dot{u}\|_{L^2}+\vr\|\D\theta\|_{L^2}+\vr\right)\nonumber\\
%\leq~&\frac{1}{5}\|\sqrt{\rho}\dot{u}\|_{L^2}^2+C(K)\vr^{3-2\al}+C(K)\vr^{2-\al},
\end{align}
\begin{align}\label{3.3:6}
I_4=~&-\al\int\rho^{\al-1}\div(\rho u)\left(2\mu|\frD(u)|^2+\lambda (\div u)^2\right)\dif x\nonumber\\
\leq~& C\vr^{\al-1}\int|\D u|^2\left(|u||\D\rho|+\rho|\D u|\right)\dif x\nonumber\\
 \leq~&  C(K)\vr^{\al}\|\D u\|_{L^2}\|\D u\|_{L^6} .
\end{align}

Substituting \eqref{3.3:4}--\eqref{3.3:6} into \eqref{3.3:1}, integrating \eqref{3.3:1} over $(0,\sigma(T))$ and using \eqref{Lem3.1:3-1}, we obtain
\begin{align}\label{3.3:8}
&\left(\frac{\vr}{2}\right)^\al\sup\limits_{t\in[0,\sigma(T)]}\|\D u\|_{L^2}^2+\int_0^{\sigma(T)}\int\rho|\dot{u}|^2\dif x\dif t\nonumber\\
\leq~&C\vr^\al\|\D u_0\|_{L^2}^2+2R\int(\rho\theta-\vr)\div u\dif x-2R\int(\rho_0\theta_0-\vr)\div u_0\dif x\nonumber\\
&+C(K)\vr^\frac 12\int_0^{\sigma(T)}\|\sqrt{\rho}\dot\theta\|_{L^2}\dif t+C(K)\vr^\al\int_0^{\sigma(T)}\left(\|\D u\|_{L^2}^2+\|\D u\|_{L^6}^2\right)\dif t+C(K)\left(\vr+\vr^2+\vr^{\frac{\al}{2}+1}\right)\nonumber\\
\leq~&C_1\vr^\al+C(K)\vr^{\frac{\al}{2}+1}\nonumber\\
\leq~&K_1\vr^\al,
\end{align}
provided that $\al\geq 2$, $\vr\geq\rL_{1,1}$, where we have used \eqref{compatibility condition}, \eqref{assume:1}, \eqref{assume:3}, \eqref{estimate:DuL6:2}--\eqref{DGL2:3}, and
\begin{align*}
\sup\limits_{t\in[0,\sigma(T)]}\|\rho\theta-\vr\|_{L^2}\|\D u\|_{L^2}\leq C\sup\limits_{t\in[0,\sigma(T)]}\left(\vr\|\theta-1\|_{L^2}+\|\rho-\vr\|_{L^2}\right)\|\D u\|_{L^2}\leq C(K)\vr^{\frac{\al}{2}+1},
\end{align*}
and taken $K_1$ and $\rL_{1,1}$ large enough, such that $C_1\vr^\al\leq\frac{1}{2}K_1\vr^\al$, $C(K)\vr^{\frac{\al}{2}+1}\leq\frac{1}{2}K_1\vr^\al$, we complete the proof of \eqref{g:3.3:1}.

Next, we will prove \eqref{g:3.3:2}. Multiplying $(\ref{FCNS:2})_3$ by $\theta_t$ and integrating the resulting equality over $\R^3$, we obtain after integration by parts that
\begin{align}\label{3.3:9}
&\frac{\kappa}{2}\frac{\dif}{\dif t}\int|\D\theta|^2\dif x+\int\rho|\dot{\theta}|^2\dif x\nonumber\\
=~&\int\rho\dot{\theta}u\cdot\D\theta\dif x+\lambda\int\rho^\al\dot{\theta}(\div u)^2\dif x+2\mu\int\rho^\al\dot{\theta}|\frD(u)|^2\dif x-R\int\rho\theta\div u\dot{\theta}\dif x\nonumber\\
&+R\int\rho\theta u\cdot\D\theta\div u\dif x-\lambda\int\rho^\al u\cdot\D\theta(\div u)^2\dif x-2\mu\int\rho^\al u\cdot\D\theta|\frD(u)|^2\dif x\nonumber\\
:=~&\sum\limits_{i=1}^7J_i,
\end{align}
where $J_i$ can be estimated as 
\begin{align}\label{3.3:11}
J_1+J_4\leq~&C\vr^{\frac{1}{2}}\int|\sqrt{\rho}\dot{\theta}|\left(|u\cdot\D\theta|+|\theta\D u|\right)\dif x\nonumber\\
\leq~&\frac{1}{4}\|\sqrt{\rho}\dot{\theta}\|_{L^2}^2+C\vr\left(\|u\D\theta\|_{L^2}^2+\|\theta\D u\|_{L^2}^2\right)\nonumber\\
\leq~&\frac{1}{4}\|\sqrt{\rho}\dot{\theta}\|_{L^2}^2+C\vr\left(\|\D u\|_{L^2}^2+\|\D u\|_{L^2}\|\D\theta\|_{L^2}^2\|\D u\|_{L^6}\right)\nonumber\\
\leq~&\frac{1}{4}\|\sqrt{\rho}\dot{\theta}\|_{L^2}^2+C(K)\vr\left(1+\|\D\theta\|_{L^2}^2\right),
\end{align}
in the last inequality of (\ref{3.3:11}) we have used
\begin{align}\label{3.3:10}
\|u\D\theta\|_{L^2}^2+\|\theta\D u\|_{L^2}^2\leq~&C\left(\|u\|_{L^\infty}^2\|\D\theta\|_{L^2}^2+\|\D u\|_{L^2}^2+\|\D u\|_{L^2}\|\theta-1\|_{L^6}^2\|\D u\|_{L^6}\right)\nonumber\\
\leq~&C\left(\|\D u\|_{L^2}^2+\|\D u\|_{L^2}\|\D\theta\|_{L^2}^2\|\D u\|_{L^6}\right).
\end{align}
\begin{align}\label{3.3:13}
J_2+J_3\leq C\vr^{\al-\frac{1}{2}}\int|\sqrt{\rho}\dot{\theta}||\D u|^2\dif x\leq \frac{1}{4}\|\sqrt{\rho}\dot{\theta}\|_{L^2}^2+C\vr^{2\al-1}\|\D u\|_{L^2}\|\D u\|_{L^6}^3.
\end{align}
\begin{align}\label{3.3:12}
J_5\leq C\vr\left(\|u\D\theta\|_{L^2}^2+\|\theta\D u\|_{L^2}^2\right)\leq C\vr\left(\|\D u\|_{L^2}^2+\|\D u\|_{L^2}\|\D\theta\|_{L^2}^2\|\D u\|_{L^6}\right)\leq C(K)\vr\left(1+\|\D\theta\|_{L^2}^2\right).
\end{align}
%At last, for $t\in[0,\sigma(T)]$ we have
\begin{align}\label{3.3:14}
J_6+J_7\leq~&C\vr^\al\int|u\cdot\D\theta||\D u|^2\dif x\nonumber\\
\leq~&C\vr\|u\D\theta\|_{L^2}^2+C\vr^{2\al-1}\|\D u\|_{L^2}\|\D u\|_{L^6}^3\nonumber\\
\leq~&C\vr\left(\|\D u\|_{L^2}^2+\|\D u\|_{L^2}\|\D\theta\|_{L^2}^2\|\D u\|_{L^6}\right)+C\vr^{2\al-1}\|\D u\|_{L^2}\|\D u\|_{L^6}^3\nonumber\\
\leq~&C\vr\left(1+\|\D\theta\|_{L^2}^2\right)+C\vr^{2\al-1}\|\D u\|_{L^2}\|\D u\|_{L^6}^3.
\end{align}

Substituting \eqref{3.3:11}--\eqref{3.3:14} into \eqref{3.3:9} and integrating \eqref{3.3:9} over $(0,\sigma(T))$, we obtain
\begin{align}\label{3.3:15}
&\kappa\sup\limits_{t\in[0,\sigma(T)]}\|\D\theta\|_{L^2}^2+\int_0^{\sigma(T)}\int\rho|\dot{\theta}|^2\dif x\dif t\nonumber\\
\leq~&\kappa\|\D\theta_0\|_{L^2}^2+C(K)\vr\int_0^{\sigma(T)}\left(1+\|\D\theta\|_{L^2}^2\right)\dif t+C\vr^{2\al-1}\int_0^{\sigma(T)}\|\D u\|_{L^2}\|\D u\|_{L^6}^3\dif t\nonumber\\
\leq~&\kappa\|\D\theta_0\|_{L^2}^2+C(K)\vr\int_0^{\sigma(T)}\left(1+\|\D\theta\|_{L^2}^2\right)\dif t+C\sqrt{K_1K_4}\,\vr^{2\al-1}\int_0^{\sigma(T)}\|\D u\|_{L^6}^2\dif t\nonumber\\
\leq~&C(K)\vr^2+C_2\sqrt{K_1K_4}\,(K_1+K_9)\vr^\al\nonumber\\
\leq~&K_2\vr^\al,
\end{align}
provided $\al\geq 2$, $\vr\geq\rL_{1,2}$, where we have used \eqref{assume:1}, \eqref{estimate:DuL6:2}, \eqref{estimate:DuL6:3}, and taken $K_2$ and $\rL_{1,2}$ large enough, such that $C(K)\vr^2\leq\frac{1}{2}K_2\vr^\al$ and $C_2\sqrt{K_1K_4}\,(K_1+K_9)\vr^\al\leq\frac{1}{2}K_2\vr^\al$. At last, taking $\rL_1=\max\{\rL_{1,1}, ~ \rL_{1,2}\}$, we have proved Lemma \ref{Lem33}.
\end{proof}

The second-order estimates of velocity are given in the following Lemma \ref{Lem34}. We notice that when viscosities depend on density, compared to the constant viscosities case as in \cite{MR3744381}, the estimates become much more complicated.
\begin{lemma}\label{Lem34}
Under the conditions of Proposition \ref{p4.1}, there exist positive constants $K_3$, $K_4$ and $\rL_2$ depending on $\mu$, $\lambda$, $\kappa$, $R$, $C_0$, $u_0$, and $\theta_0$, such that if $(\rho,u,\theta)$ is a strong solution of \eqref{FCNS:2}--\eqref{data and far} on $\R^3\times(0,T]$, the following estimate holds
\begin{equation}\label{g:3.4:1}
A_3(\sigma(T))=\sup\limits_{t\in[0,\sigma(T)]}\int\rho|\dot{u}|^2\dif x+\left(\frac{\vr}{2}\right)^\al\int_0^{\sigma(T)}\int|\D \dot{u}|^2\dif x\dif t\leq K_3\vr^{2\al-1},
\end{equation}
%\begin{equation}\label{g:3.4:2}
%A_4(\sigma(T))=\sup\limits_{t\in[0,\sigma(T)]}\int\rho|\dot{\theta}|^2\dif x+\kappa\int_0^{\sigma(T)}\int|\D \dot{\theta}|^2\dif x\dif t\leq K_4\vr^{3\al-1},
%\end{equation}
provided that $\vr\geq \rL_2$.
\end{lemma}
\begin{proof}
Operating $\dot{u}^j(\pt+\mathrm{div}(u\cdot))$ to $(\ref{FCNS:2})_2^j$, summing with respect to $j$, and integrating the resulting equation over $\R^3$, we obtain
\begin{align}\label{3.4.1}
\frac{1}{2}\frac{\dif}{\dif t}\int\rho|\dot u|^2\dif x
=~& -\int\dot u^j\Big(\partial_jP_t+\div(u\partial_j P)\Big)\dif x+2\mu\int\dot u^j\Big[\pt(\div(\rho^\al\frD(u)))+\div(u\div(\rho^\al\frD(u)))\Big]\dif x\nonumber\\
&+\lambda\int\dot u^j\Big[\pt(\partial_j(\rho^\al\div u))+\div(u\cdot\partial_j(\rho^\al\div u))\Big]\dif x\nonumber\\
:=~&\sum_{i=1}^3M_i. 
\end{align} 

From \eqref{eqn:P}, after integration by parts,  we can see that
\begin{align}\label{3.4.3}
M_1=~&-\int\dot u^j\left(\partial_jP_t+\div (u\partial_j P)\right)\dif x\nonumber\\
=~&R\int\rho\dot{\theta}\div\dot{u}\dif x+\int P\left(\partial_k\left(u^k\div\dot{u}\right)-\div u\div\dot{u}-\partial_j\left(\partial_k\dot{u}^ju^k\right)\right)\dif x\nonumber\\
=~&R\int\rho\dot{\theta}\div\dot{u}\dif x-R\int \rho\theta\partial_k\dot{u}^j\partial_j u^k\dif x\nonumber\\
%\leq~&slant C\eta\int\rho^\gamma |\mathrm{div}\dot{u}|\left |\mathrm{div}u\right |\dif x+C\eta\int\rho^\gamma|\D u||\D \dot u|\dif x\nonumber\\
\leq~&\frac{\mu}{8}\|\D\dot u\|_{L^2}^2+C\left(\|\rho\dot\theta\|_{L^2}^2+\|\rho\theta\D u\|_{L^2}^2\right)\nonumber\\
\leq~&\frac{\mu}{8}\|\D\dot u\|_{L^2}^2+C\left(\vr\|\sqrt{\rho}\dot\theta\|_{L^2}^2+\vr^2\|\D\theta\|_{L^2}^2\|\D u\|_{L^2}\|\D u\|_{L^6}\right).
\end{align}

Next, we have
\begin{align}\label{3.4.4}
M_2=~&\mu\int\dot{u}^j\left(\pt\partial_i\left(\rho^\al(\partial_iu^j+\partial_ju^i)\right)+\partial_k\left(u^k\partial_i\left(\rho^\al(\partial_iu^j+\partial_ju^i)\right)\right)\right)\dif x\nonumber\\
=~&-\mu\int\partial_i\dot{u}^j\pt\left(\rho^\al(\partial_iu^j+\partial_ju^i)\right)\dif x+\mu\int\dot{u}^j\div u\partial_i\left(\rho^\al(\partial_iu^j+\partial_ju^i)\right)\dif x\nonumber\\
&+\mu\int\dot{u}^ju^k\partial_i\partial_k\left(\rho^\al(\partial_iu^j+\partial_ju^i)\right)\dif x\nonumber\\
=~&-\mu\int\partial_i\dot{u}^j\pt\left(\rho^\al(\partial_iu^j+\partial_ju^i)\right)\dif x-\mu\int\partial_i\dot{u}^j\div u\left(\rho^\al(\partial_iu^j+\partial_ju^i)\right)\dif x\nonumber\\
&-\mu\int\rho^\al\dot{u}^j\partial_i(\div u)(\partial_iu^j+\partial_ju^i)\dif x-\mu\int\partial_i\dot{u}^ju^k\partial_k\left(\rho^\al(\partial_iu^j+\partial_ju^i)\right)\dif x\nonumber\\
&-\mu\int\dot{u}^j\partial_i u^k\partial_k\left(\rho^\al(\partial_iu^j+\partial_ju^i)\right)\dif x\nonumber\\
=~&-\mu\int\partial_i\dot{u}^j\pt\left(\rho^\al(\partial_iu^j+\partial_ju^i)\right)\dif x-\mu\int\partial_i\dot{u}^j\div u\left(\rho^\al(\partial_iu^j+\partial_ju^i)\right)\dif x\nonumber\\
&-\mu\int\rho^\al\dot{u}^j\partial_i(\div u)(\partial_iu^j+\partial_ju^i)\dif x-\mu\int\partial_i\dot{u}^ju^k\partial_k\left(\rho^\al(\partial_iu^j+\partial_ju^i)\right)\dif x\nonumber\\
&+\mu\int\partial_k\dot{u}^j\partial_i u^k\left(\rho^\al(\partial_iu^j+\partial_ju^i)\right)\dif x+\mu\int\dot{u}^j\partial_k\partial_i u^k\left(\rho^\al(\partial_iu^j+\partial_ju^i)\right)\dif x\nonumber\\
=~&-\mu\int\partial_i\dot{u}^j\pt\left(\rho^\al(\partial_iu^j+\partial_ju^i)\right)\dif x-\mu\int\partial_i\dot{u}^ju^k\partial_k\left(\rho^\al(\partial_iu^j+\partial_ju^i)\right)\dif x\nonumber\\
&+\mu\int\rho^\al\partial_k\dot{u}^j\partial_iu^k(\partial_iu^j+\partial_ju^i)\dif x-\mu\int\rho^\al\partial_i\dot{u}^j\div u(\partial_i u^j+\partial_j u^i)\dif x\nonumber\\
\leq~& -2\mu\int\rho^\al|\frD(\dot{u})|^2\dif x+C\vr^\al\int|\D\dot{u}||\D u|^2\dif x\nonumber\\
\leq~&-\frac{3\mu}{4}\left(\frac{\vr}{2}\right)^\al\|\D\dot u\|_{L^2}^2+C\vr^\al\|\D u\|_{L^2}\|\D u\|_{L^6}^3,
\end{align}
where we have used the fact
\begin{align*}
&-\mu\int\partial_i\dot{u}^j\left(\pt\left(\rho^\al(\partial_iu^j+\partial_ju^i)\right)+u\cdot\D\left(\rho^\al(\partial_iu^j+\partial_ju^i)\right)\right)\dif x\\
=~&-\mu\int\partial_i\dot{u}^j\left((\partial_iu^j+\partial_ju^i)(\pt(\rho^\al)+u\cdot\D(\rho^\al))+\rho^\al\left((\partial_iu^j_t+\partial_ju^i_t)+u\cdot\D(\partial_iu^j+\partial_ju^i)\right)\right)\dif x\\
=~&-\mu\int\partial_i\dot{u}^j\left(\al\rho^{\al-1}(\partial_iu^j+\partial_ju^i)(\rho_t+u\cdot\D\rho)+\rho^\al\left((\partial_i\dot{u}^j+\partial_j\dot{u}^i)+u\cdot\D(\partial_iu^j+\partial_ju^i)\right)\right)\dif x\\
&-\mu\int\rho^\al\partial_i\dot{u}^j\left(\partial_i(u\cdot\D)u^j+\partial_j(u\cdot\D)u^i\right)\dif x\\
=~&\mu\al\int\rho^\al\partial_i\dot{u}^j\div u(\partial_i u^j+\partial_j u^i)\dif x-\mu\int\rho^\al\partial_i\dot{u}^j\left(\partial_i\dot{u}^j+\partial_j\dot{u}^i\right)\dif x\\
&+\mu\int\rho^\al\partial_i\dot{u}^j\left(\partial_iu^k\partial_ku^j+\partial_ju^k\partial_ku^i\right)\dif x.
\end{align*}

Similarly, we get
\begin{align}\label{3.4.5}
M_3
=~&\lambda\int\dot u^j\left(\pt\partial_j(\rho^\al\div u)+\div (u\partial_j(\rho^\al\div u))\right)\dif x\nonumber\\
=~&-\lambda\int\partial_j\dot u^j\pt(\rho^\al\div u)\dif x+\lambda\int\dot u^j\left(\div u\partial_j(\rho^\al\div u)+u\cdot\D\partial_j(\rho^\al\div u)\right)\dif x\nonumber\\
=~&-\lambda\int\div\dot u\pt(\rho^\al\div u)\dif x-\lambda\int\rho^\al\div\dot{u}(\div u)^2\dif x-\lambda\int\rho^\al\div u\dot{u}^j\partial_j(\div u)\dif x\nonumber\\
&-\lambda\int\div\dot{u}u\cdot\nabla(\rho^\al\div u)\dif x-\lambda\int\dot{u}^j\partial_ju^k\partial_k(\rho^\al\div u)\dif x\nonumber\\
=~&-\lambda\int\div\dot u\pt(\rho^\al\div u)\dif x-\lambda\int\div\dot{u}u\cdot\nabla(\rho^\al\div u)\dif x-\lambda\int\rho^\al\div\dot{u}(\div u)^2\dif x\nonumber\\
&-\lambda\int\rho^\al\div u\dot{u}^j\partial_j(\div u)\dif x+\lambda\int\rho^\al\partial_k\dot{u}^j\partial_ju^k\div u\dif x+\lambda\int\rho^\al\div u\dot{u}^j\partial_j(\div u)\dif x\nonumber\\
=~&-\lambda\int\rho^\al(\div\dot{u})^2 \dif x+\lambda\int\rho^\alpha\div\dot{u}\partial_iu^j\partial_ju^i\dif x\nonumber\\
&+\lambda(\al-1)\int\rho^\al\div\dot{u}(\div u)^2\dif x+\lambda\int\rho^\al\partial_k\dot{u}^j\partial_ju^k\div u\dif x\nonumber\\
\leq~&-\lambda\left(\frac{\vr}{2}\right)^\al\|\div\dot{u}\|_{L^2}^2+\frac{\mu}{4}\left(\frac{\vr}{2}\right)^\al\|\D\dot{u}\|_{L^2}^2+C\vr^\al\|\D u\|_{L^2}\|\D u\|_{L^6}^3,
\end{align}
where we have used the fact
\begin{align*}
&-\lambda\int\div\dot{u}\left(\pt(\rho^\al\div u)+u\cdot\D(\rho^\al\div u)\right)\dif x\\
=~&-\lambda\int\div\dot{u}\left(\div u(\pt(\rho^\al)+u\cdot\D(\rho^\al))+\rho^\al(\div u_t+u\cdot\D(\div u))\right)\dif x\\
=~&-\lambda\int\div\dot{u}\left(\al\rho^{\al-1}\div u(\rho_t+u\cdot\D\rho)+\rho^\al(\div\dot{u}+u\cdot\D(\div u)-\div((u\cdot\D)u))\right)\dif x\\
=~&\lambda\al\int\rho^\al\div\dot{u}(\div u)^2\dif x-\lambda\int\rho^\al(\div\dot{u})^2\dif x+\lambda\int\rho^\al\div\dot{u}\partial_iu^j\partial_ju^i\dif x.
\end{align*}

Combining \eqref{3.4.1} and \eqref{3.4.3}--\eqref{3.4.5}, we get
\begin{align}\label{3.4.6}
\frac{\dif}{\dif t}\int\rho|\dot u|^2\dif x+\left(\frac{\vr}{2}\right)^\al\|\D\dot u\|_{L^2}^2\leq ~&C\left(\vr\|\sqrt{\rho}\dot\theta\|_{L^2}^2+\vr^2\|\D\theta\|_{L^2}^2\|\D u\|_{L^2}\|\D u\|_{L^6}+\vr^\al\|\D u\|_{L^2}\|\D u\|_{L^6}^3\right)\nonumber\\
\leq~&C(K)\left(\vr\|\sqrt{\rho}\dot\theta\|_{L^2}^2+\vr^{\al+2}+\vr^\al\right)
\end{align}
Integrating \eqref{3.4.6} over $(0,\sigma(T))$ and using \eqref{assume:1}, \eqref{estimate:DuL6:2}, and \eqref{estimate:DuL6:3}, we obtain
\begin{align}\label{3.4.7}
\int\rho|\dot u|^2\dif x+\left(\frac{\vr}{2}\right)^\al\int_0^{\sigma(T)}\|\D\dot u\|_{L^2}^2\dif t
\leq~&\int\rho|\dot u|^2\dif x\bigg|_{t=0}+C(K)\vr\int_0^{\sigma(T)}\|\sqrt{\rho}\dot\theta\|_{L^2}^2\dif t+C(K)\vr^{\al+2}\nonumber\\
\leq~&C_3\vr^{2\al-1}+C(K)\vr^{\al+2},
\end{align}
where we have used
\begin{align}\label{3.4.8}
\int\rho|\dot u|^2\dif x\bigg|_{t=0}\leq~& C\vr^{-1}\left(\left\|\D(\rho_0\theta_0)\right\|_{L^2}^2+\left\|\D(\rho^\al_0\D u_0)\right\|_{L^2}^2\right)\nonumber\\
\leq~&C\vr^{-1}\left(\left\|\theta_0\D\rho_0\right\|_{L^2}^2+\vr^2\left\|\D\theta_0\right\|_{L^2}^2+\vr^{2\al}\left\|\D^2u_0\right\|_{L^2}^2+\vr^{2\al-2}\left\|\D\rho_0\D u_0\right\|_{L^2}^2\right)\nonumber\\
\leq~&C_3\vr^{2\al-1}.
\end{align}

Inserting \eqref{3.4.8} into \eqref{3.4.7}, we get
\begin{align}\label{3.4.9}
\int\rho|\dot u|^2\dif x+\left(\frac{\vr}{2}\right)^\al\int_0^{\sigma(T)}\|\D\dot u\|_{L^2}^2\dif t\leq C_3\vr^{2\al-1}+C(K)\vr^{\al+2}\leq K_3\vr^{2\al-1},
\end{align}
provided $\al\geq 3$, $\vr\geq\rL_2$, where we have taken $K_3$ and $\rL_2$ large enough, such that $C_3\vr^{2\al-1}\leq\frac{1}{2}K_3\vr^{2\al-1}$,  and $C(K)\vr^{\al+2}\leq\frac{1}{2}K_3\vr^{2\al-1}$.
\end{proof}

The basic energy estimate in Lemma \ref{Lem3.1} cannot yield directly the bounds on the $L^2$-norm, with respect to time and space, of the first-order derivatives of both the velocity and the temperature since we do not have the super norm of the temperature on $[0,\sigma(T)]$.  In fact they could be obtained in the following higher-order estimates, however 
 the bounds are too large to close the a priori estimates due to the large initial data of the density. Under this reason, we give following Lemma \ref{Lem35} which plays a fundamental role throughout this paper.
\begin{lemma}\label{Lem35}
Under the conditions of Proposition \ref{p4.1}, there exist positive constants $K_4$ and $\rL_3$ depending on $\mu$, $\lambda$, $\kappa$, and $R$ such that if $(\rho,u,\theta)$ is a strong solution of \eqref{FCNS:2}--\eqref{data and far} on $\R^3\times(0,T]$, it holds that
\begin{align}\label{g:3.5}
A_4(\sigma(T))=\sup\limits_{t\in[0,\sigma(T)]}\int\rho(\theta-1)^2\dif x+\int_0^{\sigma(T)}\left(\left(\frac{\vr}{2}\right)^\al\|\D u\|_{L^2}^2+\kappa\|\D\theta\|_{L^2}^2 \right)\dif t\leq K_4\vr,
\end{align}
provided that $\vr\geq \rL_3$. 
\end{lemma}
\begin{proof}
First, multiplying $(\ref{FCNS:2})_2$ by $u$ and integrating the resulting equality over $\mathbb{R}^3$, together with \eqref{FCNS:2}$_1$, we have
\begin{align}\label{3.5.1}
\frac{\dif}{\dif t}\int\left(\frac{1}{2}\rho|u|^2+R(\rho\ln\rho-\rho\ln\vr-\rho+\vr)\right)\dif x+\left(\frac{\vr}{2}\right)^\al\int|\D{u}|^2\dif x\leq C\vr\|\theta-1\|_{L^2}\|\D{u}\|_{L^2}\leq CK_4^{\frac{1}{2}}\vr\|\D{u}\|_{L^2}.
\end{align}

Next, multiplying $(\ref{FCNS:2})_3$ by $\theta-1$ and integrating the resulting equality over
$\mathbb{R}^3$ lead to
\begin{align}\label{3.5.2}
&\frac{1}{2}\frac{\dif}{\dif t}\int\rho\left(\theta-1\right)^2\dif x+\kappa\int|\D\theta|^2\dif x\nonumber\\
\leq~&C\vr\int\theta|\theta-1|\div u\dif x+C\vr^\al\int|\theta-1||\D u|^2\dif x\nonumber\\
\leq~&C\vr\int|\theta-1||\D u|\dif x+C\vr\int|\theta-1|^2|\D u|\dif x+C\vr^\al\int|\theta-1||\D u|^2\dif x\nonumber\\
\leq~&C\vr\|\theta-1\|_{L^2}\|\D u\|_{L^2}+C\vr\|\theta-1\|_{L^2}\|\theta-1\|_{L^3}\|\D u\|_{L^6}+C\vr^\al\|\theta-1\|_{L^2}\|\D u\|_{L^4}^2\nonumber\\
\leq~&C\vr\|\theta-1\|_{L^2}\|\D u\|_{L^2}+C\vr\|\theta-1\|_{L^2}^{\frac{3}{2}}\|\D\theta\|_{L^2}^{\frac{1}{2}}\|\D u\|_{L^6}+C\vr^\al\|\theta-1\|_{L^2}\|\D u\|_{L^2}^{\frac{1}{2}}\|\D u\|_{L^6}^{\frac{3}{2}}\nonumber\\
\leq~&CK_4^{\frac{1}{2}}\vr\|\D u\|_{L^2}+CK_4^{\frac{3}{4}}\vr\|\D\theta\|_{L^2}^{\frac{1}{2}}\|\D u\|_{L^6}+CK_4^{\frac{1}{2}}\vr^\al\|\D u\|_{L^2}^{\frac{1}{2}}\|\D u\|_{L^6}^{\frac{3}{2}}.
%\leq~&C\vr\left(\|\theta-1\|_{L^2}^\frac{1}{2}\|\theta-1\|_{L^6}^\frac{3}{2}\|\D{u}\|_{L^2}+\|\theta-1\|_{L^2}\|\D{u}\|_{L^2}\right)+C\vr^\al\|\theta-1\|_{L^2}^\frac{1}{2}\|\theta-1\|_{L^6}^\frac{1}{2}\|\D{u}\|_{L^2}\|\D{u}\|_{L^6}\nonumber\\
%\leq~&C\vr\|\D{u}\|_{L^2}\left(\|\theta-1\|_{L^2}^\frac{1}{2}\|\D\theta\|_{L^2}^\frac{3}{2}+\|\D\theta\|_{L^2}+1\right)+C\vr^\al\|\theta-1\|_{L^2}^\frac{1}{2}\|\D\theta\|_{L^2}^\frac{1}{2}\|\D{u}\|_{L^2}\|\D{u}\|_{L^6}\nonumber\\
%\leq~&\frac{1}{8}\left(\frac{\vr}{2}\right)^\al\|\D{u}\|_{L^2}^2+C\vr^{2-\al}\left(\|\theta-1\|_{L^2}\|\D\theta\|_{L^2}^3+1+\|\D\theta\|_{L^2}^2\right)\nonumber\\
%&+C\vr^\al\|\theta-1\|_{L^2}^\frac{1}{2}\|\D\theta\|_{L^2}^\frac{1}{2}\|\D{u}\|_{L^2}\|\D u\|_{L^6}\nonumber\\
%\leq~&\frac{1}{4}\left(\frac{\vr}{2}\right)^\al\|\D{u}\|_{L^2}^2+C\vr^{2-\al}\left(1+\|\D\theta\|_{L^2}^2+\|\theta-1\|_{L^2}\|\D\theta\|_{L^2}^3\right)+C\vr^\al\|\theta-1\|_{L^2}\|\D\theta\|_{L^2}\|\D u\|_{L^6}^2.
\end{align}
Integrating \eqref{3.5.1} and \eqref{3.5.2} over $(0,\sigma(T))$, and adding the resulting inequality together,  we have
\begin{align}\label{3.5.3}
&\int\left(\rho|{u}|^2+R(\rho\ln\rho-\rho\ln\vr-\rho+\vr)+\rho\left(\theta-1\right)^2\right)\dif x+\int_0^{\sigma(T)}\left(\left(\frac{\vr}{2}\right)^\al\|\D u\|_{L^2}^2+\kappa\|\D\theta\|_{L^2}^2\right)\dif t\nonumber\\
\leq~&\int\left(\rho_0|{u}_0|^2+R(\rho_0\ln\rho_0-\rho_0\ln\vr-\rho_0+\vr)+\rho_0\left(\theta_0-1\right)^2\right)\dif x\nonumber\\
&+CK_4^{\frac{1}{2}}\vr\left(\int_0^{\sigma(T)}\|\D u\|_{L^2}^2\dif t\right)^{\frac{1}{2}}+CK_4^{\frac{3}{4}}\vr\left(\int_0^{\sigma(T)}\|\D\theta\|_{L^2}^2\dif t\right)^{\frac{1}{4}}\left(\int_0^{\sigma(T)}\|\D u\|_{L^6}^2\dif t\right)^{\frac{1}{2}}\nonumber\\
&+CK_4^{\frac{1}{2}}\vr^\al\left(\int_0^{\sigma(T)}\|\D u\|_{L^2}^2\dif t\right)^{\frac{1}{4}}\left(\int_0^{\sigma(T)}\|\D u\|_{L^6}^2\dif t\right)^{\frac{3}{4}}\nonumber\\
\leq~&C_4\vr+CK_1^{\frac{1}{2}}K_4^{\frac{1}{2}}\vr^{\frac{3}{2}-\frac{\al}{2}}+CK_4(K_1+K_9)^{\frac{1}{2}}\vr^{\frac{7}{4}-\frac{\al}{2}}+C_4\left(K_4(K_1+K_9)\right)^{\frac{3}{4}}\vr\nonumber\\
\leq~&C_4\left(1+\left(K_4(K_1+K_9)\right)^{\frac{3}{4}}\right)\vr+C(K)\vr^{\frac{7}{4}-\frac{\al}{2}}\nonumber\\
\leq~&K_4\vr,
\end{align}
provided $\vr\geq\rL_3$, where we have taken $K_4$ and $\rL_3$ large enough such that $C_4\left(1+\left(K_4(K_1+K_9)\right)^{\frac{3}{4}}\right)\leq\frac{1}{2}K_4$ and $C(K)\vr^{\frac{7}{4}-\frac{\al}{2}}\leq\frac{1}{2}K_4\vr$.
\end{proof}

\subsubsection{Time-weighted uniform estimates of velocity and temperature on $[0,T]$}\label{sec3.2}

Next, in Lemma \ref{sky}--\ref{Lem 3.5}, we derive the first-order and the second-order time-weighted uniform estimates of $(u,\theta)$ on the time interval $[0,T]$. First, based on the Lemma \ref{Lem3.1} and \ref{Lem35}, we have Lemma \ref{sky} which plays an important role in the following estimates on $A_i(T) ~ (i=5,6,7,8)$.

\begin{lemma}\label{sky}
Under the conditions of Proposition \ref{p4.1}, there exist positive constants $K_5$ and $\rL_4$ depending on $\mu$, $\lambda$, $\kappa$, and $R$ such that if $(\rho,u,\theta)$ is a strong solution of \eqref{FCNS:2}--\eqref{data and far} on $\R^3\times(0,T]$, it holds that
\begin{align}\label{-4.1:2}
\int_0^T\|\D u\|_{L^2}^2\dif t\leq K_5\vr^{2-\al}, \quad \int_0^T\|\D\theta\|_{L^2}^2\dif t\leq K_5\vr^3,
\end{align}
provided that $\vr\geq\rL_4$.
\end{lemma}
\begin{proof}
Inserting \eqref{assume:1.75} into \eqref{Lem3.1:3-2}, and taking $K_5\geq C\left(K_8+K_8^{\frac{1}{2}}\right)$, we have
\begin{align}\label{-4.1:3}
\int_{\sigma(T)}^T\|\D u\|_{L^2}^2\dif t\leq K_5\vr^{2-\al}, \quad \int_{\sigma(T)}^T\|\D\theta\|_{L^2}^2\dif t\leq K_5\vr^3.
\end{align}
Then \eqref{-4.1:2} is followed by \eqref{-4.1:3} and \eqref{g:3.5}, provided $\rL_4\geq\max\{1, K_4^2\}$.
\end{proof}

In Lemma \ref{Lem:4}, we estimate $A_{5,i}(T) ~ (i=1,2,3)$ (see (\ref{B1,1})--(\ref{B1,3})), separately, by which,  after a linear combination of them, we can derive the uniform estimate of $A_5(T)$.

\begin{lemma}\label{Lem:4}
Under the conditions of Proposition \ref{p4.1}, there exist positive constants $M$, $K_6$ and $\rL_5$ depending on $\mu$, $\lambda$, $\kappa$, and $R$ such that if $(\rho,u,\theta)$ is a strong solution of \eqref{FCNS:2}--\eqref{data and far} on $\R^3\times(0,T]$, it holds that
\begin{align}\label{Lem:4:1}
A_5(T)=~&2M\sup\limits_{t\in[0,T]}\sigma^2\left(\frac{\vr}{2}\right)^{\al+1}\|\D u\|_{L^2}^2+M\vr\int_0^T\int \sigma^2\rho|\dot{u}|^2\dif x\dif t\nonumber\\
&+\kappa M^2\sup\limits_{t\in[0,T]}\sigma^3\vr\|\D \theta\|_{L^2}^2+M^2\vr\int_0^T\int \sigma^3\rho|\dot{\theta}|^2\dif x\dif t\nonumber\\
&+\sup\limits_{t\in[0,T]}\sigma^3\int\rho|\dot{u}|^2\dif x+\left(\frac{\vr}{2}\right)^\al\int_0^T\int\sigma^3|\D \dot{u}|^2\dif x\dif t\leq K_6\vr^2,
\end{align}
provided that $\vr\geq\rL_5$.
\end{lemma}
\begin{proof}
Firstly, we derive the time-weighted uniform bound of $\|\D u\|_{L^6}$. Combining \eqref{estimate:DuL6}, \eqref{assume:2} and \eqref{assume:3}, we obtain
\begin{align}\label{estimate:DuL6:4}
\sigma^3\|\D u\|_{L^6}^2\leq~&C\left(\sigma^3\vr^{1-2\al}\|\sqrt{\rho}\dot{u}\|_{L^2}^2+\sigma^3\vr^{2-2\al}\|\D\theta\|_{L^2}^2+\vr^{-2\al}\|\D\rho\|_{L^2}^2+\sigma^2\vr^{-8}\|\D\rho\|_{L^4}^8\|\D u\|_{L^2}^2\right)\nonumber\\
%\leq~&C(K)\vr^{3-2\al}+CK_6\vr^{-1-\al}
\leq~&CK_6\vr^{-1-\al},
\end{align}
provided $\al>4$. By \eqref{estimate:DuL6}, \eqref{assume:1}, \eqref{assume:2}, \eqref{assume:3} and \eqref{-4.1:2}, we obtain
\begin{align}\label{estimate:DuL6:5}
\int_0^T\|\D u\|_{L^6}^2\dif t\leq~&C\vr^{1-2\al}\int_0^T\|\sqrt{\rho}\dot{u}\|_{L^2}^2\dif t+C\vr^{2-2\al}\int_0^T\|\D\theta\|_{L^2}^2\dif t\nonumber\\
&+C\vr^{-2\al}\int_0^T\|\D\rho\|_{L^2}^2\dif t+C\vr^{-8}\|\D\rho\|_{L^4}^8\int_0^T\|\D u\|_{L^2}^2\dif t\nonumber\\
\leq~&C\vr^{1-2\al}\left(\int_0^{\sigma(T)}\|\sqrt{\rho}\dot{u}\|_{L^2}^2\dif t+\int_0^T\sigma^2\|\sqrt{\rho}\dot{u}\|_{L^2}^2\dif t\right)+C\vr^{2-2\al}\int_0^T\|\D\theta\|_{L^2}^2\dif t\nonumber\\
&+C\vr^{-2\al}\int_0^T\|\D\rho\|_{L^2}^2\dif t+C\vr^{-8}\|\D\rho\|_{L^4}^8\int_0^T\|\D u\|_{L^2}^2\dif t\nonumber\\
%\leq~&C(K_1+K_9)\vr^{1-\al}+C(K)\vr^{5-2\al}+C(K)\vr^{-\al}\nonumber\\
\leq~&C(K_1+K_9)\vr^{1-\al}.
\end{align}

Next, in view of \eqref{B1}, we divide the proof into the following four steps.

\textbf{Step 1}: Estimates of $A_{5,1}(T)$.

Multiplying $(\ref{FCNS:2})_2$ by $\sigma^2\dot{u}$ and integrating the resulting equality over $\R^3$ yield
\begin{align}\label{Lem:4:2}
\sigma^2\int\rho|\dot{u}|^2\dif x=~&-\sigma^2\int\dot{u}\cdot\D P\dif x-2\mu\sigma^2\int\rho^\al\D\dot{u}:\frD(u)\dif x-\lambda\sigma^2\int\rho^\al\div u\div\dot{u}\dif x\nonumber\\
:=~&\sum\limits_{i=1}^3N_i.
\end{align}

Using \eqref{eqn:P}, we get after integration by parts that
\begin{align}\label{Lem:4:4}
N_1=~&R\sigma^2\int(\rho\theta-\vr)\div u_t\dif x+R\sigma^2\int(\rho\theta-\vr)u\cdot\D\div u\dif x+R\sigma^2\int(\rho\theta-\vr)\partial_ju^i\partial_iu^j\dif x\nonumber\\
=~&R\frac{\dif}{\dif t}\left(\sigma^2\int(\rho\theta-\vr)\div u\dif x\right)-2R\sigma\sigma'\int(\rho\theta-\vr)\div u\dif x-\sigma^2\int P_t\div u \dif x\nonumber\\
&+R\sigma^2\int(\rho\theta-\vr)u\cdot\D\div u\dif x+R\sigma^2\int(\rho\theta-\vr)\partial_ju^i\partial_iu^j\dif x\nonumber\\
=~&R\frac{\dif}{\dif t}\left(\sigma^2\int(\rho\theta-\vr)\div u\dif x\right)-2R\sigma\sigma'\int(\rho\theta-\vr)\div u\dif x+\sigma^2\int\div(Pu)\div u \dif x\nonumber\\
&-R\sigma^2\int\rho\dot{\theta}\div u\dif x+R\sigma^2\int(\rho\theta-\vr)u\cdot\D\div u\dif x+R\sigma^2\int(\rho\theta-\vr)\partial_ju^i\partial_iu^j\dif x\nonumber\\
=~&R\frac{\dif}{\dif t}\left(\sigma^2\int(\rho\theta-\vr)\div u\dif x\right)-2R\sigma\sigma'\int(\rho\theta-\vr)\div u\dif x-R\sigma^2\int\rho\dot{\theta}\div u\dif x\nonumber\\
&+R\sigma^2\int\rho\theta\partial_ju^i\partial_iu^j\dif x\nonumber\\
\leq~&R\frac{\dif}{\dif t}\left(\sigma^2\int(\rho\theta-\vr)\div u\dif x\right)+C\sigma\sigma'\|\rho\theta-\vr\|_{L^2}\|\D u\|_{L^2}+C\vr^{\frac{1}{2}}\sigma^2\|\sqrt{\rho}\dot{\theta}\|_{L^2}\|\D u\|_{L^2}\nonumber\\
&+C\vr\sigma^2\int\theta|\D u|^2\dif x\nonumber\\
\leq~&R\frac{\dif}{\dif t}\left(\sigma^2\int(\rho\theta-\vr)\div u\dif x\right)+C\sigma\sigma'\left(\vr\|\theta-1\|_{L^2}+\|\rho-\vr\|_{L^2}\right)\|\D u\|_{L^2}\nonumber\\
&+C\vr^{\frac{1}{2}}\sigma^2\|\sqrt{\rho}\dot{\theta}\|_{L^2}\|\D u\|_{L^2}+C\vr\sigma^2\|\D u\|_{L^2}^{\frac{3}{2}}\|\D\theta\|_{L^2}\|\D u\|_{L^6}^{\frac{1}{2}}+C\vr\sigma^2\|\D u\|_{L^2}^2
\end{align}

Next, integration by parts gives
\begin{align}\label{Lem:4:7}
N_2=~&-2\mu\sigma^2\int\rho^\al\D u_t:\frD(u)\dif x-2\mu\sigma^2\int\rho^\al\D((u\cdot\D) u):\frD(u)\dif x\nonumber\\
=~&-\mu\frac{\dif}{\dif t}\left(\sigma^2\int\rho^\al|\frD(u)|^2\dif x\right)+2\mu\sigma\sigma'\int\rho^\al|\frD(u)|^2\dif x-\al\mu\sigma^2\int\rho^{\al-1}\div(\rho u)|\frD(u)|^2\dif x\nonumber\\
&+\mu\sigma^2\int\div(\rho^\al u)|\frD(u)|^2\dif x-\mu\sigma^2\int\rho^\al\partial_ju^k\partial_ku^i\left(\partial_iu^j+\partial_ju^i\right)\dif x\nonumber\\
\leq~&-\mu\frac{\dif}{\dif t}\left(\sigma^2\int\rho^\al|\frD(u)|^2\dif x\right)+C\vr^\al\sigma\sigma'\|\D u\|_{L^2}^2+C\vr^\al\sigma^2\|\D u\|_{L^2}^{\frac{3}{2}}\|\D u\|_{L^6}^{\frac{3}{2}},
\end{align}
and
\begin{align}\label{Lem:4:8}
N_3\leq -\frac{\lambda}{2}\frac{\dif}{\dif t}\left(\sigma^2\int\rho^\al(\div u)^2\dif x\right)+C\vr^\al\sigma\sigma'\|\D u\|_{L^2}^2+C\vr^\al\sigma^2\|\D u\|_{L^2}^{\frac{3}{2}}\|\D u\|_{L^6}^{\frac{3}{2}}.
\end{align}

Substituting \eqref{Lem:4:4}--\eqref{Lem:4:8} into \eqref{Lem:4:2}, we obtain after integration with respect to $t\in[0,T]$ that
\begin{align}\label{Lem:4:9}
&\sigma^2\left(\frac{\vr}{2}\right)^\al\|\D u\|_{L^2}^2+\int_0^T\int\sigma^2\rho|\dot{u}|^2\dif x\dif t\nonumber\\
\leq~&C\sigma^2\int\rho^\al\left(\mu|\frD(u)|^2+\frac{\lambda}{2}(\div u)^2\right)\dif x+\int_0^T\int\sigma^2\rho|\dot{u}|^2\dif x\dif t\nonumber\\
\leq~&C\sup\limits_{t\in[0,T]}\sigma^2\left(\vr\|\theta-1\|_{L^2}+\|\rho-\vr\|_{L^2}\right)\|\D u\|_{L^2}+C\int_0^{\sigma(T)}\sigma\left(\vr\|\theta-1\|_{L^2}+\|\rho-\vr\|_{L^2}\right)\|\D u\|_{L^2}\dif t\nonumber\\
&+C\vr^{\frac{1}{2}}\left(\int_0^T\sigma^3\|\sqrt{\rho}\dot{\theta}\|_{L^2}^2\dif t\right)^{\frac{1}{2}}\left(\int_0^T\|\D u\|_{L^2}^2\dif t\right)^{\frac{1}{2}}+C\vr\int_0^T\sigma^2\|\D u\|_{L^2}^{\frac{3}{2}}\|\D u\|_{L^6}^{\frac{1}{2}}\|\D\theta\|_{L^2}\dif t\nonumber\\
&+C\vr\int_0^T\|\D u\|_{L^2}^2\dif t+C\vr^\al\int_0^{\sigma(T)}\|\D u\|_{L^2}^2\dif t+C\vr^\al\int_0^T\sigma^2\|\D u\|_{L^2}^{\frac{3}{2}}\|\D u\|_{L^6}^{\frac{3}{2}}\dif t\nonumber\\
\leq~&C\vr\sup\limits_{t\in[0,T]}\sigma^2\left(1+\|\D\theta\|_{L^2}\right)\|\D u\|_{L^2}+C\vr\int_0^{\sigma(T)}\left(1+\|\D\theta\|_{L^2}\right)\|\D u\|_{L^2}\dif t\nonumber\\
&+\int_0^T\int\sigma^3\rho|\dot{\theta}|^2\dif x\dif t+C\vr\int_0^T\sigma^2\|\D u\|_{L^2}^{\frac{3}{2}}\|\D u\|_{L^6}^{\frac{1}{2}}\|\D\theta\|_{L^2}\dif t+C\vr\int_0^T\|\D u\|_{L^2}^2\dif t\nonumber\\
&+C\vr^\al\int_0^{\sigma(T)}\|\D u\|_{L^2}^2\dif t+C\vr^\al\int_0^T\sigma^2\|\D u\|_{L^2}^{\frac{3}{2}}\|\D u\|_{L^6}^{\frac{3}{2}}\dif t,
\end{align}
where we have used \eqref{lem3.1:g3} in the last inequality.

%\begin{lemma}\label{Lem:5}
%Under the conditions of Proposition \ref{p4.1}, there exist a positive constant $\rL_6$ depending on $\mu$, $\lambda$, $\kappa$, and $R$ such that if $(\rho,u,\theta)$ is a strong solution of \eqref{FCNS:2}--\eqref{data and far} on $\R^3\times(0,T]$, it holds that
%\begin{align}\label{Lem:5:1}
%A_{5,2}(T)=\kappa\sup\limits_{t\in[0,T]}\sigma^3\|\D \theta\|_{L^2}^2+\int_0^T\int \sigma^3\rho|\dot{\theta}|^2\dif x\dif t\leq K\vr^{\frac{3}{2}},
%\end{align}
%provided that $\vr\geq\rL_6$.
%\end{lemma}
%\begin{proof}

\bigskip

\textbf{Step 2}: Estimates of $A_{5,2}(T)$.

Multiplying $(\ref{FCNS:2})_3$ by $\sigma^3\theta_t$ and integrating the resulting equality over $\R^3$ lead to
\begin{align}\label{Lem:5:2}
&\frac{\kappa}{2}\frac{\dif}{\dif t}\left(\sigma^3\int|\D\theta|^2\dif x\right)+\sigma^3\int\rho|\dot{\theta}|^2\dif x\nonumber\\
=~&\frac{3}{2}\kappa\sigma^2\sigma'\int|\D\theta|^2\dif x+\sigma^3\int\rho\dot{\theta}u\cdot\D\theta\dif x-R\sigma^3\int\rho\theta\div u\dot{\theta}\dif x\nonumber\\
&+R\sigma^3\int\rho\theta u\cdot\D\theta\div u\dif x+2\mu\sigma^3\int\rho^\al\dot{\theta}|\frD(u)|^2\dif x+\lambda\sigma^3\int\rho^\al\dot{\theta}(\div u)^2\dif x\nonumber\\
&-2\mu\sigma^3\int\rho^\al u\cdot\D\theta|\frD(u)|^2\dif x-\lambda\sigma^3\int\rho^\al u\cdot\D\theta(\div u)^2\dif x.
\end{align}

Similar to \eqref{3.3:11}--\eqref{3.3:14}, we obtain after integration with respect to $t\in[0,T]$ that
\begin{align}\label{Lem:5:4}
&\kappa\sigma^3\int|\D\theta|^2\dif x+\int_0^T\int\sigma^3\rho|\dot{\theta}|^2\dif x\dif t\nonumber\\
\leq~&\frac{3}{2}\kappa\int_0^{\sigma(T)}\|\D\theta\|_{L^2}^2\dif t+C\vr\int_0^T\|\D u\|_{L^2}^2\dif t+C\vr\int_0^T\sigma^3\|\D u\|_{L^2}\|\D u\|_{L^6}\|\D\theta\|_{L^2}^2\dif t\nonumber\\
&+C\vr^{2\al-1}\int_0^T\sigma^3\|\D u\|_{L^2}\|\D u\|_{L^6}^3\dif t.
\end{align}

\bigskip

\textbf{Step 3}: Estimates of $A_{5,3}(T)$.

Operating $\sigma^3\dot{u}^j(\pt+\mathrm{div}(u\cdot))$ to $(\ref{FCNS:2})_2^j$, summing with respect to $j$, and integrating the resulting equation over $\R^3$, we obtain
\begin{align}\label{3.42}
&\frac{1}{2}\frac{\dif}{\dif t}\int\sigma^3\rho|\dot u|^2\dif x-\frac{3}{2}\sigma^2\sigma'\int \rho|\dot u|^2\dif x\nonumber\\
=~& -\int\sigma^3\dot u^j\left(\partial_jP_t+\div (u\partial_j P)\right)\dif x+2\mu\int\sigma^3\dot u^j\Big[\pt(\div(\rho^\al\frD(u)))+\div (u\div(\rho^\al\frD(u)))\Big]\dif x\nonumber\\
&+\lambda\int\sigma^3\dot u^j\Big[\pt(\partial_j(\rho^\al\div u))+\div (u\cdot\partial_j(\rho^\al\div u))\Big]\dif x. 
\end{align} 
Similar to \eqref{3.4.3}--\eqref{3.4.5}, we have
\begin{align}\label{e3.308}
&\frac{1}{2}\frac{\dif}{\dif t}\int\sigma^3\rho|\dot u|^2\dif x+\frac{\mu\sigma^3}{4}\left(\frac{\vr}{2}\right)^\al\|\D\dot u\|_{L^2}^2\nonumber\\
\leq~&\frac{3}{2}\sigma^2\sigma'\int\rho|\dot u|^2\dif x+C\sigma^3\left(\vr\|\sqrt{\rho}\dot\theta\|_{L^2}^2+\vr^2\|\D\theta\|_{L^2}^2\|\D u\|_{L^2}\|\D u\|_{L^6}+\vr^\al\|\D u\|_{L^2}\|\D u\|_{L^6}^3\right).
\end{align}
Integrating \eqref{e3.308} over $[0,T]$, we have
\begin{align}\label{e3.3011}
&\sigma^3\int\rho|\dot u|^2\dif x+\left(\frac{\vr}{2}\right)^\al\int_0^T\sigma^3\|\D\dot u\|_{L^2}^2\dif t\nonumber\\
\leq~&C_5\int_0^{\sigma(T)}\int\sigma^2\rho|\dot u|^2\dif x\dif t+C_5\vr\int_0^T\sigma^3\|\sqrt{\rho}\dot\theta\|_{L^2}^2\dif t\nonumber\\
&+C\vr^2\int_0^T\sigma^3\|\D\theta\|_{L^2}^2\|\D u\|_{L^2}\|\D u\|_{L^6}\dif t+C\vr^\al\int_0^T\sigma^3\|\D u\|_{L^2}\|\D u\|_{L^6}^3\dif t.
\end{align}

\bigskip

\textbf{Step 4}: Estimates of $A_5(T)$.

Now we take $M=1+C_5$. Multiplying \eqref{Lem:4:9} by $M\vr$, multiplying \eqref{Lem:5:4} by $M^2\vr$, and adding them to \eqref{e3.3011}, we obtain
\begin{align}\label{B1:1}
A_5(T)\leq~&C\vr^2\sup\limits_{t\in[0,T]}\sigma^2\left(1+\|\D\theta\|_{L^2}\right)\|\D u\|_{L^2}+C\vr^2\int_0^{\sigma(T)}\left(1+\|\D\theta\|_{L^2}\right)\|\D u\|_{L^2}\dif t\nonumber\\
&+C\vr^{\al+1}\int_0^{\sigma(T)}\|\D u\|_{L^2}^2\dif t+C\vr\int_0^{\sigma(T)}\|\D\theta\|_{L^2}^2\dif t+C\vr^2\int_0^T\|\D u\|_{L^2}^2\dif t\nonumber\\
&+C\vr^{\al+1}\int_0^T\sigma^2\|\D u\|_{L^2}^{\frac{3}{2}}\|\D u\|_{L^6}^{\frac{3}{2}}\dif t+C\vr^2\int_0^T\sigma^3\|\D\theta\|_{L^2}^2\|\D u\|_{L^2}\|\D u\|_{L^6}\dif t\nonumber\\
&+C\vr^{2\al}\int_0^T\sigma^3\|\D u\|_{L^2}\|\D u\|_{L^6}^3\dif t+C\vr^2\int_0^T\sigma^2\|\D u\|_{L^2}^{\frac{3}{2}}\|\D u\|_{L^6}^{\frac{1}{2}}\|\D\theta\|_{L^2}\dif t\nonumber\\
\leq~&C\vr^2\sup\limits_{t\in[0,T]}\sigma\|\D u\|_{L^2}+C\vr^2\|\D \theta\|_{L^2}^{\frac{1}{3}}\sup\limits_{t\in[0,T]}\left(\sigma\|\D u\|_{L^2}\right)\sup\limits_{t\in[0,T]}\left(\sigma\|\D\theta\|_{L^2}^{\frac{2}{3}}\right)\nonumber\\
&+C\vr^2\int_0^{\sigma(T)}\left(1+\|\D\theta\|_{L^2}\right)\|\D u\|_{L^2}\dif t+C\vr^{\al+1}\int_0^{\sigma(T)}\|\D u\|_{L^2}^2\dif t+C\vr\int_0^{\sigma(T)}\|\D\theta\|_{L^2}^2\dif t\nonumber\\
&+C\vr^2\int_0^T\|\D u\|_{L^2}^2\dif t+C\vr^{\al+1}\sup\limits_{t\in[0,T]}\left(\sigma^2\|\D u\|_{L^6}\right)\int_0^T\|\D u\|_{L^2}^{\frac{3}{2}}\|\D u\|_{L^6}^{\frac{1}{2}}\dif t\nonumber\\
&+C\vr^2\sup\limits_{t\in[0,T]}\left(\sigma^3\|\D\theta\|_{L^2}^2\right)\int_0^T\|\D u\|_{L^2}\|\D u\|_{L^6}\dif t\nonumber\\
&+C\vr^{2\al}\sup\limits_{t\in[0,T]}\left(\sigma^3\|\D u\|_{L^6}^2\right)\int_0^T\|\D u\|_{L^2}\|\D u\|_{L^6}\dif t+C\vr^2\sup\limits_{t\in[0,T]}\left(\sigma^2\|\D\theta\|_{L^2}\right)\int_0^T\|\D u\|_{L^2}^{\frac{3}{2}}\|\D u\|_{L^6}^{\frac{1}{2}}\dif t\nonumber\\
\leq~&C_6K_1\vr^2+C(K)\left(\vr^{\frac{1}{2}}+\vr^{\frac{17}{6}-\frac{\al}{3}}\right)\nonumber\\
\leq~&K_6\vr^2,
\end{align}
provided $\vr\geq\rL_5$, where we have used \eqref{assume:1}, \eqref{assume:2}, \eqref{-4.1:2}, \eqref{estimate:DuL6:4},  \eqref{estimate:DuL6:5}, and taken $K_6$ and $\rL_5$ large enough, such that $C_6K_1\leq\frac{1}{2}K_6$ and $C(K)\left(\vr^{\frac{1}{2}}+\vr^{\frac{17}{6}-\frac{\al}{3}}\right)\leq\frac{1}{2}K_6\vr$.
\end{proof}

In the following Lemma \ref{Lem 3.5}, we show the second-order time-weighted uniform estimate of temperature. Then, by using interpolation inequality,  the uniform upper bound of temperature on the time interval $[\sigma(T), T]$ is obtained.

\begin{lemma}\label{Lem 3.5}
Under the conditions of Proposition \ref{p4.1}, there exist positive constants $\rL_6$, $K_7$ and $K_8$ depending on $\mu$, $\lambda$, $\kappa$, and $R$, such that if $(\rho,u,\theta)$ is a strong solution of \eqref{FCNS:2}--\eqref{data and far} on $\R^3\times(0,T]$, it holds
\begin{equation}\label{3.5:2}
A_6(T)=\sup\limits_{t\in[0,T]}\sigma^6\int\rho|\dot{\theta}|^2\dif x+\kappa\int_0^T\int\sigma^6| \D \dot{\theta}|^2\dif x\dif t\leq K_7\vr^2,
\end{equation}
and
\begin{equation}\label{3.5:1}
\sup\limits_{t\in[\sigma(T),T]}\|\theta-1\|_{L^\infty}^2\leq K_8\vr^2,
\end{equation}
provided that $\vr\geq \rL_6$.
\end{lemma}
\begin{proof}
Applying the operator $\pt+\div(u\cdot)$ to $(\ref{FCNS:2})_3$, we get 
\begin{align}\label{3.5:3}
&\rho\pt\dot{\theta}+\rho u\cdot\D\dot{\theta}-\kappa\Delta\dot{\theta}+\kappa\left(\Delta u\cdot\D\theta+2\D u:\D^2\theta-\div u\Delta\theta\right)\nonumber\\
=~&-R\rho\dot{\theta}\div u-R\rho\theta\div\dot{u}+R\rho\theta\partial_iu^j\partial_ju^i+\pt\left(\rho^\al\left(2\mu|\frD(u)|^2+\lambda(\div u)^2\right)\right)\nonumber\\
&+\div\left(\rho^\al u\left(2\mu|\frD(u)|^2+\lambda(\div u)^2\right)\right)\nonumber\\
=~&-R\rho\dot{\theta}\div u-R\rho\theta\div\dot{u}+R\rho\theta\partial_iu^j\partial_ju^i-(\al-1)\rho^\al\div u\left(2\mu|\frD(u)|^2+\lambda(\div u)^2\right)\nonumber\\
&+\mu\rho^\al\left(\partial_iu^j+\partial_ju^i\right)\left(\partial_i\dot{u}^j+\partial_j\dot{u}^i\right)-\mu\rho^\al\left(\partial_iu^j+\partial_ju^i\right)\left(\partial_iu^k\partial_ku^j+\partial_ju^k\partial_ku^i\right)\nonumber\\
&+2\lambda\rho^\al\div u\div\dot{u}-2\lambda\rho^\al\div u\partial_iu^j\partial_ju^i,
\end{align}
where we have used $(\ref{FCNS:2})_1$ and \eqref{eqn:P}.

Making inner product with $\sigma^6\dot{\theta}$ and noticing that
\begin{align}\label{3.5:4}
&\int\dot{\theta}\left(\Delta u\cdot\D\theta+2\D u:\D^2\theta-\div u\Delta\theta\right)\dif x\nonumber\\
=~&\int\dot{\theta}\left(\partial_j^2u^i\partial_i\theta+\partial_iu^j\partial_i\partial_j\theta\right)\dif x+\int\dot{\theta}\left(\partial_ju^i\partial_i\partial_j\theta-\partial_iu^i\partial_j^2\theta\right)\dif x\nonumber\\
=~&-\int\partial_j\dot{\theta}\partial_ju^i\partial_i\theta\dif x-\int\partial_i\dot{\theta}\partial_ju^i\partial_j\theta\dif x-\int\partial_j\dot{\theta}\partial_iu^i\partial_j\theta\dif x-\int\dot{\theta}\left(\partial_i\partial_ju^i\partial_j\theta-\partial_j\partial_iu^i\partial_j\theta\right)\dif x\nonumber\\
=~&-\int\partial_j\dot{\theta}\partial_ju^i\partial_i\theta\dif x-\int\partial_i\dot{\theta}\partial_ju^i\partial_j\theta\dif x-\int\partial_j\dot{\theta}\partial_iu^i\partial_j\theta\dif x,
\end{align}
we obtain
\begin{align}\label{3.5:5}
&\frac{1}{2}\frac{\dif}{\dif t}\left(\sigma^6\int\rho|\dot{\theta}|^2\dif x\right)+\kappa\sigma^6\int|\D\dot{\theta}|^2\dif x\nonumber\\
\leq~&3\sigma^5\sigma'\int\rho|\dot{\theta}|^2\dif x+C\sigma^6\int|\D\dot{\theta}||\D u||\D\theta|\dif x+C\sigma^6\int\rho|\dot{\theta}|^2|\D u|\dif x+C\sigma^6\int\rho|\theta-1||\dot{\theta}|\left(|\D\dot{u}|+|\D u|^2\right)\dif x\nonumber\\
&+C\sigma^6\int\rho|\dot{\theta}|\left(|\D\dot{u}|+|\D u|^2\right)\dif x+C\sigma^6\int\rho^\al|\dot{\theta}||\D u|\left(|\D\dot{u}|+|\D u|^2\right)\dif x\nonumber\\
:=~&3\sigma^5\sigma'\int\rho|\dot{\theta}|^2\dif x+\sum\limits_{i=1}^5P_i.
\end{align}
From \eqref{Lem:4:1}, \eqref{estimate:DuL6:4} and \eqref{lem3.1:g3}, the terms $P_i$ can be estimated as follows:
\begin{align}\label{3.5:6}
P_1=C\sigma^6\int|\D\dot{\theta}||\D u||\D\theta|\dif x\leq~&\frac{\kappa\sigma^6}{8}\|\D\dot{\theta}\|_{L^2}^2+C\sigma^6\|\D u\|_{L^2}\|\D u\|_{L^6}\|\D^2\theta\|_{L^2}^2\nonumber\\
\leq~&\frac{\kappa\sigma^6}{8}\|\D\dot{\theta}\|_{L^2}^2+C(K)\vr^{-\al}\sigma^3\|\D^2\theta\|_{L^2}^2,
\end{align}
\begin{align}\label{3.5:7}
P_2=C\sigma^6\int\rho|\dot{\theta}|^2|\D u|\dif x\leq~&\frac{\kappa\sigma^6}{8}\|\D\dot{\theta}\|_{L^2}^2+C\vr\sigma^6\|\D u\|_{L^2}\|\D u\|_{L^6}\|\sqrt{\rho}\dot{\theta}\|_{L^2}^2\nonumber\\
\leq~&\frac{\kappa\sigma^6}{8}\|\D\dot{\theta}\|_{L^2}^2+C(K)\vr^{1-\al}\sigma^3\|\sqrt{\rho}\dot{\theta}\|_{L^2}^2,
\end{align}
\begin{align}\label{3.5:8}
P_3=~&C\sigma^6\int\rho|\theta-1||\dot{\theta}|\left(|\D\dot{u}|+|\D u|^2\right)\dif x\nonumber\\
\leq~&\frac{\kappa\sigma^6}{8}\|\D\dot{\theta}\|_{L^2}^2+C\vr^2\sigma^6\|\theta-1\|_{L^2}\|\D\theta\|_{L^2}\|\D\dot{u}\|_{L^2}^2+C\vr^2\sigma^6\|\theta-1\|_{L^2}\|\D\theta\|_{L^2}\|\D u\|_{L^2}\|\D u\|_{L^6}^3\nonumber\\
\leq~&\frac{\kappa\sigma^6}{8}\|\D\dot{\theta}\|_{L^2}^2+C\vr^2\sigma^6\left(1+\|\D\theta\|_{L^2}\right)\|\D\theta\|_{L^2}\|\D\dot{u}\|_{L^2}^2+C\vr^2\sigma^6\left(1+\|\D\theta\|_{L^2}\right)\|\D\theta\|_{L^2}\|\D u\|_{L^2}\|\D u\|_{L^6}^3\nonumber\\
\leq~&\frac{\kappa\sigma^6}{8}\|\D\dot{\theta}\|_{L^2}^2+C(K)\vr^3\sigma^3\|\D\dot{u}\|_{L^2}^2+C(K)\vr^{2-\al}\|\D u\|_{L^2}\|\D u\|_{L^6},
\end{align}
\begin{align}\label{3.5:9}
P_4=C\sigma^6\int\rho|\dot{\theta}|\left(|\D\dot{u}|+|\D u|^2\right)\dif x\leq~&C\vr^{\frac{1}{2}}\sigma^6\|\sqrt{\rho}\dot{\theta}\|_{L^2}\|\D\dot{u}\|_{L^2}+C\vr^{\frac{1}{2}}\sigma^6\|\sqrt{\rho}\dot{\theta}\|_{L^2}\|\D u\|_{L^2}^{\frac{1}{2}}\|\D u\|_{L^6}^{\frac{3}{2}}\nonumber\\
\leq~&C\vr^{\frac{1}{2}}\sigma^6\|\sqrt{\rho}\dot{\theta}\|_{L^2}\|\D\dot{u}\|_{L^2}+C(K)\vr^{\frac{1-\al}{2}}\sigma^3\|\sqrt{\rho}\dot{\theta}\|_{L^2}\|\D u\|_{L^6},
\end{align}
\begin{align}\label{3.5:10}
P_5=~&C\sigma^6\int\rho^\al|\dot{\theta}||\D u|\left(|\D\dot{u}|+|\D u|^2\right)\dif x\nonumber\\
\leq~&\frac{\kappa\sigma^6}{8}\|\D\dot{\theta}\|_{L^2}^2+C\vr^{2\al}\sigma^6\|\D u\|_{L^2}\|\D u\|_{L^6}\|\D\dot{u}\|_{L^2}^2+C\vr^{\al-\frac{1}{2}}\sigma^6\|\sqrt{\rho}\dot{\theta}\|_{L^2}\|\D u\|_{L^6}^3\nonumber\\
\leq~&\frac{\kappa\sigma^6}{8}\|\D\dot{\theta}\|_{L^2}^2+CK_6\vr^\al\sigma^3\|\D\dot{u}\|_{L^2}^2+C(K)\vr^{\-\frac{3}{2}}\sigma^3\|\sqrt{\rho}\dot{\theta}\|_{L^2}\|\D u\|_{L^6}.
\end{align}

In order to estimate $\|\D^2\theta\|_{L^2}$ on the right hand side of \eqref{3.5:6}, we apply the standard $L^2$-estimate to the equation $(\ref{FCNS:2})_3$ to obtain
\begin{align}\label{estimate:D2theta}
\|\D^2\theta\|_{L^2}^2 \leq~& C\left(\|\rho\dot{\theta}\|_{L^2}^2+\vr^{2\al}\|\D u\|_{L^4}^4+\vr^2\|\theta \D u\|_{L^2}^2\right)\nonumber\\
\leq~& C\left(\|\rho\dot{\theta}\|_{L^2}^2+\vr^{2\al}\|\D u\|_{L^4}^4+\vr^2\|\theta-1\|_{L^6}^2\|\D u\|_{L^2}\|\D u\|_{L^6}+\vr^2\|\D u\|_{L^2}^2\right)\nonumber\\
\leq~& C\left(\vr\|\sqrt{\rho}\dot{\theta}\|_{L^2}^2+\vr^{2\al}\|\D u\|_{L^2}\|\D u\|_{L^6}^3+\vr^2\|\D\theta\|_{L^2}^2\|\D u\|_{L^2}\|\D u\|_{L^6}+\vr^2\|\D u\|_{L^2}^2\right).
\end{align}
Hence, from \eqref{-4.1:2}, \eqref{Lem:4:1} and \eqref{estimate:DuL6:5}, it yields that
\begin{align}\label{estimate:D2theta:2}
\int_0^T\sigma^3\|\D^2\theta\|_{L^2}^2\dif t\leq~& C\vr\int_0^T\sigma^3\|\sqrt{\rho}\dot{\theta}\|_{L^2}^2\dif t+C\vr^{2\al}\sup\limits_{t\in[0,T]}\left(\sigma^3\|\D u\|_{L^6}^2\right)\int_0^T\|\D u\|_{L^2}\|\D u\|_{L^6}\dif t\nonumber\\
&+C\vr^2\sup\limits_{t\in[0,T]}\left(\sigma^3\|\D\theta\|_{L^2}^2\right)\int_0^T\|\D u\|_{L^2}\|\D u\|_{L^6}\dif t+C\vr^2\int_0^T\|\D u\|_{L^2}^2\dif t\nonumber\\
\leq~&CK_6\vr^2+C(K)\vr^{\frac{1}{2}}\nonumber\\
\leq~& CK_6\vr^2.
\end{align}

Inserting \eqref{3.5:6}--\eqref{3.5:10} and \eqref{estimate:D2theta:2} into \eqref{3.5:5}, after integration with respect to $t\in[0,T]$ and taking \eqref{assume:2}, \eqref{-4.1:2}, \eqref{estimate:DuL6:4}, \eqref{estimate:DuL6:5} into consideration, we have
\begin{align}\label{3.5:12}
&\sigma^6\int\rho|\dot{\theta}|^2\dif x+\kappa\int_0^T\int\sigma^6|\D\dot{\theta}|^2\dif x\dif t\nonumber\\
\leq~&6\int_0^{\sigma(T)}\sigma^3\|\sqrt{\rho}\dot{\theta}\|_{L^2}^2\dif t+C(K)\vr^{-\al}\int_0^T\sigma^3\|\D^2\theta\|_{L^2}^2\dif t+C(K)\vr^{1-\al}\int_0^T\sigma^3\|\sqrt{\rho}\dot{\theta}\|_{L^2}^2\dif t\nonumber\\
&+C(K)\vr^3\int_0^T\sigma^3\|\D\dot{u}\|_{L^2}^2\dif t+C(K)\vr^{2-\al}\int_0^T\|\D u\|_{L^2}\|\D u\|_{L^6}\dif t+C\vr^{\frac{1}{2}}\int_0^T\sigma^6\|\sqrt{\rho}\dot{\theta}\|_{L^2}\|\D\dot{u}\|_{L^2}\dif t\nonumber\\
&+C(K)\left(\vr^{\frac{1-\al}{2}}+\vr^{\frac{3}{2}}\right)\int_0^T\sigma^3\|\sqrt{\rho}\dot{\theta}\|_{L^2}\|\D u\|_{L^6}\dif t+CK_6\vr^\al\int_0^T\sigma^3\|\D\dot{u}\|_{L^2}^2\dif t\nonumber\\
%\leq~&C_7K_6^2\vr^2+C(K)\left(\vr+\vr^{5-\al}+\vr^{\frac{7}{2}-2\al}+\vr^{2-\frac{\al}{2}}\right)\nonumber\\
\leq~&C_7K_6^2\vr^2+C(K)\vr\nonumber\\
\leq~&K_7\vr^2,
\end{align}
provided $\vr\geq\rL_{6,1}$, $\al>4$, and $K_7$, $\rL_{6,1}$ large enough, such that $C_7K_6^2\leq\frac{1}{2}K_7$ and $C(K)\vr\leq\frac{1}{2}K_7\vr^2$.

From \eqref{estimate:D2theta}, we have
\begin{align}\label{-4.7}
&\sup\limits_{t\in[\sigma(T),T]}\|\theta-1\|_{L^\infty}^2
\leq C\sup\limits_{t\in[\sigma(T),T]}\|\D\theta\|_{L^2}\|\D^2\theta\|_{L^2}\nonumber\\
\leq~& C\sup\limits_{t\in[\sigma(T),T]}\Big(\vr^{\frac{1}{2}}\|\D\theta\|_{L^2}\|\sqrt{\rho}\dot{\theta}\|_{L^2}+\vr^\al\|\D u\|_{L^2}^{\frac{1}{2}}\|\D\theta\|_{L^2}\|\D u\|_{L^6}^{\frac{3}{2}}+\vr\|\D u\|_{L^2}^{\frac{1}{2}}\|\D\theta\|_{L^2}^2\|\D u\|_{L^6}^{\frac{1}{2}}\nonumber\\
&+\vr\|\D u\|_{L^2}\|\D\theta\|_{L^2}\Big)\nonumber\\
\leq~&C_8\sqrt{K_6K_7} ~ \vr^2+C(K)\left(\vr^{\frac{3}{8}}+\vr^{2-\frac{\al}{2}}\right)\nonumber\\
\leq~&K_8\vr^2,
\end{align}
provided $\vr\geq\rL_{6,2}$, where we have taken $K_8$ and $\rL_{6,2}$ large enough, such that $C_8\sqrt{K_6K_7}\leq\frac{1}{2}K_8$ and $C(K)\left(\vr^{\frac{3}{8}}+\vr^{2-\frac{\al}{2}}\right)\leq\frac{1}{2}K_8\vr^2$. At last, taking $\rL_6=\max\{\rL_{6,1}, ~ \rL_{6,2}\}$, we have proved Lemma \ref{Lem 3.5}.
\end{proof}

\subsubsection{The first-order uniform estimates of density on $[0,T]$}\label{sec3.3}

Next, we consider the first-order uniform estimates of density on $[0,T]$, which is achieved by combining the estimates on $[0,\sigma(T)]$ and the time-weighted estimate on $[0,T]$. Then, with the estimates of $\|\D\rho\|_{L^2}$ and $\|\D\rho\|_{L^4}$ in hand, the uniform lower and upper bounds of density is obtained by an interpolation argument.

\begin{lemma}\label{Lem39}
Under the conditions of Proposition \ref{p4.1}, there exist positive constants $\rL_7$ and $K_9$ depending on $\mu$, $\lambda$, $\kappa$, and $R$, such that if $(\rho,u,\theta)$ is a strong solution of \eqref{FCNS:2}--\eqref{data and far} on $\R^3\times(0,T]$, it holds that
%If $\vr$ is large enough,
%\begin{align}
%2\al+\delta\geq 1+\gamma, \gamma \geq \max\left\{\frac{3}{2}-\frac{11}{4}\delta,\frac{1}{4}+\frac{9}{28}\delta\right\},
%2\gamma-1-\frac{7}{4}\delta\leq\al\leq\gamma-\frac{3}{8}\delta,
%\end{align}
%and $A_{3}(0)\leq\vr^{3+\gamma-2\al}$. Assume $A_{3}(T)\leq 3\vr^{3+\gamma-2\al}$, then
\begin{equation}\label{3.9:1}
A_7(T)=\sup_{t\in[0,T]}\|\D \rho\|_{L^2}^2+\vr^{1-\al}\int_0^T\|\D \rho\|_{L^2}^2\dif t\leq K_9\vr^2,
\end{equation}
\begin{equation}\label{3.9:2}
A_8(T)=\sup_{t\in[0,T]}\|\D \rho\|_{L^4}^2+\vr^{1-\al}\int_0^T\|\D \rho\|_{L^4}^2\dif t\leq\vr^{\frac{3}{2}},
\end{equation}
and
\begin{equation}\label{3.9:3}
\frac{2}{3}\vr\leq\rho(x,t)\leq\frac{4}{3}\vr,
\end{equation}
provided that $\vr\geq\rL_7$.
\end{lemma}

\begin{proof}
Applying $\nabla$ operator on $(\ref{FCNS:2})_1$ and multiplying $\nabla\rho|\nabla \rho|^{r-2}$ on both sides of the resulting equation, we obtain
%Thanks to Lemma 3.5 of \cite{MR4589926}, we have
\begin{align}
\frac{1}{r}\left(|\D\rho|^r\right)_t&+\frac{1}{r}\div(|\D \rho|^{r}u)+\frac{r-1}{r}|\D \rho|^r\div u+|\D \rho|^{r-2}(\D\rho)^\top\D u\D \rho\nonumber\\
&+|\D \rho|^{r-2}\D \rho\cdot\left(\frac{\rho}{2\mu+\lambda}\D G+R\rho^{1-\al}\theta\D \rho+\frac{R}{(1-\al)}\rho^{2-\al}\D\theta\right)=0.
\end{align}
Integrating the above equation in $\R^3$, we obtain
\begin{align}\label{lr}
%&\frac{\dif}{\dif t}{r}+C^{-1}\vr^{1-\al}\int|\D\rho|^{r}\dif x\nonumber\\
\frac{1}{r}\frac{\dif}{\dif t}\|\D\rho\|_{L^r}^{r}+R\vr^{1-\al}\int\theta|\D\rho|^{r}\dif x
\leq~& C\|\D u\|_{L^\infty}\|\D \rho\|_{L^r}^{r}+C\vr\|\D G\|_{L^r}\|\D\rho\|_{L^r}^{r-1}+C\vr^{2-\al}\|\D \theta\|_{L^r}\|\D\rho\|_{L^r}^{r-1}.
\end{align}

\textbf{Case 1}: When $r=2$, \eqref{lr} reduces to 
\begin{align}\label{l2}
&\frac{1}{2}\frac{\dif}{\dif t}\|\D\rho\|_{L^2}^2+R\vr^{1-\al}\int|\D\rho|^{2}\dif x\nonumber\\
\leq~& C\vr^{1-\al}\|1-\theta\|_{L^\infty}\|\D \rho\|_{L^2}^2+C\|\D u\|_{L^\infty}\|\D \rho\|_{L^2}^2+C\vr\|\D G\|_{L^2}\|\D\rho\|_{L^2}+C\vr^{1-\al}\|\D \theta\|_{L^2}\|\D\rho\|_{L^2}\nonumber\\
\leq~&\frac{R}{2}\vr^{1-\al}\|\D\rho\|_{L^2}^2+C\left(\vr^{1-\al}\|1-\theta\|_{L^\infty}+\|\D u\|_{L^\infty}\right)\|\D\rho\|_{L^2}^2+C\vr\|\D G\|_{L^2}\|\D\rho\|_{L^2}+C\vr^{1-\al}\|\D \theta\|_{L^2}^2\nonumber\\
\leq~&\frac{R}{2}\vr^{1-\al}\|\D\rho\|_{L^2}^2+C\left(\vr^{1-\al}\|\D\theta\|_{L^2}^{\frac{1}{2}}\|\D^2\theta\|_{L^2}^{\frac{1}{2}}+\|\D u\|_{L^2}^\frac{1}{7}\|\D^2 u\|_{L^4}^\frac{6}{7}\right)\|\D\rho\|_{L^2}^2+C\vr\|\D G\|_{L^2}\|\D\rho\|_{L^2}\nonumber\\
&+C\vr^{1-\al}\|\D \theta\|_{L^2}^2.
\end{align}
By \eqref{2.3:2}, we have
\begin{align}\label{f2}
\|\D G\|_{L^2}\leq C\|H\|_{L^2}\leq~&C\vr^{\frac{1}{2}-\al}\|\sqrt{\rho}\dot{u}\|_{L^2}+C\vr^{1-\al}\|\D\theta\|_{L^2}+C\vr^{-1}\|\D\rho\|_{L^2}\|\D u\|_{L^\infty}
\nonumber\\
\leq~&C\vr^{\frac{1}{2}-\al}\|\sqrt{\rho}\dot{u}\|_{L^2}+C\vr^{1-\al}\|\D\theta\|_{L^2}+C\vr^{-1}\|\D\rho\|_{L^2}\|\D u\|_{L^2}^{\frac{1}{7}}\|\D^2u\|_{L^4}^{\frac{6}{7}}.
\end{align}

Inserting \eqref{f2} into \eqref{l2}, we have
\begin{align}\label{l2-1}
&\frac{\dif}{\dif t}\|\D\rho\|_{L^2}^2+R\vr^{1-\al}\|\D\rho\|_{L^2}^2\nonumber\\
\leq~&C\left(\vr^{1-\al}\|\D\theta\|_{L^2}^{\frac{1}{2}}\|\D^2\theta\|_{L^2}^{\frac{1}{2}}+\|\D u\|_{L^2}^\frac{1}{7}\|\D^2 u\|_{L^{4}}^\frac{6}{7}\right)\|\D\rho\|_{L^2}^2+C\vr^{1-\al}\|\D \theta\|_{L^2}^2
\nonumber\\
&+C\vr\|\D\rho\|_{L^2}\left(\vr^{\frac{1}{2}-\al}\|\sqrt{\rho}\dot{u}\|_{L^{2}}
+\vr^{1-\al}\|\D\theta\|_{L^{2}}+\vr^{-1}\|\D\rho\|_{L^{2}}\|\D u\|_{L^2}^\frac{1}{7}\|\D^2 u\|_{L^{4}}^\frac{6}{7}\right)\nonumber\\
\leq~&\frac{R}{4}\vr^{1-\al}\|\D\rho\|_{L^2}^2+C\left(\vr^{1-\al}\|\D\theta\|_{L^2}^{\frac{1}{2}}\|\D^2\theta\|_{L^2}^{\frac{1}{2}}+\|\D u\|_{L^2}^\frac{1}{7}\|\D^2 u\|_{L^{4}}^\frac{6}{7}\right)\|\D\rho\|_{L^2}^2\nonumber\\
&+C\vr^{2-\al}\|\sqrt{\rho}\dot{u}\|_{L^{2}}^2+C\vr^{3-\al}\|\D\theta\|_{L^{2}}^2.
\end{align}
It remains to estimate $\|\D^2 u\|_{L^4}$  on the right hand side of \eqref{l2-1}. From $(\ref{FCNS:2})_2$, one has
\begin{align}\label{D2u4}
\|\D^2u\|_{L^4}\leq~&C\vr^{-\al}
\left(\vr\|\dot{u}\|_{L^4}+\vr^{\al-1}\|\D\rho\D u\|_{L^4}+\|\D\rho\|_{L^4}\|\theta\|_{L^\infty}+\vr\|\D\theta\|_{L^4}\right)
\nonumber\\
\leq~&C\vr^{-\al}
\left(\vr\|\dot{u}\|_{L^4}+\vr^{\al-1}\|\D\rho\|_{L^4}\|\D u\|_{L^\infty}+\|\D\rho\|_{L^4}\|\D\theta\|_{L^2}^{\frac{1}{2}}\|\D^2\theta\|_{L^2}^{\frac{1}{2}}+\vr\|\D\theta\|_{L^2}^{\frac{1}{4}}\|\D^2\theta\|_{L^2}^{\frac{3}{4}}\right)
\nonumber\\
%\leq~&C\vr^{-\al}
%\left(\vr\|\dot{u}\|_{L^4}+\vr^{\al-1}\|\D\rho\|_{L^4}\|\D u\|_{L^2}^{\frac{1}{7}}
%\|\D^2u\|_{L^4}^{\frac{6}{7}}+\|\D\rho\|_{L^4}\|\D\theta\|_{L^2}^{\frac{1}{2}}\|\D^2\theta\|_{L^2}^{\frac{1}{2}}+\vr\|\D\theta\|_{L^2}^{\frac{1}{4}}\|\D^2\theta\|_{L^2}^{\frac{3}{4}}\right)
%\nonumber\\
\leq~&C\vr^{1-\al}\|\dot{u}\|_{L^4}+C\vr^{-1}\|\D\rho\|_{L^4}\|\D u\|_{L^2}^{\frac{1}{7}}
\|\D^2u\|_{L^4}^{\frac{6}{7}}
+C\vr^{-\al}\|\D\rho\|_{L^4}\|\D\theta\|_{L^2}^{\frac{1}{2}}\|\D^2\theta\|_{L^2}^{\frac{1}{2}}\nonumber\\
&+C\vr^{1-\al}\|\D\theta\|_{L^2}^{\frac{1}{4}}\|\D^2\theta\|_{L^2}^{\frac{3}{4}}
\nonumber\\
\leq~&C\vr^{1-\al}\|\dot{u}\|_{L^4}+C\vr^{-7}\|\D\rho\|_{L^4}^7\|\D u\|_{L^2}
+C\vr^{-\al}\|\D\rho\|_{L^4}\|\D\theta\|_{L^2}^{\frac{1}{2}}\|\D^2\theta\|_{L^2}^{\frac{1}{2}}\nonumber\\
&+C\vr^{1-\al}\|\D\theta\|_{L^2}^{\frac{1}{4}}\|\D^2\theta\|_{L^2}^{\frac{3}{4}}
+\frac{1}{4}\|\D^2u\|_{L^4},
\end{align}
which implies
\begin{align}\label{D2l4}
\|\D^2u\|_{L^4}
\leq C\vr^{1-\al}\|\dot{u}\|_{L^4}
+C\vr^{-7}\|\D\rho\|_{L^4}^7\|\D u\|_{L^2}+C\vr^{-\al}\|\D\rho\|_{L^4}\|\D\theta\|_{L^2}^{\frac{1}{2}}\|\D^2\theta\|_{L^2}^{\frac{1}{2}}+C\vr^{1-\al}\|\D\theta\|_{L^2}^{\frac{1}{4}}\|\D^2\theta\|_{L^2}^{\frac{3}{4}}.
\end{align}

Integrating \eqref{D2l4} with respect to $t$ over $[0,\sigma(T)]$,
we arrive at
\begin{align}\label{p0}
\int_0^{\sigma(T)}\|\D^2u\|_{L^4}^2\dif t
\leq~&C\vr^{2-2\al}\int_{0}^{\sigma(T)}\|\dot{u}\|_{L^4}^2\dif t+C\vr^{-14}\int_{0}^{\sigma(T)}\|\D\rho \|_{L^4}^{14}
\|\D u\|_{L^2}^{2}\dif t\nonumber\\
&
+C\vr^{-2\al}\int_0^{\sigma(T)}\|\D\rho\|^2_{L^4}\|\D\theta\|_{L^2}\|\D^2\theta\|_{L^2}\dif t+C\vr^{2-2\al}\int_0^{\sigma(T)}\|\D\theta\|_{L^2}^{\frac{1}{2}}\|\D^2\theta\|_{L^2}^{\frac{3}{2}}\dif t\nonumber\\
:=~&\sum\limits_{i=1}^{4}Q_i.
\end{align}
Next, we estimate $P_i$ on the right hand side of \eqref{p0}. Firstly, from \eqref{estimate:D2theta}, we have
\begin{align}\label{estimate:D2theta:3}
\int_0^{\sigma(T)}\|\D^2 \theta\|_{L^2}^2\dif t\leq~
&C\vr\int_0^{\sigma(T)}\|\sqrt{\rho}\dot{\theta}\|_{L^2}^2\dif t+C\vr^{\frac{3}{2}-\al}\int_0^{\sigma(T)}\|\D u\|_{L^2}\|\sqrt{\rho}\dot{u}\|_{L^2}^3\dif t\nonumber\\
&+C\vr^{3-\al}\int_0^{\sigma(T)}\|\D u\|_{L^2}\|\D\theta\|_{L^2}^3\dif t+C\vr^{-\al}\int_0^{\sigma(T)}\|\D u\|_{L^2}\|\D\rho\|_{L^2}^3\dif t\nonumber\\
&+C\vr^4\int_0^{\sigma(T)}\|\D u\|_{L^2}^4\|\D\theta \|_{L^2}^2\dif t\nonumber\\
\leq~&C\left(K_2+K_1\sqrt{K_1K_3}\right)\vr^{\al+1}+C(K)\left(\vr^{4-\frac{\al}{2}}+\vr^{\frac{1}{2}(7-3\al)}+\vr^5\right)\nonumber\\
\leq~&C(K)\vr^{\al+1}.
\end{align}
Then, by \eqref{assume:1}, \eqref{assume:3} and \eqref{estimate:D2theta:3}, we have
\begin{align}\label{p1}
Q_1=C\vr^{2-2\al}\int_0^{\sigma(T)}\|\dot{u}\|_{L^4}^2\dif t\leq C\vr^{\frac{7}{4}-2\al}\int_0^{\sigma(T)}
\|\sqrt{\rho}\dot{u}\|_{L^2}^{\frac{1}{2}}\|\D \dot{u}\|_{L^2}^{\frac{3}{2}}\dif t\leq C(K)\vr^{1-\al}.
%\leq~&C\vr^{\frac{7}{4}-2\al}
%\left(\int_{0}^{\sigma(T)}\|\sqrt{\rho}\dot{u}\|^{2}_{L^2}\dif t\right)^{\frac{1}{4}}
%\left(\int_{0}^{\sigma(T)}\|\D \dot{u}\|^{2}_{L^2}\dif t\right)^{\frac{3}{4}}\nonumber\\
%\leq~& C(K)\vr^{\frac{7}{4}-2\al}\vr^{\frac{\al}{4}}\vr^{\frac{3}{4}(\al-1)}\nonumber\\
\end{align}
\begin{align}\label{p3}
Q_2=C\vr^{-14}\int_{0}^{\sigma(T)}\|\D\rho \|_{L^4}^{14}\|\D u\|_{L^2}^{2}\dif t\leq C\vr^{-14}\sup\limits_{t\in[0,\sigma(T)]}\|\D \rho\|_{L^4}^{14}\int_0^{\sigma(T)}\|\D u\|_{L^2}^2\dif t\leq C(K)\vr^{-\frac{5}{2}-\al}.
\end{align}
\begin{align}\label{p4}
Q_3=~&C\vr^{-2\al}\int_0^{\sigma(T)}\|\D\rho\|^{2}_{L^4}\|\D\theta\|_{L^2}\|\D^2\theta\|_{L^2}\dif t
\nonumber\\ 
\leq~&C\vr^{-2\al}\sup\limits_{t\in[0,\sigma(T)]}\|\D\rho\|^{2}_{L^4}\int_0^{\sigma(T)}\|\D\theta\|_{L^2}\|\D^2\theta\|_{L^2}\dif t\nonumber\\
\leq~& C(K)\vr^{\frac{5-3\al}{2}},
\end{align}
\begin{align}\label{p5}
Q_4=C\vr^{2-2\al}\int_0^{\sigma(T)}\|\D\theta\|_{L^2}^{\frac{1}{2}}\|\D^2\theta\|_{L^2}^{\frac{3}{2}}\dif t\leq C(K)\vr^{3-\frac{5}{4}\al}.
%\leq~&C\vr^{2-2\al}\left(\int_0^{\sigma(T)}\|\D\theta\|_{L^2}^2\dif t\right)^\frac{1}{4}\left(\int_0^{\sigma(T)}\|\D^2\theta\|_{L^2}^2\dif t\right)^\frac{3}{4}\nonumber\\
\end{align}
Now substituting \eqref{p1}--\eqref{p5} into \eqref{p0}, one can see that
\begin{align}\label{D2u24}
\int_0^{\sigma(T)}\|\D^2u\|_{L^4}^2\dif t\leq C(K)\left(\vr^{1-\al}+\vr^{3-\frac{5}{4}\al}\right).
\end{align}

Next, we estimate $\|\D\rho\|_{L^2}$ in two cases.

\textbf{Case 1.1}: For the case  $t\in(0,\sigma(T))$, integrating \eqref{l2-1} over $(0,\sigma(T))$, we have
\begin{align}\label{l2-2}
&\|\D\rho\|_{L^2}^2+\frac{R}{4}\vr^{1-\al}\int_0^{\sigma(T)}\|\D\rho\|_{L^2}^2\dif t\nonumber\\
\leq~& \|\D\rho_0\|_{L^2}^2+C\int_0^{\sigma(T)}\left(\vr^{1-\al}\|\D\theta\|_{L^2}^{\frac{1}{2}}\|\D^2\theta\|_{L^2}^{\frac{1}{2}}+\|\D u\|_{L^2}^\frac{1}{7}\|\D^2 u\|_{L^{4}}^\frac{6}{7}\right)\|\D\rho\|_{L^2}^2\dif t\nonumber\\
&+C\vr^{2-\al}\int_0^{\sigma(T)}\|\sqrt{\rho}\dot{u}\|_{L^{2}}^2\dif t+C\vr^{3-\al}\int_0^{\sigma(T)}\|\D\theta\|_{L^{2}}^2\dif t
\nonumber\\
%\leq~& \|\D\rho_0\|_{L^2}^2+C\vr^{2+\al}\left(\int_0^{\sigma(T)}\|\D u\|_{L^2}^2\dif t\right)^\frac{1}{7}\left(\int_0^{\sigma(T)}\|\D^2 u\|_{L^{4}}^2\dif t\right)^\frac{6}{7}+C\vr^{2}+C\vr^{4-\al}\nonumber\\
\leq~&\|\D\rho_0\|_{L^2}^2+CK_1\vr^{2}+C(K)\vr^{4-\al}\nonumber\\
&+C\int_0^{\sigma(T)}\left(\vr^{1-\al}\|\D\theta\|_{L^2}^{\frac{1}{2}}\|\D^2\theta\|_{L^2}^{\frac{1}{2}}+\|\D u\|_{L^2}^\frac{1}{7}\|\D^2 u\|_{L^{4}}^\frac{6}{7}\right)\|\D\rho\|_{L^2}^2\dif t,
\end{align}
where we have used \eqref{assume:1}.

By Gronwall's inequality, we get
\begin{align}\label{l2-2-1}
&\|\D\rho\|_{L^2}^2+\vr^{1-\al}\int_0^{\sigma(T)}\|\D\rho\|_{L^2}^2\dif t\nonumber\\
\leq~&\left(C\|\D\rho_0\|_{L^2}^2+CK_1\vr^{2}+C(K)\vr^{4-\al}\right)\exp\Bigg\{\vr^{1-\al}\left(\int_0^{\sigma(T)}\|\D\theta\|_{L^2}^2\dif t\right)^\frac{1}{4}\left(\int_0^{\sigma(T)}\|\D^2\theta\|_{L^2}^2\dif t\right)^\frac{1}{4}\nonumber\\
&+\left(\int_0^{\sigma(T)}\|\D u\|_{L^2}^2\dif t\right)^\frac{1}{14}\left(\int_0^{\sigma(T)}\|\D^2 u\|_{L^4}^2\dif t\right)^\frac{3}{7}\Bigg\}\cdot\nonumber\\
\leq~&2C_9K_1\vr^2\exp\left\{C(K)\left(\vr^{\frac{3}{4}(2-\al)}+\vr^{\frac{1-\al}{2}}+\vr^{\frac{19}{14}-\frac{17}{28}\alpha}\right)\right\}\nonumber\\
\leq~&4C_9K_1\vr^2,
\end{align}
provided $\vr\geq\rL_{7,1}$, where we have used \eqref{compatibility condition}, \eqref{assume:1}, \eqref{estimate:D2theta:3}, \eqref{D2u24}, and taken $\rL_{7,1}$ large enough such that $C(K)\left(\vr^{\frac{3}{4}(2-\al)}+\vr^{\frac{1-\al}{2}}+\vr^{\frac{19}{14}-\frac{17}{28}\alpha}\right)\leq\ln 2$.

\bigskip

\textbf{Case 1.2}: For the time-weighted estimates of $\|\D\rho\|_{L^2}$ on the time interval $(0,T)$, multiplying \eqref{lr} by $\sigma^m$ with some constant $m\in\mathbb{N}_+$, we have
\begin{align}\label{lr:sigma}
&\frac{1}{r}\frac{\dif}{\dif t}\left(\sigma^m\|\D\rho\|_{L^r}^{r}\right)+R\sigma^m\vr^{1-\al}\int\theta|\D\rho|^{r}\dif x\nonumber\\
\leq~&m\sigma^{m-1}\sigma'\|\D\rho\|_{L^r}^r+ C\sigma^m\|\D u\|_{L^\infty}\|\D \rho\|_{L^r}^{r}+C\vr\sigma^m\|\D G\|_{L^r}\|\D\rho\|_{L^r}^{r-1}+C\vr^{2-\al}\sigma^m\|\D \theta\|_{L^r}\|\D\rho\|_{L^r}^{r-1}.
\end{align}

By \eqref{D2l4}, \eqref{assume:2}, and \eqref{assume:3}, we have
\begin{align}\label{D2u24:sigma}
&\int_0^T\sigma^3\|\D^2u\|_{L^4}^2\dif t\nonumber\\
\leq~&C\vr^{\frac{7}{4}-2\al}\left(\int_0^T\sigma^2\|\sqrt{\rho}\dot{u}\|_{L^2}^2\dif t\right)^{\frac{1}{4}}\left(\int_0^T\sigma^3\|\D\dot{u}\|_{L^2}^2\dif t\right)^{\frac{3}{4}}+C\vr^{-14}\sup\limits_{t\in[0,T]}\|\D\rho\|_{L^4}^{14}\int_0^T\|\D u\|_{L^2}^2\dif t\nonumber\\
&+C\vr^{-2\al}\sup\limits_{t\in[0,T]}\|\D\rho\|_{L^4}^2\left(\int_0^T\|\D\theta\|_{L^2}^2\dif t\right)^{\frac{1}{2}}\left(\int_0^T\sigma^3\|\D^2\theta\|_{L^2}^2\dif t\right)^{\frac{1}{2}}\nonumber\\
&+C\vr^{2-2\al}\left(\int_0^T\|\D\theta\|_{L^2}^2\dif t\right)^{\frac{1}{4}}\left(\int_0^T\sigma^3\|\D^2\theta\|_{L^2}^2\dif t\right)^{\frac{3}{4}}\nonumber\\
%\leq~&C(K)\left(\vr^{\frac{7}{2}-\frac{11}{4}\al}+\vr^{-\frac{3}{2}-\al}+\vr^{4-2\al}+\vr^{\frac{13}{4}-2\al}\right)\nonumber\\
\leq~&C(K)\left(\vr^{-\frac{3}{2}-\al}+\vr^{4-2\al}\right).
\end{align}

When $r=2$, \eqref{lr:sigma} reduces to 
\begin{align}\label{Dr2:1}
&\frac{1}{2}\frac{\dif}{\dif t}\left(\sigma^m\|\D\rho\|_{L^2}^2\right)+R\sigma^m\vr^{1-\al}\|\D\rho\|_{L^2}^2\nonumber\\
\leq~&m\sigma^{m-1}\sigma'\|\D\rho\|_{L^2}^2+C\vr^{1-\al}\sigma^m\|\theta-1\|_{L^\infty}\|\D\rho\|_{L^2}^2+C\sigma^m\|\D u\|_{L^\infty}\|\D \rho\|_{L^2}^2\nonumber\\
&+C\vr\sigma^m\|\D G\|_{L^2}\|\D\rho\|_{L^2}+C\vr^{2-\al}\sigma^m\|\D\theta\|_{L^2}\|\D\rho\|_{L^2}\nonumber\\
\leq~&m\sigma^{m-1}\sigma'\|\D\rho\|_{L^2}^2+C\vr^{1-\al}\sigma^m\|\D\theta\|_{L^2}^{\frac{1}{2}}\|\D^2\theta\|_{L^2}^{\frac{1}{2}}\|\D\rho\|_{L^2}^2+C\sigma^m\|\D u\|_{L^2}^{\frac{1}{7}}\|\D^2u\|_{L^4}^{\frac{6}{7}}\|\D\rho\|_{L^2}^2\nonumber\\
&+C\vr\sigma^m\|\D G\|_{L^2}\|\D\rho\|_{L^2}+C\vr^{2-\al}\sigma^m\|\D\theta\|_{L^2}\|\D\rho\|_{L^2}\nonumber\\
\leq~&m\sigma^{m-1}\sigma'\|\D\rho\|_{L^2}^2+C\vr^{1-\al}\sigma^m\|\D\theta\|_{L^2}^{\frac{1}{2}}\|\D^2\theta\|_{L^2}^{\frac{1}{2}}\|\D\rho\|_{L^2}^2+C\sigma^m\|\D u\|_{L^2}^{\frac{1}{7}}\|\D^2u\|_{L^4}^{\frac{6}{7}}\|\D\rho\|_{L^2}^2\nonumber\\
&+C\vr^{\frac{3}{2}-\al}\sigma^m\|\sqrt{\rho}\dot{u}\|_{L^2}\|\D\rho\|_{L^2}+C\vr^{2-\al}\sigma^m\|\D\theta\|_{L^2}\|\D\rho\|_{L^2},
\end{align}
where we have used \eqref{f2}.

Taking $m=2$ and integrating \eqref{Dr2:1} over $[0,T]$, we obtain
\begin{align}\label{Dr2:2}
&\sigma^2\|\D\rho\|_{L^2}^2+\vr^{1-\al}\int_0^T\sigma^2\|\D\rho\|_{L^2}^2\dif t\nonumber\\
\leq~&C\sup\limits_{t\in[0,\sigma(T)]}\|\D\rho\|_{L^2}^2+C\vr^{1-\al}\sup\limits_{t\in[0,T]}\|\D\rho\|_{L^2}\left(\int_0^T\|\D\theta\|_{L^2}^2\dif t\right)^{\frac{1}{4}}\left(\int_0^T\sigma^3\|\D^2\theta\|_{L^2}^2\dif t\right)^{\frac{1}{4}}\left(\int_0^T\|\D\rho\|_{L^2}^2\dif t\right)^{\frac{1}{2}}\nonumber\\
&+C\sup\limits_{t\in[0,T]}\|\D\rho\|_{L^2}\left(\int_0^T\|\D u\|_{L^2}^2\dif t\right)^{\frac{1}{14}}\left(\int_0^T\sigma^3\|\D^2u\|_{L^4}^2\dif t\right)^{\frac{3}{7}}\left(\int_0^T\|\D\rho\|_{L^2}^2\dif t\right)^{\frac{1}{2}}\nonumber\\
&+C\vr^{\frac{3}{2}-\al}\left(\int_0^T\sigma^2\|\sqrt{\rho}\dot{u}\|_{L^2}^2\dif t\right)^{\frac{1}{2}}\left(\int_0^T\|\D\rho\|_{L^2}^2\dif t\right)^{\frac{1}{2}}+C\vr^{2-\al}\int_0^T\|\D\theta\|_{L^2}\|\D\rho\|_{L^2}\dif t\nonumber\\
%\leq~&C_{10}K_1\vr^2+C(K)\left(\vr^{\frac{15-2\al}{4}}+\vr+\vr^{\frac{47-6\al}{14}}+\vr^{\frac{5-\al}{2}}+\vr^{4-\frac{\al}{2}}\right)\nonumber\\
\leq~&C_{10}K_1\vr^2+C(K)\left(\vr+\vr^{\frac{47-6\al}{14}}+\vr^{4-\frac{\al}{2}}\right)\nonumber\\
\leq~&K_9\vr^2,
\end{align}
provided $\vr\geq\rL_{7,2}$, where we have used \eqref{l2-2-1}, \eqref{assume:2}, \eqref{assume:3}, \eqref{estimate:D2theta:2}, \eqref{D2u24:sigma}, and taken $K_9$, $\rL_{7,2}$ large enough such that $4C_9K_1+C_{10}K_1\leq\frac{1}{2}K_9$ and $C(K)\left(\vr+\vr^{\frac{47-6\al}{14}}+\vr^{4-\frac{\al}{2}}\right)\leq\frac{1}{2}K_9\vr^2$. Then, combining \eqref{l2-2-1} and \eqref{Dr2:2}, and taking $\vr\geq\max\{\rL_{7,1}, \rL_{7,2}\}$, we have proved \eqref{3.9:1}.

\bigskip

\textbf{Case 2}: When $r=4$, \eqref{lr} reduces to 
\begin{align}\label{l4}
&\frac{\dif}{\dif t}\|\D\rho\|_{L^4}^2+\frac{R}{2}\vr^{1-\al}\|\D\rho\|_{L^4}^2\nonumber\\
\leq~&C\left(\vr^{1-\al}\|1-\theta\|_{L^\infty}+\|\D u\|_{L^\infty}\right)\|\D \rho\|_{L^4}^2+C\vr\|\D G\|_{L^4}\|\D\rho\|_{L^4}+C\vr^{1-\al}\|\D\theta\|_{L^4}\|\D\rho\|_{L^4}.
\end{align}
By \eqref{2.3:2}, we have
\begin{align}\label{DG4}
\|\D G\|_{L^4}^2\leq~& C\vr^{-2\al}\left(\|\rho\dot{u}\|_{L^4}^2+\vr^2\|\D\theta\|_{L^4}^2+\vr^{2\al-2}\|\D u\|_{L^\infty}^2\|\D\rho\|_{L^4}^2\right)\nonumber\\
\leq~&C\vr^{\frac{5}{4}-2\al}\|\sqrt{\rho}\dot{u}\|_{L^2}^{\frac{3}{2}}\|\D\dot{u}\|_{L^2}^{\frac{1}{2}}+C\vr^{2-2\al}\|\D\theta\|_{L^4}^2+C\vr^{-2}\|\D u\|_{L^\infty}^2\|\D\rho\|_{L^4}^2.
\end{align}
Substituting \eqref{DG4} into \eqref{l4}, we obtain
\begin{align}\label{l4:2}
&\frac{\dif}{\dif t}\|\D\rho\|_{L^4}^2+\frac{R}{2}\vr^{1-\al}\|\D\rho\|_{L^4}^2\nonumber\\
\leq~&C\left(\vr^{1-\al}\|1-\theta\|_{L^\infty}+\|\D u\|_{L^\infty}\right)\|\D \rho\|_{L^4}^2+C\vr^{\frac{13}{8}-\al}\|\sqrt{\rho}\dot{u}\|_{L^2}^{\frac{3}{4}}\|\D\dot{u}\|_{L^2}^{\frac{1}{4}}\|\D\rho\|_{L^4}+C\vr^{2-\al}\|\D\theta\|_{L^4}\|\D\rho\|_{L^4}\nonumber\\
\leq~&C\left(\vr^{1-\al}\|\D\theta\|_{L^2}^{\frac{1}{2}}\|\D^2\theta\|_{L^2}^{\frac{1}{2}}+\|\D u\|_{L^2}^\frac{1}{7}\|\D^2 u\|_{L^4}^\frac{6}{7}\right)\|\D \rho\|_{L^4}^2+C\vr^{\frac{13}{8}-\al}\|\sqrt{\rho}\dot{u}\|_{L^2}^{\frac{3}{4}}\|\D\dot{u}\|_{L^2}^{\frac{1}{4}}\|\D\rho\|_{L^4}\nonumber\\
&+C\vr^{2-\al}\|\D\theta\|_{L^2}^{\frac{1}{4}}\|\D^2\theta\|_{L^2}^{\frac{3}{4}}\|\D\rho\|_{L^4}.
\end{align}

Then, similar as the estimates of  $\|\D\rho\|_{L^2}$, we divide the proof of \eqref{3.9:2} into two cases.

\textbf{Case 2.1}: For the case  $t\in(0,\sigma(T))$, integrating \eqref{l4:2} over $(0,\sigma(T))$, we have
\begin{align}\label{l4:3}
&\|\D\rho\|_{L^4}^2+\frac{R}{2}\vr^{1-\al}\int_0^{\sigma(T)}\|\D\rho\|_{L^4}^2\dif t\nonumber\\
\leq~& \|\D\rho_0\|_{4}^2+C\int_0^{\sigma(T)}\left(\vr^{1-\al}\|\D\theta\|_{L^2}^{\frac{1}{2}}\|\D^2\theta\|_{L^2}^{\frac{1}{2}}+\|\D u\|_{L^2}^\frac{1}{7}\|\D^2 u\|_{L^4}^\frac{6}{7}\right)\|\D\rho\|_{L^2}^2\dif t\nonumber\\
&+C\vr^{\frac{13}{8}-\al}\int_0^{\sigma(T)}\|\sqrt{\rho}\dot{u}\|_{L^2}^{\frac{3}{4}}\|\D\dot{u}\|_{L^2}^{\frac{1}{4}}\|\D\rho\|_{L^4}\dif t+C\vr^{2-\al}\int_0^{\sigma(T)}\|\D\theta\|_{L^2}^{\frac{1}{4}}\|\D^2\theta\|_{L^2}^{\frac{3}{4}}\|\D\rho\|_{L^4}\dif t\nonumber\\
\leq~&\|\D\rho_0\|_{4}^2+C\vr^{\frac{13}{8}-\al}\sup\limits_{t\in[0,\sigma(T)]}\|\D\rho\|_{L^4}\left(\int_0^{\sigma(T)}\|\sqrt{\rho}\dot{u}\|_{L^2}^2\dif t\right)^{\frac{3}{8}}\left(\int_0^{\sigma(T)}\|\D\dot{u}\|_{L^2}^2\dif t\right)^{\frac{1}{8}}\nonumber\\
&+C\vr^{2-\al}\sup\limits_{t\in[0,\sigma(T)]}\|\D\rho\|_{L^4}\left(\int_0^{\sigma(T)}\|\D\theta\|_{L^2}^2\dif t\right)^{\frac{1}{8}}\left(\int_0^{\sigma(T)}\|\D^2\theta\|_{L^2}^2\dif t\right)^{\frac{3}{8}}\nonumber\\
&+C\int_0^{\sigma(T)}\left(\vr^{1-\al}\|\D\theta\|_{L^2}^{\frac{1}{2}}\|\D^2\theta\|_{L^2}^{\frac{1}{2}}+\|\D u\|_{L^2}^\frac{1}{7}\|\D^2 u\|_{L^4}^\frac{6}{7}\right)\|\D\rho\|_{L^2}^2\dif t\nonumber\\
\leq~&\|\D\rho_0\|_{L^4}^2+C(K)\left(\vr^{\frac{9}{4}-\frac{\al}{2}}+\vr^{\frac{13}{4}-\frac{5}{8}\al}\right)+C\int_0^{\sigma(T)}\left(\vr^{1-\al}\|\D\theta\|_{L^2}^{\frac{1}{2}}\|\D^2\theta\|_{L^2}^{\frac{1}{2}}+\|\D u\|_{L^2}^\frac{1}{7}\|\D^2 u\|_{L^4}^\frac{6}{7}\right)\|\D\rho\|_{L^2}^2\dif t,
\end{align}
where we have used \eqref{assume:1} and \eqref{estimate:D2theta:3}.

By Gronwall's inequality, we obtain
\begin{align}\label{l4:4}
&\|\D\rho\|_{L^4}^2+\vr^{1-\al}\int_0^{\sigma(T)}\|\D\rho\|_{L^4}^2\dif t\nonumber\\
\leq~&C(K)\left(\|\D\rho_0\|_{L^4}^2+\vr^{\frac{9}{4}-\frac{\al}{2}}+\vr^{\frac{13}{4}-\frac{5}{8}\al}\right)\exp\Bigg\{\vr^{1-\al}\left(\int_0^{\sigma(T)}\|\D\theta\|_{L^2}^2\dif t\right)^\frac{1}{4}\left(\int_0^{\sigma(T)}\|\D^2\theta\|_{L^2}^2\dif t\right)^\frac{1}{4}\nonumber\\
&+\left(\int_0^{\sigma(T)}\|\D u\|_{L^2}^2\dif t\right)^\frac{1}{14}\left(\int_0^{\sigma(T)}\|\D^2 u\|_{L^4}^2\dif t\right)^\frac{3}{7}\Bigg\}\cdot\nonumber\\
%\leq~&2C_9K_1\vr^2\exp\left\{C(K)\left(\vr^{\frac{3}{4}(2-\al)}+\vr^{\frac{1-\al}{2}}+\vr^{\frac{19}{14}-\frac{17}{28}\alpha}\right)\right\}\nonumber\\
\leq~&C(K)\left(\vr^{\frac{9}{4}-\frac{\al}{2}}+\vr^{\frac{13}{4}-\frac{5}{8}\al}\right)\exp\left\{C(K)\left(\vr^{\frac{3}{4}(2-\al)}+\vr^{\frac{1-\al}{2}}+\vr^{\frac{19}{14}-\frac{17}{28}\alpha}\right)\right\}\nonumber\\
\leq~&\vr,
\end{align}
provided $\vr\geq\rL_{7,3}$, where we have used \eqref{compatibility condition}, \eqref{assume:1}, \eqref{estimate:D2theta:3}, \eqref{D2u24}, and taken $\rL_{7,3}$ large enough such that $C(K)\left(\vr^{\frac{9}{4}-\frac{\al}{2}}+\vr^{\frac{13}{4}-\frac{5}{8}\al}\right)\leq\frac{\vr}{2}$ and $C(K)\left(\vr^{\frac{3}{4}(2-\al)}+\vr^{\frac{1-\al}{2}}+\vr^{\frac{19}{14}-\frac{17}{28}\alpha}\right)\leq\ln 2$.

\bigskip

\textbf{Case 2.2}: For the time-weighted estimates of $\|\D\rho\|_{L^4}$ on the time interval $(0,T)$, multiplying \eqref{l4} by $\sigma^m$ with some constant $m\in\mathbb{N}_+$, we have
\begin{align}\label{l4:sigma}
&\frac{\dif}{\dif t}\left(\sigma^m\|\D\rho\|_{L^4}^2\right)+\frac{R\sigma^m}{2}\vr^{1-\al}\|\D\rho\|_{L^4}^2\nonumber\\
\leq~&m\sigma^{m-1}\sigma'\|\D\rho\|_{L^4}^2+C\vr^{1-\al}\sigma^m\|1-\theta\|_{L^\infty}\|\D\rho\|_{L^4}^2+C\sigma^m\|\D u\|_{L^\infty}\|\D\rho\|_{L^4}^2\nonumber\\
&+C\vr\sigma^m\|\D G\|_{L^4}\|\D\rho\|_{L^4}+C\vr^{1-\al}\sigma^m\|\D\theta\|_{L^4}\|\D\rho\|_{L^4}\nonumber\\
\leq~&m\sigma^{m-1}\sigma'\|\D\rho\|_{L^4}^2+C\vr^{1-\al}\sigma^m\|\D\theta\|_{L^2}^{\frac{1}{2}}\|\D^2\theta\|_{L^2}^{\frac{1}{2}}\|\D\rho\|_{L^4}^2+C\sigma^m\|\D u\|_{L^2}^{\frac{1}{7}}\|\D^2 u\|_{L^4}^{\frac{6}{7}}\|\D\rho\|_{L^4}^2\nonumber\\
&+C\vr^{\frac{13}{8}-\al}\sigma^m\|\sqrt{\rho}\dot{u}\|_{L^2}^{\frac{3}{4}}\|\D\dot{u}\|_{L^2}^{\frac{1}{4}}\|\D\rho\|_{L^4}+C\vr^{2-\al}\sigma^m\|\D\theta\|_{L^2}^{\frac{1}{4}}\|\D^2\theta\|_{L^2}^{\frac{3}{4}}\|\D\rho\|_{L^4},
\end{align}
where we have used \eqref{DG4}.

Taking $m=2$ and integrating \eqref{l4:sigma} over $[0,T]$, we obtain
\begin{align}\label{Dr4:2}
&\sigma^2\|\D\rho\|_{L^4}^2+\vr^{1-\al}\int_0^T\sigma^2\|\D\rho\|_{L^4}^2\dif t\nonumber\\
\leq~&C\sup\limits_{t\in[0,\sigma(T)]}\|\D\rho\|_{L^4}^2+C\vr^{1-\al}\sup\limits_{t\in[0,T]}\|\D\rho\|_{L^4}\left(\int_0^T\|\D\theta\|_{L^2}^2\dif t\right)^{\frac{1}{4}}\left(\int_0^T\sigma^3\|\D^2\theta\|_{L^2}^2\dif t\right)^{\frac{1}{4}}\left(\int_0^T\|\D\rho\|_{L^4}^2\dif t\right)^{\frac{1}{2}}\nonumber\\
&+C\sup\limits_{t\in[0,T]}\|\D\rho\|_{L^4}\left(\int_0^T\|\D u\|_{L^2}^2\dif t\right)^{\frac{1}{14}}\left(\int_0^T\sigma^3\|\D^2u\|_4^2\dif t\right)^{\frac{3}{7}}\left(\int_0^T\|\D\rho\|_{L^4}^2\dif t\right)^{\frac{1}{2}}\nonumber\\
&+C\vr^{\frac{13}{8}-\al}\left(\int_0^T\sigma^2\|\sqrt{\rho}\dot{u}\|_{L^2}^2\dif t\right)^{\frac{3}{8}}\left(\int_0^T\sigma^3\|\D\dot{u}\|_{L^2}^2\dif t\right)^{\frac{1}{8}}\left(\int_0^T\|\D\rho\|_{L^4}^2\dif t\right)^{\frac{1}{2}}\nonumber\\
&+C\vr^{2-\al}\left(\int_0^T\|\D\theta\|_{L^2}^2\dif t\right)^{\frac{1}{8}}\left(\int_0^T\sigma^3\|\D^2\theta\|_{L^2}^2\dif t\right)^{\frac{3}{8}}\left(\int_0^T\|\D\rho\|_{L^4}^2\dif t\right)^{\frac{1}{2}}\nonumber\\
%\leq~&C\vr+C(K)\left(\vr^{\frac{13-2\al}{4}}+\vr^{\frac{1}{2}}+\vr^{\frac{20-3\al}{7}}+\vr^{\frac{5(4-\al)}{8}}+\vr^{\frac{27-4\al}{8}}\right)\nonumber\\
\leq~&C_{11}\vr+C(K)\left(\vr^{\frac{1}{2}}+\vr^{\frac{20-3\al}{7}}+\vr^{\frac{27-4\al}{8}}\right)\nonumber\\
\leq~&\vr^{\frac{3}{2}},
\end{align}
provided $\vr\geq\rL_{7,4}$, where we have used \eqref{l4:4}, \eqref{assume:2}, \eqref{assume:3}, \eqref{estimate:D2theta:2}, \eqref{D2u24:sigma}, and taken $\rL_{7,4}$ large enough such that $C_{11}\vr\leq\frac{1}{2}\vr^{\frac{3}{2}}$ and $C(K)\left(\vr^{\frac{1}{2}}+\vr^{\frac{20-3\al}{7}}+\vr^{\frac{27-4\al}{8}}\right)\leq\frac{1}{2}\vr^{\frac{3}{2}}$. Then, combining \eqref{l4:4} and \eqref{Dr4:2}, and taking $\vr\geq\max\{\rL_{7,3}, \rL_{7,4}\}$, we have proved \eqref{3.9:2}.

At last, we take $\vr\geq\rL_7:=\max\{\rL_{7,1}, \rL_{7,2}, \rL_{7,3}, \rL_{7,4}\}$. From \eqref{3.9:1} and \eqref{3.9:2}, we have
\begin{align}
\|\rho-\vr\|_{L^\infty}\leq C\|\D\rho\|_{L^2}^{\frac{1}{3}}\|\D\rho\|_{L^4}^{\frac{2}{3}}\leq C(K)\vr^{\frac{5}{6}}\leq\frac{\vr}{3},
\end{align}
which yields the uniform lower and upper bounds of the density, hence \eqref{3.9:3} is proved.
\end{proof}

Now we are in a position to give a proof of Proposition \ref{p4.1}.
\begin{proof}
If we take $\rL=\max\{\rL_i\}_{i=1}^7$ with $\rL_i ~ (i=1,2,\cdots,7)$ defined in Lemmas \ref{Lem33}--\ref{Lem39}, then Proposition \ref{p4.1} can be derived from the above lemmas.
\end{proof}

\subsection{Higher-order uniform estimates of $(\rho,u,\theta)$ on $[0,T]$}\label{sec3.4}
Based on Subsection \ref{sub:3.1}, in this subsection we give the uniform bounds of the strong solution in $H^2$ Sobolev space, and in the following we will use $C$ to denote an arbitrary positive constant that may depend on $\mu$, $\lambda$, $\kappa$, $R$, $C_0$, and $K_i$, but is independent of $T$ and $\vr$.
\begin{lemma}\label{Lem:10}
Under the conditions of Proposition \ref{p4.1}, it holds that
\begin{equation}\label{sec3.4:1}
\sup_{t\in[0,T]}\left(\|u\|_{H^2}^2+\|\theta-1\|_{H^2}^2\right)+\int_0^T\left(\|u_t\|_{H^1}^2+\|\theta_t\|_{H^1}^2\right)\dif t\leq C\vr^{3\al},
\end{equation}
\begin{equation}\label{sec3.4:1.2}
\sup_{t\in[0,T]}\left(\|\rho_t\|_{L^2}^2+\|\rho-\vr\|_{H^1}^2\right)\leq C\vr^{2},
\end{equation}
\begin{equation}\label{sec3.4:1.3}
\sup_{t\in[0,T]}\left(\|\D\rho_t\|_{L^2}^2+\|\D^2\rho\|_{L^2}^2\right)\leq C\vr^{4}\exp\left\{C\vr^{\frac{1}{7}}+C\vr^{\frac{13}{7}-\frac{3}{14}\al}\right\}.
\end{equation}
\end{lemma}
\begin{proof}
Firstly, the following estimates come from Proposition \ref{p4.1} and Lemma \ref{Lem3.1} directly:
\begin{align}\label{342}
\|\rho-\vr\|_{H^1}^2+\|u\|_{H^1}^2+\|\theta-1\|_{H^1}^2\leq C\vr^\al.
\end{align}
It remains to estimate
$$\|\rho_t\|_{H^1}^2+\|\D^2\rho\|_{L^2}^2+\|\D^2u\|_{L^2}^2+\|\D^2\theta\|_{L^2}^2+\int_0^T\left(\|u_t\|_{H^1}^2+\|\theta_t\|_{H^1}^2\right)\dif t.$$

We divide the proof into several steps.

\textbf{Step 1}: Estimate of $\|\D^2u\|_{L^2}^2$ and $\|\D^2\theta\|_{L^2}^2$.

Making inner product between \eqref{3.5:3} and $\dot{\theta}$, and noticing \eqref{3.5:4}, we obtain
\begin{align}\label{343}
&\frac{1}{2}\frac{\dif}{\dif t}\int\rho|\dot{\theta}|^2\dif x+\kappa\int|\D\dot{\theta}|^2\dif x\nonumber\\
\leq~&C\int|\D\dot{\theta}||\D u||\D\theta|\dif x+C\int\rho|\dot{\theta}|^2|\D u|\dif x+C\int\rho|\theta-1||\dot{\theta}|\left(|\D\dot{u}|+|\D u|^2\right)\dif x\nonumber\\
&+C\int\rho|\dot{\theta}|\left(|\D\dot{u}|+|\D u|^2\right)\dif x+C\int\rho^\al|\dot{\theta}||\D u|\left(|\D\dot{u}|+|\D u|^2\right)\dif x\nonumber\\
\leq~&\frac{\kappa}{2}\|\D\dot{\theta}\|_{L^2}^2+C\|\D u\|_{L^2}\|\D u\|_{L^6}\|\D^2\theta\|_{L^2}^2+C\vr\|\D u\|_{L^2}\|\D u\|_{L^6}\|\sqrt{\rho}\dot{\theta}\|_{L^2}^2\nonumber\\
&+C\vr^2\left(1+\|\D\theta\|_{L^2}\right)\|\D\theta\|_{L^2}\|\D\dot{u}\|_{L^2}^2+C\vr^2\left(1+\|\D\theta\|_{L^2}\right)\|\D\theta\|_{L^2}\|\D u\|_{L^2}\|\D u\|_{L^6}^3\nonumber\\
&+C\vr^{\frac{1}{2}}\|\sqrt{\rho}\dot{\theta}\|_{L^2}\|\D\dot{u}\|_{L^2} +C\vr^{\frac{1}{2}}\|\sqrt{\rho}\dot{\theta}\|_{L^2}\|\D u\|_{L^2}^{\frac{1}{2}}\|\D u\|_{L^6}^{\frac{3}{2}}\nonumber\\
&+C\vr^{2\al}\|\D u\|_{L^2}\|\D u\|_{L^2}\|\D\dot{u}\|_{L^2}^2+C\vr^{\al-\frac{1}{2}}\|\sqrt{\rho}\dot{\theta}\|_{L^2}\|\D u\|_{L^6}^3,
\end{align}
where we have used \eqref{lem3.1:g3}.

We obtain after integration with respect to $t\in[0,T]$ that
\begin{align}\label{344}
&\int\rho|\dot{\theta}|^2\dif x+\kappa\int_0^T\int|\D\dot{\theta}|^2\dif x\dif t\nonumber\\
\leq~&C\sup\limits_{t\in[0,T]}\left(\|\D u\|_{L^2}\|\D u\|_{L^6}\right)\int_0^T\|\D^2\theta\|_{L^2}^2\dif t+C\vr\sup\limits_{t\in[0,T]}\left(\|\D u\|_{L^2}\|\D u\|_{L^6}\right)\int_0^T\|\sqrt{\rho}\dot{\theta}\|_{L^2}^2\dif t\nonumber\\
&+C\vr^2\Big(1+\sup\limits_{t\in[0,T]}\|\D\theta\|_{L^2}^2\Big)\int_0^T\|\D\dot{u}\|_{L^2}^2\dif t\nonumber\\
&+C\vr^2\Big(1+\sup\limits_{t\in[0,T]}\|\D\theta\|_{L^2}^2\Big)\sup\limits_{t\in[0,T]}\|\D u\|_{L^6}^2\int_0^T\|\D u\|_{L^2}\|\D u\|_{L^6}\dif t+C\vr^{\frac{1}{2}}\int_0^T\|\sqrt{\rho}\dot{\theta}\|_{L^2}\|\D\dot{u}\|_{L^2}\dif t\nonumber\\
&+C\vr^{\frac{1}{2}}\sup\limits_{t\in[0,T]}\|\D u\|_{L^6}\int_0^T\|\sqrt{\rho}\dot{\theta}\|_{L^2}\|\D u\|_{L^2}^{\frac{1}{2}}\|\D u\|_{L^6}^{\frac{1}{2}}\dif t+C\vr^{2\al}\sup\limits_{t\in[0,T]}\left(\|\D u\|_{L^2}\|\D u\|_{L^6}\right)\int_0^T\|\D\dot{u}\|_{L^2}^2\dif t\nonumber\\
&+C\vr^{\al-\frac{1}{2}}\sup\limits_{t\in[0,T]}\|\D u\|_{L^6}^2\int_0^T\|\sqrt{\rho}\dot{\theta}\|_{L^2}\|\D u\|_{L^6}\dif t\nonumber\\
\leq~&C\left(\vr^{\al+1}+\vr^{2\al+1}+\vr^{\frac{7}{2}}+\vr^{\al}+\vr^{\frac{7}{4}}+\vr^{3\al-1}\right)\nonumber\\
\leq~& C\vr^{3\al-1},
\end{align}
where we have used \eqref{goal:1}, \eqref{goal:2}, \eqref{-4.1:2}, \eqref{estimate:DuL6:2}, \eqref{estimate:DuL6:3}, \eqref{estimate:DuL6:4}, \eqref{estimate:DuL6:5}, \eqref{estimate:D2theta:2} and \eqref{estimate:D2theta:3}. Inserting \eqref{344} into \eqref{estimate:D2theta}, we have
\begin{align}\label{345}
\|\D^2\theta\|_{L^2}^2\leq C\vr^{3\al}.
\end{align}
Combining \eqref{342} and \eqref{345}, we get
\begin{equation}\label{345.5}
\|\theta-1\|_{L^\infty}\leq C\|\D\theta\|_{L^2}^{\frac{1}{2}}\|\D^2\theta\|_{L^2}^{\frac{1}{2}}\leq C\vr^\al.
\end{equation}
Applying standard elliptic estimate to $(\ref{FCNS:2})_2$, one has
\begin{align}\label{346}
\|\D^2u\|_{L^2}^2\leq~&C\vr^{-2}\|\D u\|_{L^6}^2\|\D\rho\|_{L^2}^{\frac{2}{3}}\|\D\rho\|_{L^4}^{\frac{4}{3}}+C\vr^{-2\al}\left(\vr\|\sqrt{\rho}\dot{u}\|_{L^2}^2+\vr^2\|\D\theta\|_{L^2}^2+\|\D\rho\|_{L^2}^2\|\D\theta\|_{L^2}\|\D^2\theta\|_{L^2}\right)
\nonumber\\
\leq~&C\vr^2.
\end{align}

\textbf{Step 2}: Estimate of $u_t$ and $\theta_t$ in $L^2\left(0,T;H^1\right)$.

From \eqref{goal:1} and \eqref{goal:2}, one has
\begin{align}
\int_0^T\left(\|u_t+u\cdot\D u\|_{L^2}^2+\|\D u_t+\D(u\cdot\D u)\|_{L^2}^2\right)\dif t\leq C\vr^{\al-1}.
\end{align}
Hence, by \eqref{goal:1}, \eqref{goal:2}, \eqref{estimate:DuL6:2} and \eqref{estimate:DuL6:4}, we have
\begin{align}\label{347}
\int_0^T\|u_t\|_{H^1}^2\dif t\leq~& C\vr^{\al-1}+C\int_0^T\left(\|u\cdot\D u\|_{L^2}^2+\|\D u\|_{L^4}^4+\|u\D^2u\|_{L^2}^2\right)\dif t\nonumber\\
\leq~&C\vr^{\al-1}+C\int_0^T\left(\|\D u\|_{L^2}^3\|\D u\|_{L^6}+\|\D u\|_{L^2}\|\D u\|_{L^6}^3+\|\D u\|_{L^2}\|\D u\|_{L^6}\|\D^2u\|_{L^2}^2\right)\dif t\nonumber\\
\leq~&C\vr^{\al-1}+C\int_0^T\left(\|\D u\|_{L^2}^2+\|\D u\|_{L^6}^2+\|\D^2u\|_{L^2}^2\right)\dif t.
\end{align}
Applying standard elliptic estimate to $(\ref{FCNS:2})_2$, we get
\begin{align}\label{348}
\int_0^T\|\D^2u\|_{L^2}^2\dif t\leq~&C\vr^{-2\al}\int_0^T\left(\vr\|\sqrt{\rho}\dot{u}\|_{L^2}^2+\vr^2\|\D\theta\|_{L^2}^2+\|\D\rho\|_{L^2}^2\|\D\theta\|_{L^2}\|\D^2\theta\|_{L^2}\right)\dif t
\nonumber\\
&+C\vr^{-2}\int_0^T\|\D u\|_{L^6}^2\|\D\rho\|_{L^2}^{\frac{2}{3}}\|\D\rho\|_{L^4}^{\frac{4}{3}}\dif t\nonumber\\
\leq~&C\vr^{1-2\al}\int_0^T\|\sqrt{\rho}\dot{u}\|_{L^2}^2\dif t+C\vr^{2-2\al}\int_0^T\left(\|\D\theta\|_{L^2}^2+\|\D^2\theta\|_{L^2}^2\right)\dif t+C\vr^{-\frac{1}{3}}\int_0^T\|\D u\|_{L^6}^2\dif t\nonumber\\
\leq~&C\vr^{3-\al},
\end{align}
where we have used \eqref{goal:1}, \eqref{goal:2}, \eqref{goal:3}, \eqref{estimate:D2theta:2}, \eqref{estimate:D2theta:3}, \eqref{-4.1:2}, and \eqref{estimate:DuL6:5}. Inserting \eqref{-4.1:2}, \eqref{estimate:DuL6:5} and \eqref{348} into \eqref{347}, we obtain
\begin{align}\label{349}
\int_0^T\|u_t\|_{H^1}^2\dif t\leq C\vr^{\al-1}.
\end{align}
From \eqref{goal:1}, \eqref{goal:2} and \eqref{344}, we have
\begin{align}
\int_0^T\left(\|\theta_t+u\cdot\D \theta\|_{L^2}^2+\|\D \theta_t+\D(u\cdot\D \theta)\|_{L^2}^2\right)\dif t\leq C\vr^{3\al-1}.
\end{align}
Therefore, by \eqref{goal:1}, \eqref{goal:2}, \eqref{estimate:DuL6:2}, \eqref{-4.1:2}, \eqref{estimate:DuL6:4}, \eqref{estimate:D2theta:2} and \eqref{estimate:D2theta:3}, we have
\begin{align}\label{3410}
\int_0^T\|\theta_t\|_{H^1}^2\dif t\leq~& C\vr^{3\al-1}+C\int_0^T\left(\|u\cdot\D\theta\|_{L^2}^2+\|\D u\cdot\D\theta\|_{L^2}^2+\|u\D^2\theta\|_{L^2}^2\right)\dif t\nonumber\\
\leq~&C\vr^{3\al-1}+C\int_0^T\|\D u\|_{L^2}\|\D u\|_{L^6}\left(\|\D\theta\|_{L^2}^2+\|\D^2\theta\|_{L^2}^2\right)\dif t\nonumber\\
\leq~&C\vr^{3\al-1}+C\int_0^T\left(\|\D\theta\|_{L^2}^2+\|\D^2\theta\|_{L^2}^2\right)\dif t\nonumber\\
\leq~&C\vr^{3\al-1}.
\end{align}

\bigskip

\textbf{Step 3}: Estimate of $\|\D^2\rho\|_{L^2}^2$ and $\|\rho_t\|_{H^1}^2$.

By direct computation, we derive the equation of $\D^2\rho$ from \eqref{FCNS:2}$_1$ as follows:
\begin{align*}
&(\D^2 \rho)_t+2(\D u\cdot\D)\D\rho+(u\cdot\nabla)\D^2\rho+\D \rho\cdot\D^2 u+\D^2 \rho\text{div}u+2\D\rho\otimes\D\text{div}u+\rho\D^2\text{div}u=0.
\end{align*}  
Multiplying the above equation by $\D^2 \rho$ and integrating the resulting equation over $\R^3$, we obtain
\begin{align}\label{sec3.4:2}
&\frac{\dif}{\dif t}\|\D^2\rho\|_{L^2}^{2}+\vr^{1-\al}\int|\D^2\rho|^{2}\dif x\nonumber\\
\leq~&\vr^{1-\al}\|1-\theta\|_{L^\infty}\|\D^2 \rho\|_{L^2}^2+ C\|\D u\|_{L^\infty}\|\D^2\rho\|_{L^2}^{2}+ C\|\D^2u\|_{L^4}\|\D \rho\|_{L^4}\|\D^2 \rho\|_{L^2}+C\vr\|\D^2 G\|_{L^2}\|\D^2\rho\|_{L^2}\nonumber\\
&+C\vr^{2-\al}\|\D^2\theta\|_{L^2}\|\D^2\rho\|_{L^2}+ C\vr^{1-\al}\|\D\theta\|_{L^6}\|\D \rho\|_{L^3}\|\D^2 \rho\|_{L^2}+C\vr^{-\al}\|\D^2\rho\|_{L^2}\|\D\rho\|_{L^4}^2\left(1+\|\theta-1\|_{L^\infty}\right)\nonumber\\
\leq~& C\left(\vr^{1-\al}\|\D\theta\|_{L^2}^{\frac{1}{2}}\|\D^2\theta\|_{L^2}^{\frac{1}{2}}+\|\D u\|_{L^2}^\frac{1}{7}\|\D^2 u\|_{L^4}^\frac{6}{7}\right)\|\D^2\rho\|_{L^2}^2+ C\|\D^2 u\|_{L^4}\|\D \rho\|_{L^4}\|\D^2 \rho\|_{L^2}\nonumber\\
&+C\vr\|\D^2 G\|_{L^2}\|\D^2\rho\|_{L^2}+C\vr^{2-\al}\|\D^2 \theta\|_{L^2}\|\D^2\rho\|_{L^2}+C\vr^{1-\al}\|\D^2\theta\|_{L^2}\|\D \rho\|_{L^3}\|\D^2 \rho\|_{L^2}\nonumber\\
&+C\vr^{-\al}\|\D^2\rho\|_{L^2}\|\D\rho\|_{L^4}^2\left(1+\|\D\theta\|_{L^2}^{\frac{1}{2}}\|\D^2\theta\|_{L^2}^{\frac{1}{2}}\right)\nonumber\\
\leq~& C\left(\vr^{1-\al}\|\D\theta\|_{L^2}^{\frac{1}{2}}\|\D^2\theta\|_{L^2}^{\frac{1}{2}}+\|\D u\|_{L^2}^\frac{1}{7}\|\D^2 u\|_{L^4}^\frac{6}{7}\right)\|\D^2\rho\|_{L^2}^2+ C\|\D^2 u\|_{L^4}\|\D \rho\|_{L^4}\|\D^2\rho\|_{L^2}\nonumber\\
&+C\vr\|\D^2\rho\|_{L^2}\left(\vr^{-\al}\|\D {\rho}\|_{L^3}\|\D\dot{u}\|_{L^2}
+\vr^{1-\al}\|\D\dot{u}\|_{L^2} +\vr^{-2}\|\D \rho \|_{L^4}^2\|\D u\|_{L^\infty}\right)+C\vr^{2-\al}\|\D^2\theta\|_{L^2}\|\D^2\rho\|_{L^2}\nonumber\\
&+ C\vr^{1-\al}\|\D^2\theta\|_{L^2}\|\D \rho\|_{L^3}\|\D^2\rho\|_{L^2}+C\vr^{-\al}\|\D^2\rho\|_{L^2}\|\D\rho\|_{L^4}^2\left(1+\|\D\theta\|_{L^2}^{\frac{1}{2}}\|\D^2\theta\|_{L^2}^{\frac{1}{2}}\right)\nonumber\\
\leq~&\frac{1}{2}\vr^{1-\al}\|\D^2 \rho\|_{L^2}^2+C\left(\vr^{1-\al}\|\D\theta\|_{L^2}\|\D^2\theta\|_{L^2}+\vr^{\al-1}\|\D u\|_{L^2}^\frac{2}{7}\|\D^2 u\|_{L^4}^\frac{12}{7}\right)\|\D^2\rho\|_{L^2}^2+C\vr^{\al-1}\|\D^2 u\|_{L^4}^2\|\D \rho\|_{L^4}^2\nonumber\\
&+C\left(\vr^{1-\al}\|\D {\rho}\|_{L^3}^2\|\D\dot{u}\|_{L^2}^2
+\vr^{3-\al}\|\D\dot{u}\|_{L^2}^2+\vr^{\alpha-3}\|\D \rho \|_{L^4}^4\|\D u\|_{L^\infty}^2\right)+C\vr^{3-\al}\|\D^2 \theta\|_{L^2}^2\nonumber\\
&+ C\vr^{1-\al}\|\D^2\theta\|_{L^2}^2\|\D \rho\|_{L^3}^2+C\vr^{-1-\al}\|\D\rho\|_{L^4}^4\left(1+\|\D\theta\|_{L^2}\|\D^2\theta\|_{L^2}\right),
\end{align}
where we have used 
\begin{align}\label{sec3.4:3}
\|\D^2 G\|_{L^2}\leq C\|\nabla H\|_{L^2}
%\leq~& \frac{C}{\vr^{\al-\frac{1}{2}}}\|\sqrt{\rho}\dot{u}\|_{L^{2}}+\frac{C}{\vr^{\al-1}}\|\D\theta\|_{L^{2}}+\frac{C}{\vr}\|\D \rho \|_{L^{2}}\|\D u \|_{L^{\infty}}\nonumber\\
%\leq~& \frac{C}{\vr^{\al}}\|\D {\rho}\|_{L^{3}}\|\dot{u}\|_{L^{6}}
%+\frac{C}{\vr^{\al-1}}\|\D\dot{u}\|_{L^{2}}+\frac{C}{\vr^{\al}}\|\D {\rho}\|_{L^{3}}\|\D\theta\|_{L^{6}}+\frac{C}{\vr^{\al-1}}\|\D^2\theta\|_{L^{2}}\nonumber\\
%&+\frac{C}{\vr^2}\|\D \rho \|_{L^{4}}^2\|\D u\|_{L^\infty}+\frac{C}{\vr}\|\D \rho \|_{L^{4}}\|\D^2 u\|_{L^{4}}+\frac{C}{\vr}\|\D u\|_{L^{\infty}}\|\D^2 \rho\|_{L^2}\nonumber\\
\leq~& \frac{C}{\vr^{\al}}\|\D {\rho}\|_{L^3}\|\D\dot{u}\|_{L^2}
+\frac{C}{\vr^{\al-1}}\|\D\dot{u}\|_{L^2} +\frac{C}{\vr^{\al}}\|\D {\rho}\|_{L^3}\|\D^2\theta\|_{L^2}+\frac{C}{\vr^{\al-1}}\|\D^2\theta\|_{L^2}\nonumber\\
&+\frac{C}{\vr^2}\|\D \rho \|_{L^4}^2\|\D u\|_{L^\infty}+\frac{C}{\vr}\|\D \rho \|_{L^4}\|\D^2 u\|_{L^4}+\frac{C}{\vr}\|\D u\|_{L^\infty}\|\D^2 \rho\|_{L^2}.
\end{align}
Integrating \eqref{sec3.4:2} over $(0,T)$, we have
\begin{align}\label{sec3.4:6}
&\|\D^2\rho\|_{L^2}^2+\frac{1}{2}\vr^{1-\al}\int_0^T\|\D^2\rho\|_{L^2}^2\dif t\nonumber\\
\leq~&\|\D^2\rho_0\|_{L^2}^2+C\int_0^{T}\left(\vr^{1-\al}\|\D\theta\|_{L^2}\|\D^2\theta\|_{L^2}+\vr^{\al-1}\|\D u\|_{L^2}^\frac{2}{7}\|\D^2 u\|_{L^4}^\frac{12}{7}\right)\|\D^2\rho\|_{L^2}^2\dif t\nonumber\\
&+C\int_0^{T}\left(\vr^{\al+\frac{1}{2}}\|\D^2 u\|_{L^4}^2+\vr^{3-\al}\|\D\dot{u}\|_{L^2}^2+\vr^{\alpha}\|\D u\|_{L^2}^\frac{2}{7}\|\D^2 u\|_{L^4}^\frac{12}{7}+\vr^{3-\al}\|\D^2 \theta\|_{L^2}^2\right)\dif t\nonumber\\
&+C\int_0^T\left(\vr^{\frac{1}{2}-\al}\|\D\rho\|_{L^4}^2+\vr^{2-\al}\|\D\theta\|_{L^2}\|\D^2\theta\|_{L^2}\right)\dif t\nonumber\\
\leq~&\|\D^2\rho_0\|_{L^2}^2+C\vr^4+C\int_0^{T}\left(\vr^{1-\al}\|\D\theta\|_{L^2}\|\D^2\theta\|_{L^2}+\vr^{\al-1}\|\D u\|_{L^2}^\frac{2}{7}\|\D^2 u\|_{L^4}^\frac{12}{7}\right)\|\D^2\rho\|_{L^2}^2\dif t.
% \leq~&C\vr^5+C\int_0^{T}\sigma^2\left(\vr^{-\al}\|\D\theta\|_{L^2}\|\D^2\theta\|_{L^2}+\vr^{\al}\|\D u\|_{L^2}^\frac{2}{7}\|\D^2 u\|_{L^{4}}^\frac{12}{7}\right)\|\D^2\rho\|_{L^2}^2\dif t\nonumber\\
% &+C\int_0^{T}\sigma^2\left(\vr^{\frac{3}{2}}\|\D^2 u\|_{L^{4}}^2+\vr^{4-\al}\|\D\dot{u}\|_{L^{2}}^2+\vr^{\alpha+1}\|\D u\|_{L^2}^\frac{2}{7}\|\D^2 u\|_{L^4}^\frac{12}{7}+\vr^{4-\al}\|\D^2 \theta\|_{L^2}^2\right)\dif t\nonumber\\
\end{align}
Then, by using Gronwall's inequality and (\ref{l2-2-1}), we have
\begin{align}\label{sec3.4:7}
&\sup\limits_{t\in[0,T]}\|\D^2\rho\|_{L^2}^2+\vr^{1-\al}\int_0^T\|\D^2\rho\|_{L^2}^2\dif t\nonumber\\
\leq~& C\vr^{4}\exp\left\{C\int_0^{T}\left(\vr^{1-\al}\|\D\theta\|_{L^2}\|\D^2\theta\|_{L^2}+\vr^{\al-1}\|\D u\|_{L^2}^\frac{2}{7}\|\D^2 u\|_{L^4}^\frac{12}{7}\right)\dif t\right\}\nonumber\\
%\leq~& C\vr^{5}\exp\left\{C\left(\vr^{2-\frac{1}{2}\al}+\vr^{\frac{8}{7}}\right)\right\}\nonumber\\
\leq~& C\vr^{4}\exp\left\{C\vr^{\frac{1}{7}}+C\vr^{\frac{13}{7}-\frac{3}{14}\al}\right\},
\end{align}
where we have used (\ref{g:3.4:1}), (\ref{D2u24:sigma}), (\ref{Lem:4:1}), (\ref{D2u24}), (\ref{g:3.5}), (\ref{-4.1:3}),(\ref{estimate:D2theta:3}), (\ref{estimate:D2theta:2}),(\ref{-4.1:2}) and (\ref{-4.1:3}).

By \eqref{FCNS:2}$_1$, we have
\begin{align}\label{sec3.4:8}
\|\rho_t\|_{L^2}\leq
C\|{u}\|_{L^\infty}\|\D\rho\|_{L^2}+C\vr\|\D u\|_{L^2}\leq C\vr,
\end{align}
and
\begin{align}\label{sec3.4:9}
\nabla\rho_t+(u\cdot\D)\D\rho+\nabla u^i\nabla\partial_i\rho+\nabla\rho\text{div}u+\rho\nabla\text{div}u=0.
\end{align}
Then \eqref{sec3.4:9} implies
\begin{align}\label{sec3.4:10}
\|\nabla\rho_t\|_{L^2}^2
\leq
C\|{u}\|_{L^\infty}^2\|\nabla^2\rho\|_{L^2}^2+C\|\nabla{u}\|_{L^2}^{\frac{1}{2}}\|\nabla{u}\|_{L^6}^{\frac{3}{2}}\|\nabla\rho\|_{L^4}+C\vr\|\nabla^2{u}\|_{L^2}^2\leq C\vr^{4}\exp\left\{C\vr^{\frac{1}{7}}+C\vr^{\frac{13}{7}-\frac{3}{14}\al}\right\}.
\end{align}
Combining \eqref{342}, \eqref{345}, \eqref{346}, \eqref{349}, \eqref{3410}, \eqref{sec3.4:7}, and \eqref{sec3.4:10}, we complete the proof.
\end{proof}

Finally, from Proposition \ref{p4.1} and Lemma \ref{Lem:10}, we can use a standard continuation argument to show that the local solution can be extended to be a global one.

\subsection{Large time behavior and lower bound of the temperature}

In this subsection, we give the large time behaviour of the strong solution $(\rho, u,\theta)$ in $H^2$ space. Thanks to the fact that $\D\theta$ tends to zero in $L^2$ as time goes to infinity, we are able to give a uniform lower bound of the temperature. Compared to the results in \cite{MR3744381}, we prove the large time behaviour of the second-order derivatives of strong solution, additionally, due to the uniform-in-time higher-order estimates in Subsection \ref{sec3.4}.

% \begin{remark}
% Recalling the definition of $G$ in \eqref{G}, we have the following equation of $\rho-\vr$:
% \begin{align}\label{rho6:eqn}
% (\rho-\vr)_t-\frac{R\vr}{2\mu+\lambda}\left(\rho^{1-\al}-\vr^{1-\al}\right)=-\div((\rho-\vr)u)-\frac{\vr G}{2\mu+\lambda}+\frac{R\vr\rho^{1-\al}}{2\mu+\lambda}(\theta-1).
% \end{align}

% Moreover, there exists a constant $C_*>0$ depending on $\al$, $\mu$, $\lambda$, and $R$, such that
% \begin{align}\label{3.8:2}
% -\frac{R\vr}{2\mu+\lambda}\left(\rho^{1-\al}-\vr^{1-\al}\right)(\rho-\vr)\geq C_*\vr^{1-\al}(\rho-\vr)^2.
% \end{align}
% \end{remark}

\begin{lemma}
Let $(\rho, u, \theta)$ be the global strong solution obtained in Theorem \ref{th}. Then, we have
\begin{equation}\label{decay1}
    \lim_{t\rightarrow \infty} \left(\|\D\rho\|_{H^1}+\|\nabla\theta\|_{H^1}+\|\nabla u\|_{H^1}\right)=0.
\end{equation}    
\end{lemma}

\begin{proof}
    From Lemma \ref{sky}, we have
\begin{equation}\label{1001}
    \int_0^\infty \left(\|\nabla u\|_{L^2}^2+\|\nabla \theta\|_{L^2}^2\right) \mathrm{d}t\leq C(\vr),
\end{equation}
where $C(\vr)$ denotes a generic positive constant that depends only on $\vr$, $\mu$, $\lambda$, $\kappa$, $R$, $K_i$ and the initial data.
% \begin{equation}
% \begin{split}
%     \int_1^\infty \int \rho^\alpha(|\frD(u)|^2+(\text{div} u)^2)\mathrm{d}x\mathrm{d}t\leq~& C \vr^{\alpha} \int_1^\infty \int (|\frD(u)|^2+(\text{div} u)^2)\mathrm{d}x\mathrm{d}t\\
%     \leq~& C \vr^{\alpha} \int_1^\infty  \|\nabla u\|_{L^2}^2\mathrm{d}t<+\infty.
%     \end{split}
% \end{equation}

It follows from \eqref{1001} and Lemma \ref{Lem:4} that
\begin{equation}\label{1002}
\int_0^\infty \left|\frac{\mathrm{d}}{\mathrm{d}t} \|\nabla u\|_{L^2}^2\right|\mathrm{d}t\leq\int_0^\infty \int|\nabla u\cdot \nabla u_t|\mathrm{d}x\mathrm{d}t\leq\int_0^\infty\|\D u\|_{L^2}\|\D u_t\|_{L^2}\dif t\leq C(\vr).
\end{equation}
Combining \eqref{1001} and \eqref{1002} gives that $\lim\limits_{t\rightarrow \infty} \|\nabla u\|_{L^2}=0$. Integrating \eqref{Lem:5:2} with respect to $t\in[0,\infty)$, we can obtain that 
\begin{equation}\label{1003}
\begin{split}
&\int_0^\infty \left|\frac{\mathrm{d}}{\mathrm{d}t}\|\nabla \theta\|_{L^2}^2\right| \mathrm{d}t\\
   \leq~ &C\left(\vr\int_0^\infty\|\D u\|_{L^2}^2\dif t+\vr\int_0^\infty\|\D u\|_{L^2}\|\D u\|_{L^6}\|\D\theta\|_{L^2}^2\dif t+\vr^{2\al-1}\int_0^\infty\|\D u\|_{L^2}\|\D u\|_{L^6}^3\dif t\right)\\
   \leq~ & C(\vr).
\end{split}
\end{equation}
Combining \eqref{1001} and \eqref{1003} gives that $\lim\limits_{t\rightarrow \infty} \|\nabla \theta\|_{L^2}=0$. Similarly, from \eqref{3.9:1}, \eqref{l2-1}, \eqref{sec3.4:2} and \eqref{sec3.4:7}, one has $\lim\limits_{t\rightarrow \infty} \| \nabla\rho\|_{H^1}=0$.

Now it remains to prove that $\lim\limits_{t\rightarrow \infty} (\|\nabla^2 u\|_{L^2}+\|\nabla^2 \theta\|_{L^2})=0.$ It follows from \eqref{g:3.3:1} and \eqref{Lem:4:1} that
\begin{equation}\label{L100}
    \int_0^\infty \|\sqrt{\rho}\dot{u}\|_{L^2}^2 \mathrm{d}t \leq C(\vr). 
\end{equation}
Besides, from \eqref{3.4.1}, we have
\begin{equation}
    \int_0^\infty \left|\frac{\mathrm{d}}{\mathrm{d} t} \|\sqrt{\rho}\dot{u}\|_{L^2}^2\right| \mathrm{d}t \leq C(\vr),
\end{equation}
which, along with \eqref{L100}, yields that $\lim\limits_{t\rightarrow \infty} \|\sqrt{\rho}\dot{u}\|_{L^2}^2 =0$. Similar to \eqref{D2u4}, we have
\begin{equation}
\|\D^2u\|_{L^2}\leq C\vr^{-\al}\left(\vr^{\frac{1}{2}}\|\sqrt{\rho}\dot{u}\|_{L^2}+\vr^{\al-1}\|\D\rho\D u\|_{L^2}+\|\D\rho\|_{L^2}\|\theta\|_{L^\infty}+\vr\|\D\theta\|_{L^2}\right),
\end{equation}
which implies that $\lim\limits_{t\rightarrow \infty} \|\nabla^2 u\|_{L^2}=0$.

It follows from \eqref{g:3.3:2} and \eqref{Lem:4:1} that
\begin{equation}\label{L101}
   \int_0^\infty \| \sqrt{\rho}\dot{\theta} \|_{L^2}^2 \mathrm{d}t \leq C(\vr).
\end{equation}
From \eqref{343}, we have
\begin{equation}
    \int_0^\infty \left|\frac{\mathrm{d}}{\mathrm{d} t} \|\sqrt{\rho}\dot{\theta}\|_{L^2}^2\right| \mathrm{d}t \leq C(\vr),
\end{equation}
which, along with \eqref{L101}, yields that $\lim\limits_{t\rightarrow \infty} \|\sqrt{\rho}\dot{\theta}\|_{L^2}^2 =0$. Then from \eqref{estimate:D2theta}, we have $\lim\limits_{t\rightarrow \infty} \|\nabla^2 \theta\|_{L^2}=0$. Thus we complete the proof of the lemma.
\end{proof}

The next lemma shows that the temperature has the uniform positive lower and upper bounds.
\begin{lemma}\label{lower-theta}
Let $(\rho,u,\theta)$ be a strong solution of \eqref{FCNS:2}--\eqref{data and far} obtained in Theorem \ref{th}. Then there exists positive constants $\overline{\Theta}$ and $\underline{\Theta}$ depending on $\mu$, $\lambda$, $\kappa$, $R$, $\vr$ and the initial data, such that for any $(x,t)\in\R^3\times[0, +\infty)$,
\begin{align}
\underline{\Theta}\leq\theta(x,t)\leq \overline{\Theta}.
\end{align}
\end{lemma}
\begin{proof}
The upper bound of the temperature follows directly from \eqref{-4.7}. Next, we divide the proof of lower bound into two steps. The first step is to prove that, for any $T_0>0$, there exists a positive constant $C(T_0)$ depending on $T_0$, $\mu$, $\lambda$, $\kappa$, $R$, $\vr$ and the initial data, such that for any $(x,t)\in\R^3\times[0, T_0]$,
\begin{align}\label{lower-finite}
\theta(x,t)\geq C(T_0).
\end{align}

Denote $\Theta=\frac{1}{\theta}-1$ and $\Theta_+=\max\{\Theta, 0\}$. Then from \eqref{FCNS:2}$_3$ we can derive the following inequality,
\begin{align}\label{eqn:Theta}
\rho\left(\Theta_t+u\cdot\D\Theta\right)-\kappa\Delta\Theta=~&-\frac{1}{\theta^2}\left(2\mu\rho^\al|\frD(u)|^2+\lambda\rho^\al(\div u)^2\right)-\frac{1}{\theta^3}|\D\theta|^2+R\rho(\Theta+1)\div u\nonumber\\
\leq~&R\rho(\Theta+1)\div u.
\end{align}
Multiplying $\Theta_+^{p-1}$ on both sides of \eqref{eqn:Theta}, and integrating the resulting equation over $\R^3$, we have
\begin{align}
&\frac{1}{p}\frac{\dif}{\dif t}\int\rho|\Theta_+|^p\dif x+\kappa(p-1)\int\Theta_+^{p-2}|\D\Theta_+|^2\dif x\nonumber\\
\leq~&C\left(\|\rho\div u\|_{L^\infty}\|\Theta_+\|_{L^p}+\|\rho\div u\|_{L^p}\right)\|\Theta_+\|_{L^p}^{p-1}\nonumber\\
\leq~& C\left(\|\rho\div u\|_{L^\infty}\|\rho^{\frac{1}{p}}\Theta_+\|_{L^p}+\|\rho\div u\|_{L^p}\right)\|\rho^{\frac{1}{p}}\Theta_+\|_{L^p}^{p-1},
\end{align}
for any $p\in[2,\infty)$, which implies 
\begin{equation}
\frac{\dif}{\dif t}\|\rho^{\frac{1}{p}}\Theta_+\|_{L^p}\leq C\vr\left(\|\div u\|_{L^\infty}\|\rho^{\frac{1}{p}}\Theta_+\|_{L^p}+\|\div u\|_{L^p}\right).
\end{equation}
By Gronwall's inequality and H\"older's inequality, we get
\begin{align*}
\|\rho^{\frac{1}{p}}\Theta_+\|_{L^p}\leq~&\big\|\rho_0^{\frac{1}{p}}\Theta_+(0)\big\|_{L^p}\exp\left\{C\vr\int_0^t\|\div u\|_{L^\infty}\dif s\right\}+C\vr\int_0^t\|\div u\|_{L^p}\dif s\nonumber\\
\leq~&\big\|\rho_0^{\frac{1}{p}}\Theta_+(0)\big\|_{L^2\cap L^\infty}\exp\left\{C\vr t^{\frac{1}{2}}\left(\int_0^\infty\|\D u\|_{L^\infty}^2\dif t\right)^{\frac{1}{2}}\right\}+C\vr t^{\frac{1}{2}}\left(\int_0^\infty\|\D u\|_{L^2\cap L^\infty}^2\dif t\right)^{\frac{1}{2}}.
\end{align*}

Noting that the above inequality holds for any $p\in[2,\infty)$, we can let $p\to +\infty$ to obtain the time-dependent upper bound for $\|\rho^{\frac{1}{p}}\Theta_+\|_{L^\infty}$.

Next, from \eqref{decay1}, we have
\begin{equation}\label{theta:6:decay}
\|\theta(t)-1\|_{L^\infty}^2\leq C\|\D\theta(t)\|_{L^2}\big\|\D^2\theta(t)\big\|_{L^2}\to 0, \quad \text{as} ~ t\to +\infty.
\end{equation}
Combining \eqref{lower-finite} and \eqref{theta:6:decay}, we obtain the uniform lower bound of $\theta$:
\begin{align}\label{lower-theta:uniform}
\theta(x,t)\geq \underline{\Theta}, \quad \text{for} ~ (x,t)\in\R^3\times[0,+\infty).
\end{align}
This completes the proof of Lemma \ref{lower-theta}.
\end{proof}

\section{$H^1$ Asymptotic Behavior}\label{sec4}

The aim of this section is to acquire the convergence rate of the global solution $(\rho, u, \theta)$ obtained in Theorem \ref{th} to the equilibrium. In this section, we use $C(\vr)$ and $\overline{C}_i(\vr) ~ (i\in\mathbb{N})$ to denote a generic positive constant that depends only on $\vr$, $\mu$, $\lambda$, $\kappa$, $R$, $K_i$ and the initial data. The following lemma is devoted to deriving a dissipation inequality of the global solution $(\rho, u, \theta)$, which will play a key role in Lemma \ref{decayprop}.
\begin{lemma}\label{Lem:4.1}
Let $(\rho, u, \theta)$ be the global strong solution obtained in Theorem \ref{th}. Then, we have
\begin{align}\label{Lem4.1:6-0}
&\frac{\dif}{\dif t}\bigg[\int\vr^{\alpha+1}\left(\frac{1}{2}\rho|u|^2+R(\rho\ln\rho-\rho-\rho\ln\vr+\vr)+\rho(\theta-\ln\theta-1)\right)\dif x\nonumber\\
&+\mu\int\rho^\al|\frD(u)|^2\dif x+\frac{\lambda}{2}\int\rho^\al(\div u)^2\dif x-R\int(\rho\theta-\vr)\div u\dif x+\frac{\kappa}{2}\vr^{\frac{5}{2}-\alpha}\|\D\theta\|_{L^2}^2
\nonumber\\
&+\|\D\rho\|_{L^2}^2+\vr^{-\alpha}\|\sqrt{\rho}\dot{u}\|_{L^2}^2+\vr^{-2\alpha}\|\sqrt{\rho}\dot{\theta}\|_{L^2}^2\bigg]+C^{-1}\Big(\vr^{2\al}\|\nabla u\|_{L^2}^2+\vr^{\al-1}\|\D\theta\|_{L^2}^2\nonumber\\
&+\|\sqrt{\rho}\dot{u}\|_{L^2}^2+\vr^{\frac{5}{2}-\alpha}\|\sqrt{\rho}\dot{\theta}\|_{L^2}^2+\vr^{1-\al}\|\D\rho\|_{L^2}^2+\|\D\dot u\|_{L^2}^2+\vr^{-2\alpha}\|\D\dot{\theta}\|_{L^2}^2\Big)\leq 0,
\end{align}
for any $t>1$.
\end{lemma}
\begin{proof}
\textbf{Step 1}: The dissipation inequality of $\|(\rho-\vr,\sqrt{\rho}u,\theta-1)\|_{L^2}$.

From (\ref{Lem3.1:1}) we have
\begin{align}\label{Lem4.1:1}
&\frac{\dif}{\dif t}\int\left(\frac{1}{2}\rho|u|^2+R(\rho\ln\rho-\rho-\rho\ln\vr+\vr)+\rho(\theta-\ln\theta-1)\right)\dif x\nonumber\\
&+\int\left(\frac{1}{\theta}\left(\lambda\rho^\al(\div u)^2+2\mu\rho^\al|\frD(u)|^2\right)+\frac{\kappa}{\theta^2} |\D\theta|^2\right)\dif x=0,
\end{align}
which, together with (\ref{3.5:1}), leads to
\begin{align}\label{Lem4.1:2}
&\frac{\dif}{\dif t}\int\left(\frac{1}{2}\rho|u|^2+R(\rho\ln\rho-\rho-\rho\ln\vr+\vr)+\rho(\theta-\ln\theta-1)\right)\dif x\nonumber\\
&+C\int|\nabla u|^2\dif x+C\vr^{-2\alpha}\int|\D\theta|^2\dif x\leq0.
\end{align}
Moreover, from Lemma \ref{Lem:4}, taking $\sigma=1$, we have

\textbf{Step 2}: The dissipation inequality of $\|\nabla u\|_{L^2}$.
\begin{align}\label{Lem:4.1:3}
\int\rho|\dot{u}|^2\dif x=~&-\int\dot{u}\cdot\D P\dif x-2\mu\int\rho^\al\D\dot{u}:\frD(u)\dif x-\lambda\int\rho^\al\div u\div\dot{u}\dif x\nonumber\\
\leq~&R\frac{\dif}{\dif t}\left(\int(\rho\theta-\vr)\div u\dif x\right)+C\vr^{\frac{1}{2}}\|\sqrt{\rho}\dot{\theta}\|_{L^2}\|\D u\|_{L^2}+C\vr\|\D u\|_{L^2}^{\frac{3}{2}}\|\D\theta\|_{L^2}\|\D u\|_{L^6}^{\frac{1}{2}}+C\vr\|\D u\|_{L^2}^2\nonumber\\
&-\mu\frac{\dif}{\dif t}\left(\int\rho^\al|\frD(u)|^2\dif x\right)+C\vr^\al\|\D u\|_{L^2}^{\frac{3}{2}}\|\D u\|_{L^6}^{\frac{3}{2}}-\frac{\lambda}{2}\frac{\dif}{\dif t}\left(\int\rho^\al(\div u)^2\dif x\right),
\end{align}
which leads to 
\begin{align}\label{Lem:4.1:4}
&\frac{\dif}{\dif t}\left(\mu\int\rho^\al|\frD(u)|^2\dif x+\frac{\lambda}{2}\int\rho^\al(\div u)^2\dif x-R\int(\rho\theta-\vr)\div u\dif x\right)+\int\rho|\dot{u}|^2\dif x\nonumber\\
\leq~&C\vr^{\frac{1}{2}}\|\sqrt{\rho}\dot{\theta}\|_{L^2}\|\D u\|_{L^2}+C\vr\|\D u\|_{L^2}^{\frac{3}{2}}\|\D\theta\|_{L^2}\|\D u\|_{L^6}^{\frac{1}{2}}+C\vr\|\D u\|_{L^2}^2+C\vr^\al\|\D u\|_{L^2}^{\frac{3}{2}}\|\D u\|_{L^6}^{\frac{3}{2}}\nonumber\\
\leq~& C\vr^{\al-1}\|\D u\|_{L^2}^2+ C\vr^{2-\al}\|\sqrt{\rho}\dot{\theta}\|_{L^2}^2+C\vr^{-3\al+7}\|\D\theta\|_{L^2}^4\|\D u\|_{L^6}^2+C\vr^{\al+3}\|\D u\|_{L^6}^{6}\nonumber\\
\leq~& C\vr^{\al-1}\|\D u\|_{L^2}^2+ C\vr^{2-\al}\|\sqrt{\rho}\dot{\theta}\|_{L^2}^2+C\vr^{-3\al+7}\|\D\theta\|_{L^2}^2\nonumber\\
&+C\vr^{\al+3}\left(\vr^{1-2\al}\|\sqrt{\rho}\dot{u}\|_{L^2}^2+\vr^{2-2\al}\|\D\theta\|_{L^2}^2+\vr^{-2\al}\|\nabla\rho\|_{L^2}^2+\vr^{-2}\|\D u\|_{L^2}^2\right)\nonumber\\
\leq~& C\vr^{4-\al}\|\sqrt{\rho}\dot{u}\|_{L^2}^2+C\left(\vr^{\al+1}\|\D u\|_{L^2}^2+\vr^{2-\al}\|\sqrt{\rho}\dot{\theta}\|_{L^2}^2+\vr^{5-\al}\|\D\theta\|_{L^2}^2+\vr^{3-\al}\|\D\rho\|_{L^2}^2\right),
% \leq~& \delta\vr^{\al-1}\|\D u\|_{L^2}^2+\frac{1}{4}\|\sqrt{\rho}\dot{u}\|_{L^2}^2+\delta\vr^{-2}\|\D\theta\|_{L^2}^2+ C\vr^{2-\al}\|\sqrt{\rho}\dot{\theta}\|_{L^2}^2+C\vr^{7-5\al}\|\D\rho\|_{L^2}^2,
\end{align}
where in the last inequality we have used (\ref{estimate:DuL6}) and for any $t>1$,
\begin{align*}
\|\D u\|_{L^6}\leq~& C\left(\vr^{\frac{1}{2}-\al}\|\sqrt{\rho}\dot{u}\|_{L^2}+\vr^{1-\al}\|\D\theta\|_{L^2}+\vr^{-\al}\|\nabla\rho\|_{L^2}+\vr^{-4}\|\D\rho\|_{L^4}^4\|\D u\|_{L^2}\right)\nonumber\\
\leq~& C\left(\vr^{\frac{1}{2}-\al}\|\sqrt{\rho}\dot{u}\|_{L^2}+\vr^{1-\al}\|\D\theta\|_{L^2}+\vr^{-\al}\|\nabla\rho\|_{L^2}+\vr^{-1}\|\D u\|_{L^2}\right)\nonumber\\
\leq~& C\vr^{-1}.
\end{align*}

\textbf{Step 3}: The dissipation inequality of $\|\nabla \theta\|_{L^2}$.

From (\ref{Lem:5:2}), we have
\begin{align}\label{Lem:4.1:5}
&\frac{\dif}{\dif t}\left(\frac{\kappa}{2}\int|\D\theta|^2\dif x\right)+\int\rho|\dot{\theta}|^2\dif x\nonumber\\
=~&\int\rho\dot{\theta}u\cdot\D\theta\dif x+\lambda\int\rho^\al\dot{\theta}(\div u)^2\dif x+2\mu\int\rho^\al\dot{\theta}|\frD(u)|^2\dif x-R\int\rho\theta\div u\dot{\theta}\dif x\nonumber\\
&+R\int\rho\theta u\cdot\D\theta\div u\dif x-\lambda\int\rho^\al u\cdot\D\theta(\div u)^2\dif x-2\mu\int\rho^\al u\cdot\D\theta|\frD(u)|^2\dif x\nonumber\\
\leq~&\frac{1}{4}\|\sqrt{\rho}\dot{\theta}\|_{L^2}^2+C\vr\left(\|\D \theta\|_{L^2}\|\D^2\theta\|_{L^2}\|\D u\|_{L^2}^2+\|\D u\|_{L^2}\|\D\theta\|_{L^2}^2\|\D u\|_{L^6}\right)+C\vr^{2\al-1}\|\D u\|_{L^2}\|\D u\|_{L^6}^3\nonumber\\
&+C\vr^{\al}\|\D u\|_{L^2}\|\D \theta\|_{L^2}\|\D u\|_{L^6}^2\nonumber\\
\leq~&\frac{1}{4}\|\sqrt{\rho}\dot{\theta}\|_{L^2}^2+C\left(\vr^{\al-2}\|\D\theta\|_{L^2}^2+\vr^{-1}\|\sqrt{\rho}\dot{u}\|_{L^2}^2+\vr^{-2}\|\nabla\rho\|_{L^2}^2+\vr^{2\alpha-4}\|\D u\|_{L^2}^2\right).
\end{align}

\textbf{Step 4}: The dissipation inequality of $\|\nabla \rho\|_{L^2}$.

From (\ref{l2-1}), we have
\begin{align}\label{Lem:4.1:6}
&\frac{\dif}{\dif t}\|\D\rho\|_{L^2}^2+\frac{1}{4}\vr^{1-\al}\|\D\rho\|_{L^2}^2\nonumber\\
\leq~&C\left(\vr^{1-\al}\|\D\theta\|_{L^2}^{\frac{1}{2}}\|\D^2\theta\|_{L^2}^{\frac{1}{2}}+\|\D u\|_{L^2}^\frac{1}{7}\|\D^2 u\|_{L^{4}}^\frac{6}{7}\right)\|\D\rho\|_{L^2}^2+C\vr^{2-\al}\|\sqrt{\rho}\dot{u}\|_{L^{2}}^2+C\vr^{3-\al}\|\D\theta\|_{L^{2}}^2\nonumber\\
\leq~&\frac{1}{16}\vr^{1-\al}\|\D\rho\|_{L^2}^2+C\left(\vr^{3-\al}\|\D\theta\|_{L^2}\|\D^2\theta\|_{L^2}+\vr^{\al+1}\|\D u\|_{L^2}^\frac{2}{7}\|\D^2 u\|_{L^{4}}^\frac{12}{7}\right)+C\vr^{2-\al}\|\sqrt{\rho}\dot{u}\|_{L^{2}}^2+C\vr^{3-\al}\|\D\theta\|_{L^{2}}^2\nonumber\\
\leq~&\frac{1}{16}\vr^{1-\al}\|\D\rho\|_{L^2}^2+C\vr^{-\al}\|\D^2\theta\|_{L^2}^2+C\vr^{2-\al}\|\sqrt{\rho}\dot{u}\|_{L^{2}}^2+C\vr^{6-\al}\|\D\theta\|_{L^{2}}^2+C\vr^{\al+1}\|\D u\|_{L^2}^\frac{2}{7}\bigg(\vr^{1-\al}\|\dot{u}\|_{L^4}\nonumber\\
&+\vr^{-7}\|\D\rho\|_{L^4}^7\|\D u\|_{L^2}+\vr^{-\al}\|\D\rho\|_{L^4}\|\D\theta\|_{L^2}^{\frac{1}{2}}\|\D^2\theta\|_{L^2}^{\frac{1}{2}}+\vr^{1-\al}\|\D\theta\|_{L^2}^{\frac{1}{4}}\|\D^2\theta\|_{L^2}^{\frac{3}{4}}\bigg)^\frac{12}{7}\nonumber\\
\leq~&\frac{1}{16}\vr^{1-\al}\|\D\rho\|_{L^2}^2+C\vr^{-\al}\|\D^2\theta\|_{L^2}^2+C\vr^{2-\al}\|\sqrt{\rho}\dot{u}\|_{L^{2}}^2+C\vr^{6-\al}\|\D\theta\|_{L^{2}}^2+C\vr^{\al+1}\bigg(\|\D u\|_{L^2}^2+\vr^{2-2\al}\|\dot{u}\|_{L^4}^2\nonumber\\
&+\vr^{-14}\|\D\rho\|_{L^4}^{14}\|\D u\|_{L^2}^2+\vr^{-2\al}\|\D\rho\|_{L^4}^2\|\D\theta\|_{L^2}\|\D^2\theta\|_{L^2}+\vr^{2-2\al}\|\D\theta\|_{L^2}^{\frac{1}{2}}\|\D^2\theta\|_{L^2}^{\frac{3}{2}}\bigg)\nonumber\\
\leq~&\frac{1}{16}\vr^{1-\al}\|\D\rho\|_{L^2}^2+C\left(\vr^{1-\al}\|\sqrt{\rho}\dot{\theta}\|_{L^2}^2+\vr^{\al-1}\|\D u\|_{L^6}^2+\vr^{1-\al}\|\D\theta\|_{L^2}^2
\right)\nonumber\\
&+C\vr^{3-\al}\|\sqrt{\rho}\dot{u}\|_{L^{2}}^2+C\vr^{\al+1}\|\D u\|_{L^2}^2+C\vr^{3-\al}\|\D \dot{u}\|_{L^2}^2\nonumber\\
\leq~&\frac{1}{16}\vr^{1-\al}\|\D\rho\|_{L^2}^2+C\left(\vr^{1-\al}\|\sqrt{\rho}\dot{\theta}\|_{L^2}^2+\vr^{-1-\al}\|\nabla\rho\|_{L^2}^2+\vr^{4-\al}\|\D\theta\|_{L^2}^2
\right)\nonumber\\
&+C\vr^{3-\al}\|\sqrt{\rho}\dot{u}\|_{L^{2}}^2+C\vr^{\al+1}\|\D u\|_{L^2}^2+C\vr^{3-\al}\|\D \dot{u}\|_{L^2}^2\nonumber\\
\leq~&\frac{1}{8}\vr^{1-\al}\|\D\rho\|_{L^2}^2+C\left(\vr^{1-\al}\|\sqrt{\rho}\dot{\theta}\|_{L^2}^2+\vr^{4-\al}\|\D\theta\|_{L^2}^2
+\vr^{3-\al}\|\sqrt{\rho}\dot{u}\|_{L^{2}}^2+\vr^{\al+1}\|\D u\|_{L^2}^2+\vr^{3-\al}\|\D \dot{u}\|_{L^2}^2\right),
\end{align}
where we have used (\ref{estimate:D2theta})
\begin{align}\label{2dt}
\|\D^2\theta\|_{L^2}^2 
&\leq C\left(\vr\|\sqrt{\rho}\dot{\theta}\|_{L^2}^2+\vr^{2\al-1}\|\D u\|_{L^6}^2+\vr\|\D\theta\|_{L^2}^2+\vr^2\|\D u\|_{L^2}^2\right).
% \nonumber\\
% &\leq C\left(\vr^3+\vr^{2\al-3}+\vr^2+\vr^{3-\alpha}\right)\nonumber\\
\end{align}

Then, multiplying $\vr^{\alpha+1}$ on (\ref{Lem4.1:1}), multiplying $\vr^{\frac{5}{2}-\alpha}$ on (\ref{Lem:4.1:5}), and adding the resulting inequality to (\ref{Lem:4.1:4}), (\ref{Lem:4.1:6}) together, we have
\begin{align}\label{Lem4.1:6-1}
&\frac{\dif}{\dif t}\int\vr^{\alpha+1}\left(\frac{1}{2}\rho|u|^2+R(\rho\ln\rho-\rho-\rho\ln\vr+\vr)+\rho(\theta-\ln\theta-1)\right)\dif x\nonumber\\
&+\frac{\dif}{\dif t}\left(\mu\int\rho^\al|\frD(u)|^2\dif x+\frac{\lambda}{2}\int\rho^\al(\div u)^2\dif x-R\int(\rho\theta-\vr)\div u\dif x\right)+\frac{\dif}{\dif t}\left(\frac{\kappa}{2}\vr^{\frac{5}{2}-\alpha}\int|\D\theta|^2\dif x\right)
\nonumber\\
&+\frac{\dif}{\dif t}\|\D\rho\|_{L^2}^2+\frac{C}{2}\vr^{2\al}\int|\nabla u|^2\dif x+\frac{C}{2}\vr^{\al-1}\int|\D\theta|^2\dif x+\frac{1}{2}\int\rho|\dot{u}|^2\dif x+\frac{\vr^{\frac{5}{2}-\alpha}}{2}\int\rho|\dot{\theta}|^2\dif x+\frac{\vr^{1-\al}}{16}\|\D\rho\|_{L^2}^2\nonumber\\
\leq~& C\vr^{3-\al}\|\D \dot{u}\|_{L^2}^2.
\end{align}

\textbf{Step 5}: The dissipation inequality of $\|\sqrt{\rho}\dot u\|_{L^2}$.

From (\ref{e3.308}), we have
\begin{align}\label{Lem:4.1:7}
&\frac{\dif}{\dif t}\int\rho|\dot u|^2\dif x+\frac{\mu}{2}\left(\frac{\vr}{2}\right)^\al\|\D\dot u\|_{L^2}^2\nonumber\\
\leq~&C\left(\vr\|\sqrt{\rho}\dot\theta\|_{L^2}^2+\vr^2\|\D\theta\|_{L^2}^2\|\D u\|_{L^2}\|\D u\|_{L^6}+\vr^\al\|\D u\|_{L^2}\|\D u\|_{L^6}^3\right)\nonumber\\
\leq~&C\left(\vr\|\sqrt{\rho}\dot\theta\|_{L^2}^2+\vr\|\D\theta\|_{L^2}^2+\vr^{-3}\|\sqrt{\rho}\dot{u}\|_{L^2}^2+\vr^{-4}\|\nabla\rho\|_{L^2}^2+\vr^{\alpha-3}\|\D u\|_{L^2}^2\right),
\end{align}

\textbf{Step 6}: The dissipation inequality of $\|\sqrt{\rho}\dot \theta\|_{L^2}$.

From (\ref{3.5:5}), we have
\begin{align}\label{Lem:4.1:8}
&\frac{\dif}{\dif t}\int\rho|\dot{\theta}|^2\dif x+\kappa\int|\D\dot{\theta}|^2\dif x\nonumber\\
% \leq~&3\sigma^5\sigma'\int\rho|\dot{\theta}|^2\dif x+C\sigma^6\int|\D\dot{\theta}||\D u||\D\theta|\dif x+C\sigma^6\int\rho|\dot{\theta}|^2|\D u|\dif x\nonumber\\
% &+C\sigma^6\int\rho|\theta-1||\dot{\theta}|\left(|\D\dot{u}|+|\D u|^2\right)\dif x+C\sigma^6\int\rho|\dot{\theta}|\left(|\D\dot{u}|+|\D u|^2\right)\dif x\nonumber\\
% &+C\sigma^6\int\rho^\al|\dot{\theta}||\D u|\left(|\D\dot{u}|+|\D u|^2\right)\dif x\nonumber\\
% \leq~&\frac{\kappa\sigma^6}{2}\|\D\dot{\theta}\|_{L^2}^2+3\sigma^5\sigma'\int\rho|\dot{\theta}|^2\dif x+C\sigma^6\|\D u\|_{L^2}\|\D u\|_{L^6}\|\D^2\theta\|_{L^2}^2\nonumber\\
% &+C\sigma^6\vr\|\D u\|_{L^2}\|\D u\|_{L^6}\|\sqrt{\rho}\dot{\theta}\|_{L^2}^2+C\vr^2\sigma^6\|\theta-1\|_{L^2}\|\D\theta\|_{L^2}\|\D\dot{u}\|_{L^2}^2\nonumber\\
% &+C\vr^2\sigma^6\|\theta-1\|_{L^2}\|\D\theta\|_{L^2}\|\D u\|_{L^2}\|\D u\|_{L^6}^3+C\vr^{\frac{1}{2}}\sigma^6\|\sqrt{\rho}\dot{\theta}\|_{L^2}\|\D\dot{u}\|_{L^2}\nonumber\\
% &+C\vr^{\frac{1}{2}}\sigma^6\|\sqrt{\rho}\dot{\theta}\|_{L^2}\|\D u\|_{L^2}^{\frac{1}{2}}\|\D u\|_{L^6}^{\frac{3}{2}}+C\vr^{2\al}\sigma^6\|\D u\|_{L^2}\|\D u\|_{L^6}\|\D\dot{u}\|_{L^2}^2\nonumber\\
% &+C\vr^{\al-\frac{1}{2}}\sigma^6\|\sqrt{\rho}\dot{\theta}\|_{L^2}\|\D u\|_{L^6}^3\nonumber\\
\leq~& C\|\D u\|_{L^2}\|\D u\|_{L^6}\|\D^2\theta\|_{L^2}^2+C\vr\|\D u\|_{L^2}\|\D u\|_{L^6}\|\sqrt{\rho}\dot{\theta}\|_{L^2}^2+C\vr^2\left(1+\|\D\theta\|_{L^2}\right)\|\D\theta\|_{L^2}\|\D\dot{u}\|_{L^2}^2\nonumber\\
&+C\vr^2\left(1+\|\D\theta\|_{L^2}\right)\|\D\theta\|_{L^2}\|\D u\|_{L^2}\|\D u\|_{L^6}^3+C\vr^{\frac{1}{2}}\|\sqrt{\rho}\dot{\theta}\|_{L^2}\|\D\dot{u}\|_{L^2}\nonumber\\
&+C\vr^{\frac{1}{2}}\|\sqrt{\rho}\dot{\theta}\|_{L^2}\|\D u\|_{L^2}^{\frac{1}{2}}\|\D u\|_{L^6}^{\frac{3}{2}}+C\vr^{2\al}\|\D u\|_{L^2}\|\D u\|_{L^6}\|\D\dot{u}\|_{L^2}^2+C\vr^{\al-\frac{1}{2}}\|\sqrt{\rho}\dot{\theta}\|_{L^2}\|\D u\|_{L^6}^3\nonumber\\
\leq~& C\vr^{-1}\|\D^2\theta\|_{L^2}^2+C\|\sqrt{\rho}\dot{\theta}\|_{L^2}^2+C\vr^{-2}\|\D\theta\|_{L^2}^2+C\|\D u\|_{L^2}^2+C\vr^{2\al-1}\|\D\dot{u}\|_{L^2}^2+C\vr^{2\al-5}\|\D u\|_{L^6}^2\nonumber\\
\leq~& C\vr^{2\al-2}\|\D u\|_{L^6}^2+C\|\D\theta\|_{L^2}^2+C\vr\|\D u\|_{L^2}^2+C\|\sqrt{\rho}\dot{\theta}\|_{L^2}^2+C\vr^{2\al-1}\|\D\dot{u}\|_{L^2}^2\nonumber\\
\leq~& C\vr^{2\al-2}\left(\vr^{1-2\al}\|\sqrt{\rho}\dot{u}\|_{L^2}^2+\vr^{2-2\al}\|\D\theta\|_{L^2}^2+\vr^{-2\al}\|\nabla\rho\|_{L^2}^2+\vr^{-2}\|\D u\|_{L^2}^2\right)+C\|\D\theta\|_{L^2}^2+C\vr\|\D u\|_{L^2}^2\nonumber\\
&+C\|\sqrt{\rho}\dot{\theta}\|_{L^2}^2+C\vr^{2\al-1}\|\D\dot{u}\|_{L^2}^2\nonumber\\
\leq~& C\vr^{-1}\|\sqrt{\rho}\dot{u}\|_{L^2}^2+C\left(\|\D\theta\|_{L^2}^2+\vr^{-2}\|\nabla\rho\|_{L^2}^2+\vr^{2\alpha-4}\|\D u\|_{L^2}^2+\|\sqrt{\rho}\dot{\theta}\|_{L^2}^2+\vr^{2\al-1}\|\D\dot{u}\|_{L^2}^2\right),
\end{align}

% \textbf{Step 7}: The dissipation inequality of $\|\nabla^2\rho\|_{L^2}$.

% From (\ref{sec3.4:2}), we have
% \begin{align}\label{sec3.4:2-1}
% &\frac{\dif}{\dif t}\|\D^2\rho\|_{L^2}^{2}+\frac{1}{2}\vr^{1-\al}\int|\D^2\rho|^{2}\dif x\nonumber\\
% \leq~&C\left(\vr^{1-\al}\|\D\theta\|_{L^2}\|\D^2\theta\|_{L^2}+\vr^{\al-1}\|\D u\|_{L^2}^\frac{2}{7}\|\D^2 u\|_{L^4}^\frac{12}{7}\right)\|\D^2\rho\|_{L^2}^2+C\vr^{\al-1}\|\D^2 u\|_{L^4}^2\|\D \rho\|_{L^4}^2\nonumber\\
% &+C\left(\vr^{1-\al}\|\D {\rho}\|_{L^3}^2\|\D\dot{u}\|_{L^2}^2
% +\vr^{3-\al}\|\D\dot{u}\|_{L^2}^2+\vr^{\alpha-3}\|\D \rho \|_{L^4}^4\|\D u\|_{L^\infty}^2\right)+C\vr^{3-\al}\|\D^2 \theta\|_{L^2}^2\nonumber\\
% &+ C\vr^{1-\al}\|\D^2\theta\|_{L^2}^2\|\D \rho\|_{L^3}^2+C\vr^{-1-\al}\|\D\rho\|_{L^4}^4\left(1+\|\D\theta\|_{L^2}\|\D^2\theta\|_{L^2}\right),
% \end{align}

Then, multiplying $\vr^{-\alpha}$ on (\ref{Lem:4.1:7}), multiplying $\vr^{-2\alpha}$ on (\ref{Lem:4.1:8}), and adding the resulting inequality to (\ref{Lem4.1:6-1}), we have
\begin{align*}
&\frac{\dif}{\dif t}\int\vr^{\alpha+1}\left(\frac{1}{2}\rho|u|^2+R(\rho\ln\rho-\rho-\rho\ln\vr+\vr)+\rho(\theta-\ln\theta-1)\right)\dif x\nonumber\\
&+\frac{\dif}{\dif t}\left(\mu\int\rho^\al|\frD(u)|^2\dif x+\frac{\lambda}{2}\int\rho^\al(\div u)^2\dif x-R\int(\rho\theta-\vr)\div u\dif x\right)+\frac{\dif}{\dif t}\left(\frac{\kappa}{2}\vr^{\frac{5}{2}-\alpha}\int|\D\theta|^2\dif x\right)
\nonumber\\
&+\frac{\dif}{\dif t}\|\D\rho\|_{L^2}^2+\frac{\dif}{\dif t}\left(\vr^{-\alpha}\int\rho|\dot u|^2\dif x\right)+\frac{\dif}{\dif t}\left(\vr^{-2\alpha}\int\rho|\dot{\theta}|^2\dif x\right)\nonumber\\
&+\frac{C}{4}\vr^{2\al}\int|\nabla u|^2\dif x+\frac{C}{4}\vr^{\al-1}\int|\D\theta|^2\dif x+\frac{1}{4}\int\rho|\dot{u}|^2\dif x+\frac{1}{4}\vr^{\frac{5}{2}-\alpha}\int\rho|\dot{\theta}|^2\dif x+\frac{1}{32}\vr^{1-\al}\|\D\rho\|_{L^2}^2\nonumber\\
&+\frac{1}{4}C\|\D\dot u\|_{L^2}^2+\vr^{-2\alpha}\kappa\int|\D\dot{\theta}|^2\dif x\nonumber\\
\leq~& 0.
\end{align*}
Hence, we complete the proof of Lemma \ref{Lem:4.1}.
\end{proof}

The next lemma is concerned with the low frequency part of the solution.
\begin{lemma}\label{Lem4.2}
Let $(\rho, u, \theta)$ be the global strong solution obtained in Theorem \ref{th}. Moreover, if the initial data satisfy $\rho_0-\vr$, $\rho_0 u_0$, $\rho_0 (\theta_0-1)\in L^{p_0}(\R^3)$ for $1\leq {p_0}\leq 2$, we have
\begin{align}\label{fl2-0}
&\int_{S(t)} \left(R|\reallywidehat{\rho-\vr}(\xi,t)|^2+|\reallywidehat{\rho u}(\xi,t)|^2+ |\reallywidehat{\rho(\theta-1)}(\xi,t)|^2\right) \,\dif\xi\nonumber\\
&+\mu\vr^{\alpha+1}\int_0^t \int_{S(t)} |\xi|^2 |\hat{u}|^2 \,\dif\xi \dif s+(\mu+\lambda)\vr^{\alpha+1}\int_0^t \int_{S(t)} |\xi\cdot\hat{u}|^2 \,\dif\xi \dif s+\kappa\vr\int_0^t \int_{S(t)}|\xi|^2\big|\reallywidehat{\theta-1}\big|^2\,\dif\xi \dif s\nonumber\\
\leq~& C\big(\|\rho_0-\vr\|_{L^{p_0}}^2+ \|\rho_0 u_0\|_{L^{p_0}}^2+ \|\rho_0(\theta_0-1)\|_{L^{p_0}}^2\big)(1+t)^{-2\beta({p_0})}\nonumber\\
&+C(\vr)(1+t)^{-\frac 12}\int_0^t \left(\|\rho-\vr\|_{L^2}^2+\|\theta-1\|_{L^2}^2+\|\nabla u\|_{L^2}^2\right)\|\nabla u\|_{L^2}^2\,\dif s\nonumber\\
&+C(\vr)(1+t)^{-\frac 32}\int_0^t \left(\|\rho-\vr\|_{L^2}^4+\|u\|_{L^2}^4+\|\theta-1\|_{L^2}^4+\|\nabla \rho\|_{L^2}^4+\|\nabla u\|_{L^2}^4\right)\,\dif s
\end{align}
with $\displaystyle S(t)\overset{\text{def}}{=}\left\{\xi: |\xi|^2\leq \frac{\overline{C}_0(\vr)}{(1+t)}\right\}$, $\displaystyle\beta({p_0})=\frac{3}{4}\left(\frac{2}{p_0}-1\right)$, where the constant $\overline{C}_0(\vr)$ is to be determined later.
\end{lemma} 
\begin{proof}
We take the Fourier transform of  (\ref{FCNS}), and then
multiply $R\overline{\reallywidehat{\rho-\vr}}$ to the first equation, $\overline{\reallywidehat{\rho u}}$ to the second equation, and $\overline{\reallywidehat{\rho(\theta-1)}}$ to the third equation, respectively, we obtain
that
\begin{equation}\label{FT:1}
\begin{cases}
 \displaystyle\frac{R}{2} \left(|\reallywidehat{\rho-\vr}|^2\right)_t + Re\left(iR\xi \cdot \reallywidehat{\rho u} \overline{\reallywidehat{\rho-\vr}}\right)=0,\\[6pt]
\displaystyle\frac{1}{2}\left(|\reallywidehat{\rho u}|^2\right)_t+Re\left(\reallywidehat{\div(\rho u\otimes u)}+\reallywidehat{\D P}-\reallywidehat{\div\T} \right) \cdot \overline{\reallywidehat{\rho u}}=0,\\[6pt]
\displaystyle\frac{1}{2}\left(|\reallywidehat{\rho(\theta-1)}|^2\right)_t+Re\left(\reallywidehat{\div(\rho u(\theta-1))}+\reallywidehat{P\div u}-\reallywidehat{2\mu\rho^\al |\frD(u)|^2}-\reallywidehat{\lambda\rho^\al(\div u)^2}-\reallywidehat{\kappa\Delta\theta}\right)\overline{\reallywidehat{\rho(\theta-1)}}=0.
\end{cases}
\end{equation}
Then, integrating the above equations with $(0,t)\times S(t)$,  we get that
\begin{align}\label{fl2}
&\frac{1}{2}\int_{S(t)} \left(R|\reallywidehat{\rho-\vr}(\xi,t)|^2+|\reallywidehat{\rho u}(\xi,t)|^2+ |\reallywidehat{\rho(\theta-1)}(\xi,t)|^2\right) \,\dif\xi\nonumber\\
=~&\frac{1}{2}\int_{S(t)} \left(R|\reallywidehat{\rho-\vr}(\xi,0)|^2+|\reallywidehat{\rho u}(\xi,0)|^2+ |\reallywidehat{\rho(\theta-1)}(\xi,0)|^2\right) \,\dif\xi+Re\int_0^t \int_{S(t)}\bigg(-iR\xi \cdot \reallywidehat{\rho u} \overline{\reallywidehat{\rho-\vr}}\nonumber\\
&-\left(\reallywidehat{\div(\rho u\otimes u)}+\reallywidehat{\D P}-\reallywidehat{\div\T} \right) \cdot \overline{\reallywidehat{\rho u}}-\bigg(\reallywidehat{\div(\rho u)(\theta-1)}+\reallywidehat{\rho u\cdot\nabla\theta}+\reallywidehat{P\div u}-\reallywidehat{2\mu\rho^\al |\frD(u)|^2}\nonumber\\
&-\reallywidehat{\lambda\rho^\al(\div u)^2}-\reallywidehat{\kappa\Delta\theta}\bigg)\overline{\reallywidehat{\rho(\theta-1)}}\bigg) \,\dif\xi \dif s\nonumber\\
=~&\frac{1}{2}\int_{S(t)} \left(R|\reallywidehat{\rho-\vr}(\xi,0)|^2+|\reallywidehat{\rho u}(\xi,0)|^2+ |\reallywidehat{\rho(\theta-1)}(\xi,0)|^2\right) \,\dif\xi+\sum\limits_{i=1}^9{R}_i.
\end{align}
Next, we give estimates to the terms $R_i ~ (i=1,2,\cdots, 9)$. Noting that the following inequality
\begin{align*}
\big\|\widehat{f}(\xi)\big\|_{L^\infty(\mathbb{R}^n)}\leq\|f(x)\|_{L^1(\mathbb{R}^n)},
\end{align*}
holds for any $f\in L^1$, we get
\begin{align}\label{q1}
R_1+R_3=~&-Re\int_0^t \int_{S(t)}iR\xi \cdot \reallywidehat{\rho u} \overline{\reallywidehat{\rho-\vr}}\,\dif\xi \dif s-Re\int_0^t \int_{S(t)}\reallywidehat{\D P}\cdot \overline{\reallywidehat{\rho u}}\,\dif\xi \dif s\nonumber\\
=~&-Re\int_0^t \int_{S(t)}iR\xi \cdot \reallywidehat{\rho u} \overline{\reallywidehat{\rho-\vr}}\,\dif\xi \dif s-Re\int_0^t \int_{S(t)}iR\xi\reallywidehat{ (\rho\theta-\rho+\rho-\vr)}\cdot \overline{\reallywidehat{\rho u}}\,\dif\xi \dif s\nonumber\\
% &=\left|\int_0^t \int_{S(t)}\reallywidehat{\D (R\rho\theta-R\rho)}\cdot \overline{\reallywidehat{\rho u}}\,\dif\xi \dif s\right|\nonumber\\
=~&-Re\int_0^t \int_{S(t)}iR\xi\reallywidehat{\rho(\theta-1)}\cdot \overline{\reallywidehat{(\rho-\vr+\vr) u}}\,\dif\xi \dif s\nonumber\\
\leq~& \delta\vr^2\int_0^t \int_{S(t)}|\xi|^2|\reallywidehat{u}|^2\,\dif\xi \dif s+\delta\vr\int_0^t \int_{S(t)}|\xi|^2|\reallywidehat{\theta-1}|^2\,\dif\xi \dif s+C\int_0^t \int_{S(t)}\big|\reallywidehat{(\rho-\vr)(\theta-1)}\big|^2\,\dif\xi \dif s\nonumber\\
&+C\int_0^t \int_{S(t)}\left(|\xi|^2+\vr\right)\big|{\reallywidehat{(\rho-\vr) u}}\big|^2\,\dif\xi \dif s-Re\int_0^t \int_{S(t)}i\xi R\vr^2\reallywidehat{(\theta-1)}\cdot \overline{\reallywidehat{u}}\,\dif\xi \dif s\nonumber\\
\leq~& \delta\vr^2\int_0^t \int_{S(t)}|\xi|^2|\reallywidehat{u}|^2\,\dif\xi \dif s+\delta\vr\int_0^t \int_{S(t)}|\xi|^2|\reallywidehat{\theta-1}|^2\,\dif\xi \dif s+C\int_0^t \|\reallywidehat{(\rho-\vr)(\theta-1)}\|_{L^\infty}^2\int_{S(t)}\,\dif\xi \dif s\nonumber\\
&+C\vr\int_0^t \|{\reallywidehat{(\rho-\vr) u}}\|_{L^\infty}^2\int_{S(t)}(|\xi|^2+1)\,\dif\xi \dif s-Re\int_0^t \int_{S(t)}i\xi R\vr^2\reallywidehat{(\theta-1)}\cdot \overline{\reallywidehat{u}}\,\dif\xi \dif s\nonumber\\
\leq~& \delta\vr^2\int_0^t \int_{S(t)}|\xi|^2|\reallywidehat{u}|^2\,\dif\xi \dif s+\delta\vr\int_0^t \int_{S(t)}|\xi|^2|\reallywidehat{\theta-1}|^2\,\dif\xi \dif s\\
&+C\vr(1+t)^{-\frac{3}{2}}\int_0^t \left(\|\rho-\vr\|_{L^2}^4+\| u\|_{L^2}^4+\|\theta-1\|_{L^2}^4\right)\,\dif s-Re\int_0^t \int_{S(t)}i\xi R\vr^2\reallywidehat{(\theta-1)}\cdot \overline{\reallywidehat{u}}\,\dif\xi \dif s,\nonumber
% \leq~& \delta\int_0^t \int_{S(t)}|\xi|^2|\reallywidehat{u}|^2\,\dif\xi \dif s+C(1+t)^{-\frac{3}{2}}\int_0^t \|{\theta-1}\|_{L^2}^2\,\dif s+C(1+t)^{-\frac{5}{2}}\int_0^t \|\rho-\vr\|_{L^2}^2\| u\|_{L^2}^2\,\dif s,
\end{align}
where $\delta>0$ is a small constant to be determined later. Next, we have
\begin{align}\label{q2}
	R_2=~&-Re\int_0^t \int_{S(t)}  \reallywidehat{\div(\rho u\otimes u)}\cdot (\overline{\reallywidehat{\vr u}}+\overline{\reallywidehat{(\rho-\vr)u}})\,\dif\xi \dif s\nonumber\\
	\leq~& \delta\vr^2\int_0^t \int_{S(t)} |\xi|^2 |\hat{u}|^2 \,\dif\xi \dif s  +C\int_0^t \int_{S(t)} \big| \reallywidehat{\rho u\otimes u} \big|^2 \,\dif\xi \dif s +\int_0^t \int_{S(t)} |\xi| \big| \reallywidehat{\rho u\otimes u} \big| \big|\reallywidehat{(\rho-\vr)u}\big| \,\dif\xi \dif s\nonumber\\
	 \leq~& \delta\vr^2\int_0^t \int_{S(t)} |\xi|^2 |\hat{u}|^2 \,\dif\xi \dif s +C \int_0^t \|\reallywidehat{\rho u\otimes u}\|_{L^\infty}^2 \int_{S(t)}\,\dif\xi \dif s\nonumber\\
  &+ C(1+t)^{-\frac 12}\int_0^t \|\reallywidehat{\rho u\otimes u}\|_{L^\infty} \|\reallywidehat{(\rho-\vr) u}\|_{L^\infty}\int_{S(t)}\,\dif\xi \dif s \nonumber\\
	 \leq~& \delta \vr^2\int_0^t \int_{S(t)}|\xi|^2 |\hat{u}|^2 \,\dif\xi \dif s+C\vr^2(1+t)^{-\frac 32}\int_0^t \|u\|_{L^2}^4 \dif s + C(1+t)^{-2}\vr\int_0^t \|u\|_{L^2}^3 \|\rho-\vr\|_{L^2}\,\dif s.
	\end{align}

 \begin{align}\label{q4}
	R_4=~&Re\int_0^t \int_{S(t)}  \reallywidehat{\div\T} \cdot \overline{\reallywidehat{\rho u}}\,\dif\xi \dif s=Re\int_0^t \int_{S(t)}  i\xi\cdot\left(\reallywidehat{2\tilde{\mu}(\rho)\frD(u)+\tilde{\lambda}(\rho
)\div u\mathbb{I}_3}\right) \cdot \overline{\reallywidehat{\rho u}}\,\dif\xi \dif s\nonumber\\
\leq~& - Re\int_0^t \int_{S(t)}\mu  |\xi|^2 {\reallywidehat{\rho^\alpha u}}\cdot \overline{\reallywidehat{\rho u}}\,\dif\xi \dif s- Re\int_0^t \int_{S(t)}(\mu+\lambda) i\left(\xi\cdot{\reallywidehat{\rho^\alpha u}}\right)\left(\xi \cdot\overline{\reallywidehat{\rho u}}\right)\dif\xi \dif s\nonumber\\
&+C\left|\int_0^t \int_{S(t)}  |\xi| {\reallywidehat{\rho^{\alpha-1}u \div \rho}}\cdot \overline{\reallywidehat{\rho u}}\,\dif\xi \dif s\right|\nonumber\\
\leq~& - Re\int_0^t \int_{S(t)} \mu |\xi|^2 \left({\reallywidehat{(\rho^\alpha-\vr^\alpha) u}}+{\reallywidehat{\vr^\alpha u}}\right)\cdot \left({\reallywidehat{(\rho-\vr) u}}+{\reallywidehat{\vr u}}\right)\,\dif\xi \dif s\nonumber\\
&- Re\int_0^t \int_{S(t)}(\mu+\lambda) \left(\xi \cdot\left({\reallywidehat{(\rho^\alpha-\vr^\alpha) u}}+{\reallywidehat{\vr^\alpha u}}\right)\right)\left(\xi\cdot \left({\reallywidehat{(\rho-\vr) u}}+{\reallywidehat{\vr u}}\right)\right)\,\dif\xi \dif s\nonumber\\
&+C\left|\int_0^t \int_{S(t)}  |\xi| {\reallywidehat{\rho^{\alpha-1}u \div \rho}}\cdot \left({\reallywidehat{(\rho-\vr) u}}+{\reallywidehat{\vr u}}\right)\,\dif\xi \dif s\right|\nonumber\\
\leq~& -\mu\vr^{\alpha+1}\int_0^t \int_{S(t)} |\xi|^2 |\hat{u}|^2 \,\dif\xi \dif s-(\mu+\lambda)\vr^{\alpha+1}\int_0^t \int_{S(t)} |\xi\cdot\hat{u}|^2 \,\dif\xi \dif s   \nonumber\\
&+C\int_0^t \int_{S(t)} |\xi|^2\big| {\reallywidehat{(\rho^\alpha-\vr^\alpha) u}}\big|\big|{\reallywidehat{(\rho-\vr) u}}\big| \,\dif\xi \dif s+C\int_0^t \int_{S(t)} |\xi|^2\big| {\reallywidehat{(\rho^\alpha-\vr^\alpha) u}}\big|\big|{\reallywidehat{\vr u}}\big| \,\dif\xi \dif s\nonumber\\
&+C\int_0^t \int_{S(t)} |\xi|^2\big| {\reallywidehat{\vr^\alpha u}}\big|\big|{\reallywidehat{(\rho-\vr) u}}\big| \,\dif\xi \dif s+C\int_0^t \int_{S(t)} |\xi|\big|{\reallywidehat{\rho^{\alpha-1}u \div \rho}}\big|\big|{\reallywidehat{(\rho-\vr) u}}\big| \,\dif\xi \dif s\nonumber\\
&+C\int_0^t \int_{S(t)} |\xi|\big|{\reallywidehat{\rho^{\alpha-1}u \div \rho}}\big|\big|{\reallywidehat{\vr u}}\big| \,\dif\xi \dif s\nonumber\\
	 \leq~& -\mu\vr^{\alpha+1}\int_0^t \int_{S(t)} |\xi|^2 |\hat{u}|^2 \,\dif\xi \dif s-(\mu+\lambda)\vr^{\alpha+1}\int_0^t \int_{S(t)} |\xi\cdot\hat{u}|^2 \,\dif\xi \dif s   \nonumber\\
  &+ C(1+t)^{-1}\int_0^t \|{\reallywidehat{(\rho^\alpha-\vr^\alpha) u}}\|_{L^\infty} \|\reallywidehat{(\rho-\vr) u}\|_{L^\infty}\int_{S(t)}\,\dif\xi \dif s+2\delta\vr^{\al+1}\int_0^t \int_{S(t)} |\xi|^2 |\hat{u}|^2 \,\dif\xi \dif s \nonumber\\
  &+ C\vr^{1-\al}(1+t)^{-1}\int_0^t \|{\reallywidehat{(\rho^\alpha-\vr^\alpha) u}}\|_{L^\infty}^2\int_{S(t)}\,\dif\xi \dif s+ C\vr^{\al-1}(1+t)^{-1}\int_0^t \|\reallywidehat{(\rho-\vr) u}\|_{L^\infty}^2\int_{S(t)}\,\dif\xi \dif s\nonumber\\
  &+C(1+t)^{-\frac{1}{2}}\int_0^t \|{\reallywidehat{\rho^{\alpha-1}u \div \rho}}\|_{L^\infty}\|\reallywidehat{(\rho-\vr) u}\|_{L^\infty}\int_{S(t)}  \,\dif\xi \dif s+C\vr^{1-\al}\int_0^t \|{\reallywidehat{\rho^{\alpha-1}u \div \rho}}\|_{L^\infty}^2\int_{S(t)}  \,\dif\xi \dif s\nonumber\\
	 \leq~& -\frac{\mu}{2}\vr^{\alpha+1}\int_0^t \int_{S(t)} |\xi|^2 |\hat{u}|^2 \,\dif\xi \dif s-(\mu+\lambda)\vr^{\alpha+1}\int_0^t \int_{S(t)} |\xi\cdot\hat{u}|^2 \,\dif\xi \dif s   \nonumber\\
  &+ C\vr^{\al-1}(1+t)^{-\frac{5}{2}}\int_0^t \|u\|_{L^2}^2 \|\rho-\vr\|_{L^2}^2\,\dif s+C\vr^{\al-1}(1+t)^{-2}\int_0^t \|u\|_{L^2}^2 \|\nabla \rho\|_{L^2}\| \rho-\vr\|_{L^2}\,\dif s\nonumber\\
  &+C\vr^{\al-1}(1+t)^{-\frac{3}{2}}\int_0^t \|u\|_{L^2}^2 \|\nabla \rho\|_{L^2}^2\,\dif s \nonumber\\
   \leq~& -\frac{\mu}{2}\vr^{\alpha+1}\int_0^t \int_{S(t)} |\xi|^2 |\hat{u}|^2 \,\dif\xi \dif s-(\mu+\lambda)\vr^{\alpha+1}\int_0^t \int_{S(t)} |\xi\cdot\hat{u}|^2 \,\dif\xi \dif s   \nonumber\\
  &+ C\vr^{\al-1}(1+t)^{-\frac{5}{2}}\int_0^t \|u\|_{L^2}^2 \|\rho-\vr\|_{L^2}^2\,\dif s+C\vr^{\al-1}(1+t)^{-\frac{3}{2}}\int_0^t \|u\|_{L^2}^2 \|\nabla \rho\|_{L^2}^2\,\dif s.
  \end{align}

  \begin{align}\label{q5}
	R_5=~&-Re\int_0^t \int_{S(t)}\reallywidehat{\div(\rho u(\theta-1))}\overline{\reallywidehat{\rho(\theta-1)}}\,\dif\xi \dif s\nonumber\\
 =~&-Re\int_0^t \int_{S(t)}i\xi\cdot\reallywidehat{\rho u(\theta-1)}\left(\overline{\reallywidehat{(\rho-\vr)(\theta-1)}}+\overline{\reallywidehat{\vr(\theta-1)}}\right)\,\dif\xi \dif s\nonumber\\
	\leq~& \delta\vr\int_0^t \int_{S(t)} |\xi|^2 |\reallywidehat{\theta-1}|^2 \,\dif\xi \dif s  +C\vr\int_0^t \int_{S(t)} \big|\reallywidehat{\rho u(\theta-1)}\big|^2 \,\dif\xi \dif s \nonumber\\
 &+\int_0^t \int_{S(t)} |\xi| \big|\reallywidehat{\rho u(\theta-1)}\big| \big|\overline{\reallywidehat{(\rho-\vr)(\theta-1)}}\big| \,\dif\xi \dif s\nonumber\\
	 \leq~& \delta\vr\int_0^t \int_{S(t)} |\xi|^2 |\reallywidehat{\theta-1}|^2 \,\dif\xi \dif s +C \vr\int_0^t \|\reallywidehat{\rho u(\theta-1)}\|_{L^\infty}^2 \int_{S(t)}\,\dif\xi \dif s\nonumber\\
  &+ C(1+t)^{-\frac 12}\int_0^t \|\reallywidehat{\rho u(\theta-1)}\|_{L^\infty} \|\overline{\reallywidehat{(\rho-\vr)(\theta-1)}}\|_{L^\infty}\int_{S(t)}\,\dif\xi \dif s \nonumber\\
	 \leq~& \delta \vr\int_0^t \int_{S(t)}|\xi|^2 |\hat{u}|^2 \,\dif\xi \dif s+C\vr^3(1+t)^{-\frac 32}\int_0^t \|u\|_{L^2}^2\|\theta-1\|_{L^2}^2 \dif s \nonumber\\
  &+ C\vr(1+t)^{-2}\int_0^t \|u\|_{L^2} \|\rho-\vr\|_{L^2} \|\theta-1\|_{L^2}^2\,\dif s.
	\end{align}

  \begin{align}\label{q7}
	R_6=~&-Re\int_0^t \int_{S(t)}\reallywidehat{P\div u}\overline{\reallywidehat{\rho(\theta-1)}}\,\dif\xi \dif s\nonumber\\
 =~&-Re\int_0^t \int_{S(t)}\left(\reallywidehat{(P-R\vr)\div u}+\reallywidehat{R\vr\div u}\right)\left(\overline{\reallywidehat{(\rho-\vr)(\theta-1)}}+\overline{\reallywidehat{\vr(\theta-1)}}\right)\,\dif\xi \dif s\nonumber\\
 =~&-Re\int_0^t \int_{S(t)}\left(\reallywidehat{R\rho(\theta-1)\div u}+\reallywidehat{R(\rho-\vr)\div u}+R\vr i\xi\cdot\reallywidehat{u}\right)\left(\overline{\reallywidehat{(\rho-\vr)(\theta-1)}}+\overline{\reallywidehat{\vr(\theta-1)}}\right)\,\dif\xi \dif s\nonumber\\
	\leq~& \delta\vr^2\int_0^t \int_{S(t)} |\xi|^2 |\reallywidehat{u}|^2 \,\dif\xi \dif s  +C\int_0^t \int_{S(t)} \big|\reallywidehat{(\rho-\vr)(\theta-1)}\big|^2 \,\dif\xi \dif s  -Re\int_0^t \int_{S(t)}R\vr^2 i\xi\cdot\reallywidehat{u}\overline{\reallywidehat{(\theta-1)}}\,\dif\xi \dif s \nonumber\\
 &+\int_0^t \int_{S(t)} \left(\big|\reallywidehat{R\rho(\theta-1)\div u}\big|+\big|\reallywidehat{R(\rho-\vr)\div u}\big|\right)\left(\big|{\reallywidehat{(\rho-\vr)(\theta-1)}}\big|+\big|{\reallywidehat{\vr(\theta-1)}}\big|\right)\,\dif\xi \dif s\nonumber\\
	 \leq~& \delta\vr^2\int_0^t \int_{S(t)} |\xi|^2 |\reallywidehat{u}|^2 \,\dif\xi \dif s+\delta\vr\int_0^t \int_{S(t)} |\xi|^2 |\reallywidehat{\theta-1}|^2 \,\dif\xi \dif s +C\int_0^t \|\reallywidehat{(\rho-\vr)(\theta-1)}\|_{L^\infty}^2 \int_{S(t)}\,\dif\xi \dif s\nonumber\\
  &-Re\int_0^t \int_{S(t)}R\vr^2 i\xi\cdot\reallywidehat{u}\overline{\reallywidehat{(\theta-1)}}\,\dif\xi \dif s \nonumber\\
  &+ C\int_0^t \left(\|\reallywidehat{\rho(\theta-1)\div u}\|_{L^\infty}+\|\reallywidehat{(\rho-\vr)\div u}\|_{L^\infty}\right)\|{\reallywidehat{(\rho-\vr)(\theta-1)}}\|_{L^\infty}\int_{S(t)}\,\dif\xi \dif s \nonumber\\
   &+ C\vr\int_0^t \left(\|\reallywidehat{\rho(\theta-1)\div u}\|_{L^\infty}^2+\|\reallywidehat{(\rho-\vr)\div u}\|_{L^\infty}^2\right)\int_{S(t)}|\xi|^{-2}\,\dif\xi \dif s \nonumber\\
	 \leq~& \delta \vr^2\int_0^t \int_{S(t)}|\xi|^2 |\hat{u}|^2 \,\dif\xi \dif s+\delta\vr\int_0^t \int_{S(t)} |\xi|^2 |\reallywidehat{\theta-1}|^2 \,\dif\xi \dif s+C(1+t)^{-\frac 32}\int_0^t \|\rho-\vr\|_{L^2}^2\|\theta-1\|_{L^2}^2 \dif s\nonumber\\
  &+ C(1+t)^{-\frac 32}\int_0^t \left( \vr\|\theta-1\|_{L^2}\|\nabla u\|_{L^2} +\|\rho-\vr\|_{L^2} \|\nabla u\|_{L^2}\right)\|\rho-\vr\|_{L^2} \|\theta-1\|_{L^2}\,\dif s\nonumber\\
  &+ C\vr(1+t)^{-\frac 12}\int_0^t \left( \vr^2\|\theta-1\|_{L^2}^2\|\nabla u\|_{L^2}^2+\|\rho-\vr\|_{L^2}^2\|\nabla u\|_{L^2}^2\right)\,\dif s-Re\int_0^t \int_{S(t)}R\vr^2 i\xi\cdot\reallywidehat{u}\overline{\reallywidehat{(\theta-1)}}\,\dif\xi \dif s \nonumber\\
    \leq~& \delta \vr^2\int_0^t \int_{S(t)}|\xi|^2 |\hat{u}|^2 \,\dif\xi \dif s+\delta\vr\int_0^t \int_{S(t)} |\xi|^2 |\reallywidehat{\theta-1}|^2 \,\dif\xi \dif s+C(1+t)^{-\frac 32}\int_0^t \|\rho-\vr\|_{L^2}^2\|\theta-1\|_{L^2}^2 \dif s\nonumber\\
  % &+ C(1+t)^{-\frac 32}\int_0^t \left( \|\theta-1\|_{L^2}\|\nabla u\|_{L^2} +\|\rho-\vr\|_{L^2} \|\nabla u\|_{L^2}\right)\left( \|\rho-\vr\|_{L^2} \|\theta-1\|_{L^2}\right)\,\dif s\nonumber\\
  &+ C\vr(1+t)^{-\frac 12}\int_0^t \left( \vr^2\|\theta-1\|_{L^2}^2\|\nabla u\|_{L^2}^2+\|\rho-\vr\|_{L^2}^2\|\nabla u\|_{L^2}^2\right)\,\dif s-Re\int_0^t \int_{S(t)}R\vr^2 i\xi\cdot\reallywidehat{u}\overline{\reallywidehat{(\theta-1)}}\,\dif\xi \dif s,
	\end{align}
 where we have used 
 $$\int_{S(t)}|\xi|^{-2}\,\dif\xi\leq C(1+t)^{-\frac{1}{2}}.$$

 \begin{align}\label{q8}
	R_7+R_8=~&Re\int_0^t \int_{S(t)}\left(\reallywidehat{2\mu\rho^\al |\frD(u)|^2}
+\reallywidehat{\lambda\rho^\al(\div u)^2}\right)\overline{\reallywidehat{\rho(\theta-1)}}\,\dif\xi \dif s\nonumber\\
=~&Re\int_0^t \int_{S(t)}\left(\reallywidehat{2\mu\rho^\al |\frD(u)|^2}
+\reallywidehat{\lambda\rho^\al(\div u)^2}\right)\overline{\reallywidehat{(\rho-\vr+\vr)(\theta-1)}}\,\dif\xi \dif s\nonumber\\
\leq~& \delta\vr\int_0^t \int_{S(t)}|\xi|^2\big|\reallywidehat{\theta-1}\big|^2\,\dif\xi \dif s +C\vr\int_0^t \left(\|\reallywidehat{\rho^\al |\frD(u)|^2}\|_{L^\infty}^2+\|\reallywidehat{\rho^\al(\div u)^2}\|_{L^\infty}^2\right)\int_{S(t)}|\xi|^{-2}\,\dif\xi \dif s \nonumber\\
  &+C\int_0^t \left(\|\reallywidehat{\rho^\al |\frD(u)|^2}\|_{L^\infty}+\|\reallywidehat{\rho^\al(\div u)^2}\|_{L^\infty}\right)\|\reallywidehat{(\rho-\vr)(\theta-1)}\|_{L^\infty}\int_{S(t)}\,\dif\xi \dif s\nonumber\\
	 \leq~& \delta\vr\int_0^t \int_{S(t)}|\xi|^2\big|\reallywidehat{\theta-1}\big|^2\,\dif\xi \dif s+C\vr^{2\al+1}(1+t)^{-\frac 12}\int_0^t \|\nabla u\|_{L^2}^4 \dif s\nonumber\\
  &+C\vr^\al(1+t)^{-\frac 32}\int_0^t \|\nabla u\|_{L^2}^2 \|\theta-1\|_{L^2}\|\rho-\vr\|_{L^2} \dif s.
	\end{align}

  \begin{align}\label{q10}
	R_9=~& Re\int_0^t \int_{S(t)}\kappa\reallywidehat{\Delta\theta}\overline{\reallywidehat{\rho(\theta-1)}}\,\dif\xi \dif s=- Re\int_0^t \int_{S(t)}\kappa|\xi|^2\reallywidehat{\theta-1}\left(\overline{\reallywidehat{(\rho-\vr)(\theta-1)}}+\overline{\reallywidehat{\vr(\theta-1)}}\right)\,\dif\xi \dif s\nonumber\\
 =~&-Re\int_0^t \int_{S(t)}\kappa|\xi|^2\reallywidehat{\theta-1}\overline{\reallywidehat{(\rho-\vr)(\theta-1)}}\,\dif\xi \dif s-\kappa\vr\int_0^t \int_{S(t)}|\xi|^2\big|\reallywidehat{\theta-1}\big|^2\,\dif\xi \dif s\nonumber\\
 \leq~& C\int_0^t \int_{S(t)}|\xi|^2\big|{\reallywidehat{(\rho-\vr)(\theta-1)}}\big|^2\,\dif\xi \dif s-\frac{\kappa}{2}\vr\int_0^t \int_{S(t)}|\xi|^2\big|\reallywidehat{\theta-1}\big|^2\,\dif\xi \dif s\nonumber\\
	 \leq~& C(1+t)^{-1}\int_0^t \|{\reallywidehat{(\rho-\vr)(\theta-1)}}\|_{L^\infty}^2\int_{S(t)}\,\dif\xi \dif s -\frac{\kappa}{2}\vr\int_0^t \int_{S(t)}|\xi|^2\big|\reallywidehat{\theta-1}\big|^2\,\dif\xi \dif s\nonumber\\
	 \leq~& C(1+t)^{-\frac{5}{2}}\int_0^t \|\rho-\vr\|_{L^2}^2 \|\theta-1\|_{L^2}^2 \dif s-\frac{\kappa}{2}\vr\int_0^t \int_{S(t)}|\xi|^2\big|\reallywidehat{\theta-1}\big|^2\,\dif\xi \dif s.
\end{align}

Then, choosing $\delta$ small enough, combining (\ref{fl2}) and (\ref{q1})--(\ref{q10}), we have
\begin{align}\label{fl2-1}
&\int_{S(t)} \left(R|\reallywidehat{\rho-\vr}(\xi,t)|^2+|\reallywidehat{\rho u}(\xi,t)|^2+ |\reallywidehat{\rho(\theta-1)}(\xi,t)|^2\right) \,\dif\xi\nonumber\\
&+\mu\vr^{\alpha+1}\int_0^t \int_{S(t)} |\xi|^2 |\hat{u}|^2 \,\dif\xi \dif s+(\mu+\lambda)\vr^{\alpha+1}\int_0^t \int_{S(t)} |\xi\cdot\hat{u}|^2 \,\dif\xi \dif s+\kappa\vr\int_0^t \int_{S(t)}|\xi|^2\big|\reallywidehat{\theta-1}\big|^2\,\dif\xi \dif s\nonumber\\
% \leq~& \int_{S(t)} \left(R|\reallywidehat{\rho-\vr}(\xi,0)|^2+|\reallywidehat{\rho u}(\xi,0)|^2+ |\reallywidehat{\rho(\theta-1)}(\xi,0)|^2\right) \,\dif\xi+C(1+t)^{-\frac{3}{2}}\int_0^t \|{\theta-1}\|_{L^1}^2\,\dif s\nonumber\\
% &+C(1+t)^{-\frac{5}{2}}\int_0^t \|\rho-\vr\|_{L^2}^2\| u\|_{L^2}^2\,\dif s+C(1+t)^{-\frac 32}\int_0^t \|u\|_{L^2}^4 \dif s+C(1+t)^{-\frac{3}{2}}\int_0^t \|u\|_{L^2}^2 \|\nabla \rho\|_{L^2}^2\,\dif s \nonumber\\
% & +C(1+t)^{-\frac 32}\int_0^t \|u\|_{L^2}^2\|\theta-1\|_{L^2}^2 \dif s+C(1+t)^{-\frac 32}\int_0^t \|\rho-\vr\|_{L^2}^2\|\theta-1\|_{L^2}^2 \dif s \nonumber\\
%   &+ C(1+t)^{-\frac 32}\int_0^t \left( \|\theta-1\|_{L^2}^2\|\nabla u\|_{L^2}^2 +\|\rho-\vr\|_{L^2}^2\|\nabla u\|_{L^2}^2\right)\,\dif s+C(1+t)^{-\frac 32}\int_0^t \|\nabla u\|_{L^2}^2 \|\theta-1\|_{L^1} \dif s\nonumber\\
\leq~& C\big(\|\rho_0-\vr\|_{L^{p_0}}^2+ \|\rho_0 u_0\|_{L^{p_0}}^2+ \|\rho_0(\theta_0-1)\|_{L^{p_0}}^2\big)(1+t)^{-2\beta({p_0})}\nonumber\\
&+ C\vr(1+t)^{-\frac 12}\int_0^t \left(\|\rho-\vr\|_{L^2}^2+\vr^2\|\theta-1\|_{L^2}^2+\vr^{2\al}\|\nabla u\|_{L^2}^2\right)\|\nabla u\|_{L^2}^2\,\dif s\nonumber\\
&+ C\vr^\al(1+t)^{-\frac 32}\int_0^t \left(\|\rho-\vr\|_{L^2}^4+\|u\|_{L^2}^4+\|\theta-1\|_{L^2}^4+\|\nabla \rho\|_{L^2}^4+\|\nabla u\|_{L^2}^4\right)\,\dif s,
\end{align}
where we have used 
\begin{align}
&\int_{S(t)} \left( \big|\reallywidehat{\rho-\vr}(\xi,0)\big|^2+\big|\reallywidehat{\rho u}(\xi,0)\big|^2 +\big|\reallywidehat{\rho(\theta-1)}(\xi,0)\big|^2 \right) \,\dif\xi\nonumber\\
 \leq~&  \left(\|\reallywidehat{\rho_0-\vr}\|_{L^{p_0'}}^2+ \|\reallywidehat{\rho_0 u_0}\|_{L^{p_0'}}^2+ \|\reallywidehat{\rho_0(\theta_0-1)}\|_{L^{p_0'}}^2\right) \left( \int_{S(t)}  \,\dif\xi\right)^{1-\frac 2{p_0'}}\nonumber\\
\leq~& C\left(\|\rho_0-\vr\|_{L^{p_0}}^2+\|\rho_0 u_0\|_{L^{p_0}}^2+\|\rho_0(\theta_0-1)\|_{L^{p_0}}^2\right)(1+t)^{-2\beta({p_0})},
\end{align}
due to $\rho_0-\vr$, $\rho_0 u_0$, $\rho_0 (\theta_0-1)$ belong to $L^{p_0}(\R^3)$ for $\displaystyle1\leq {p_0}\leq 2$ and 
 $\displaystyle \frac 1{p_0}+\frac 1{p_0'}=1$.
\end{proof}

Now we are in a position to prove the following convergence rate of the global large strong solution.
\begin{lemma}\label{decayprop}  Let $(\rho, u, \theta)$ be the global strong solution of (\ref{FCNS}) with initial data $(\rho_0, u_0, \theta_0)$ obtained in Theorem \ref{th}. Suppose that $(\rho_0-\vr, u_0, \theta_0-1)\in L^{p_0}(\R^3)$ with $ p_0\in [1,2]$. Then we have
\begin{equation}\label{decayest2}
\|\rho-\vr\|_{H^1}+\|u\|_{H^1} + \|\theta-1\|_{H^1}+\|\dot{u}\|_{L^2}+\|\dot{\theta}\|_{L^2}\leq C(\vr)(1+t)^{-\beta({p_0})},
\end{equation}
where $\displaystyle\beta({p_0})=\frac 34\left(\frac 2{p_0}-1\right)$.
\end{lemma}
\begin{proof} We separate the proof into several steps.
	
\textbf{Step 1}: From (\ref{fl2-0}) and the fact that $(\rho-\vr, u,\theta-1)$ belong to $L^\infty(0,\infty;H^1)$, $\nabla u$ belongs to $L^2(0,\infty;L^2)$,  we have
\begin{align}
&\int_{S(t)} \left(R|\reallywidehat{\rho-\vr}(\xi,t)|^2+|\reallywidehat{\rho u}(\xi,t)|^2+ |\reallywidehat{\rho(\theta-1)}(\xi,t)|^2\right) \,\dif\xi\nonumber\\
  \leq~& C\big(\|\rho_0-\vr\|_{L^{p_0}}^2+ \|\rho_0 u_0\|_{L^{p_0}}^2+ \|\rho_0(\theta_0-1)\|_{L^{p_0}}^2\big)(1+t)^{-2\beta({p_0})}+ \bar{C}(1+t)^{-\frac 12}\nonumber\\
  \leq~& C(\vr)(1+t)^{-r_1},
\end{align}
where $\displaystyle  r_1=\min\left\{2\beta({p_0}),\frac 12\right\}$. Then, we have
\begin{align}
\vr^2\int_{S(t)}  |\reallywidehat{  u}(\xi,t)|^2  \,\dif\xi \leq~& \int_{S(t)}   |\reallywidehat{\rho u}(\xi,t)|^2  \,\dif\xi+\int_{S(t)} |\reallywidehat{(\rho-\vr) u}(\xi,t)|^2   \,\dif\xi\nonumber\\
\leq~& C(1+t)^{-r_1}+C(1+t)^{-\frac 32}\|\reallywidehat{(\rho-\vr) u}(\xi,t)\|_{L^\infty}^2\nonumber\\
\leq~& C(1+t)^{-r_1},
\end{align}
and 
\begin{align}
\vr^2\int_{S(t)}  |\reallywidehat{\theta-1}(\xi,t)|^2  \,\dif\xi 
\leq~& \int_{S(t)}   |\reallywidehat{\rho(\theta-1)}(\xi,t)|^2  \,\dif\xi+\int_{S(t)} |\reallywidehat{(\rho-\vr)(\theta-1)}(\xi,t)|^2   \,\dif\xi\nonumber\\
\leq~& C(\vr)(1+t)^{-r_1}+C(1+t)^{-\frac 32}\|\reallywidehat{(\rho-\vr)(\theta-1)}(\xi,t)\|_{L^\infty}^2\nonumber\\
\leq~& C(\vr)(1+t)^{-r_1}.
\end{align}

Next, because of $\rho \dot u=-\D P+2\mu\div(\rho^\al\frD(u)) +\lambda\D(\rho^\al\div u)$,  we can obtain
\begin{align}
\int_{S(t)}  |\reallywidehat{  \rho \dot {u} }(\xi,t)|^2  \,\dif\xi 
\leq~&  \int_{S(t)}  \left|\reallywidehat{-\D P}+2\mu\reallywidehat{\div(\rho^\al\frD(u))}+  \lambda\reallywidehat{\D(\rho^\al\div u)}(\xi,t)\right|^2  \,\dif\xi \nonumber\\
\leq~&  C\int_{S(t)} \left( |\xi|^2|\reallywidehat{\rho\theta-\vr}|^2+|\xi|^2|\reallywidehat{\rho^\al\frD(u)}|^2+  |\xi|^2|\reallywidehat{\rho^\al\div u}|^2\right) (\xi,t) \,\dif\xi \nonumber\\
\leq~&  C(1+t)^{-1}\int_{S(t)} \left( |\reallywidehat{\rho\theta-\rho}|^2+|\reallywidehat{\rho-\vr}|^2+|\reallywidehat{\rho^\al\frD(u)}|^2+  |\reallywidehat{\rho^\al\div u}|^2\right) (\xi,t) \,\dif\xi \nonumber\\
\leq~&  C(\vr)(1+t)^{-1-r_1},
\end{align}
which implies that
$$
\int_{S(t)}|\reallywidehat{\dot{u}}(\xi,t)|^2\,\dif\xi\leq C(\vr)(1+t)^{-1-r_1}.
$$

At last, noting \eqref{FCNS:2}$_3$ and using the same argument as above, we get
\begin{align}
&\int_{S(t)}|\reallywidehat{  \rho\dot\theta }(\xi,t)|^2  \,\dif\xi  \nonumber\\
\leq~&  \int_{S(t)}  \left|\reallywidehat{-P\div u}+\reallywidehat{2\mu\rho^\al |\frD(u)|^2}+\reallywidehat{\lambda\rho^\al(\div u)^2}+\reallywidehat{\kappa\Delta\theta}(\xi,t)\right|^2  \,\dif\xi \nonumber\\
\leq~&  C\int_{S(t)} \Big(\left|\reallywidehat{R\rho(\theta-1)\div u}\right|^2+\left|\reallywidehat{R(\rho-\vr)\div u}\right|^2+\left|\reallywidehat{R\vr\div u}\right|^2+\left|\reallywidehat{\rho^\al|\frD(u)|^2}\right|^2+ \left|\reallywidehat{\rho^\al|\div u|^2}\right|+|\xi|^4|\reallywidehat{\theta-1}|^2\Big)  \,\dif\xi \nonumber\\
% \leq~&  C(1+t)^{-1}\int_{S(t)} \left( |\reallywidehat{\rho\theta-\rho|^2}+|\reallywidehat{\rho-\vr}|^2+|\reallywidehat{\rho^\al\frD(u)}|^2+  |\reallywidehat{\rho^\al\div u}|^2\right) (\xi,t) \,\dif\xi \nonumber\\
\leq~& C(\vr)(1+t)^{-\frac{3}{2}}+C\int_{S(t)}|\xi|^2|\hat{u} (\xi,t)|^2 \,\dif\xi ,
\end{align}
which implies that
$$
\int_{S(t)}|\reallywidehat{\dot{\theta}}(\xi,t)|^2\,\dif\xi\leq C(\vr)(1+t)^{-1-r_1}.
$$

Now, we denote
\begin{align}
X(t):=~&\int\vr^{\alpha+1}\left(\frac{1}{2}\rho|u|^2+R(\rho\ln\rho-\rho-\rho\ln\vr+\vr)+\rho(\theta-\ln\theta-1)\right)\dif x\nonumber\\
&+\mu\int\rho^\al|\frD(u)|^2\dif x+\frac{\lambda}{2}\int\rho^\al(\div u)^2\dif x-R\int(\rho\theta-\vr)\div u\dif x+\frac{\kappa}{2}\vr^{\frac{5}{2}-\alpha}\|\D\theta\|_{L^2}^2
\nonumber\\
&+\|\D\rho\|_{L^2}^2+\vr^{-\alpha}\|\sqrt{\rho}\dot{u}\|_{L^2}^2+\vr^{-2\alpha}\|\sqrt{\rho}\dot{\theta}\|_{L^2}^2.
\end{align}
From Young's inequality, \eqref{lower-theta:uniform} and the fact that $\vr$ is large enough, we have
\begin{equation}\label{X:sim}
\overline{C}_1^{-1}(\vr)\left(\|(\rho-\vr, u, \theta-1)\|_{H^1}^2+\|(\dot{u}, \dot{\theta})\|_{L^2}^2\right)\leq X(t)\leq\overline{C}_1(\vr)\left(\|(\rho-\vr, u, \theta-1)\|_{H^1}^2+\|(\dot{u}, \dot{\theta})\|_{L^2}^2\right),
\end{equation}
for some constant $\overline{C}_1(\vr)>0$. From the dissipation inequality (\ref{Lem4.1:6-0}), we have
%\begin{equation}\label{rewrite:0}
%\frac{\dif}{\dif t}X(t)+C^{-1}(\vr)\left(\|\D u\|_{L^2}^2+\|\D\theta\|_{L^2}^2+\|\D\rho\|_{L^2}^2+\|\D\dot{u}\|_{L^2}^2+\|\D\dot{\theta}\|_{L^2}^2\right)\leq 0.
%\end{equation}

%Next, we want to rewrite \eqref{rewrite:0}. From \eqref{estimate:D2theta}, we get
%\begin{equation}\label{rewrite:1}
%\|\D^2\theta\|_{L^2}\leq C(\vr)\left(\|\sqrt{\rho}\dot{\theta}\|_{L^2}+\|\D u\|_{L^2}+\|\D\theta\|_{L^2}+\|\D^2u\|_{L^2}\right).
%\end{equation}
%Applying the standard $L^2$-estimate to the equation \eqref{FCNS:2}$_2$, similar to \eqref{D2u4}, we obtain
%\begin{align}\label{rewrite:2}
%\|\D^2u\|_{L^2}\leq~&C(\vr)\left(\|\sqrt{\rho}\dot{u}\|_{L^2}+\|\D\rho\D u\|_{L^2}+\|\D\rho\|_{L^2}\|\theta\|_{L^\infty}+\|\D\theta\|_{L^2}\right)\nonumber\\
%\leq~&\bar{C}\left(\|\sqrt{\rho}\dot{u}\|_{L^2}+\|\D\rho\|_{L^2}\|\D u\|_{L^2}^{\frac{1}{7}}\|\D^2u\|_{L^2}^{\frac{6}{7}}+\|\D\rho\|_{L^2}+\|\D\theta\|_{L^2}\right)\nonumber\\
%\leq~&\bar{C}\left(\|\sqrt{\rho}\dot{u}\|_{L^2}+\|\D u\|_{L^2}+\|\D\theta\|_{L^2}+\|\D\rho\|_{L^2}\right)+\frac{1}{2}\|\D^2u\|_{L^2}.
%\end{align}
%Inserting \eqref{rewrite:1} and \eqref{rewrite:2} into \eqref{rewrite:0}, we obtain
\begin{equation}\label{rewrite}
\frac{\dif}{\dif t}X(t)+\overline{C}_2^{-1}(\vr)\left(\|\D u\|_{L^2}^2+\|\D\theta\|_{L^2}^2+\|\D\rho\|_{L^2}^2+\|\D\dot{u}\|_{L^2}^2+\|\D\dot{\theta}\|_{L^2}^2\right)\leq 0,
\end{equation}
for some constant $\overline{C}_2(\vr)>0$.

Recalling that $S(t)$ is the ball in $\mathbb{R}^3$ centered at the origin with radius $\displaystyle r(t)=\left(\frac{\overline{C}_0(\vr)}{(1+t)}\right)^{\frac{1}{2}}$, a decomposition of the frequency domain into two time-dependent subdomains by $S(t)^{c}$ and $S(t)$ yields that  
\begin{align}\label{decayeq9}
&\frac{\dif}{\dif t} X(t)+ \frac{\overline{C}_0(\vr)\overline{C}_2^{-1}(\vr)}{1+t} \left(\|u\|_{L^2}^2+\|\theta-1\|_{L^2}^2+\|\rho-\vr\|_{L^2}^2+\|\dot{u}\|_{L^2}^2+\|\dot{\theta}\|_{L^2}^2\right)\nonumber\\
\leq~& \frac{C(\vr)}{1+t}\int_{S(t)}\left( |\reallywidehat{\rho-\vr}(\xi,t)|^2+|\reallywidehat{u}(\xi,t)|^2+|\reallywidehat{\theta-1}(\xi,t)|^2+|\reallywidehat{ \dot{u}}(\xi,t)|^2+|\reallywidehat{ \dot{\theta}}(\xi,t)|^2 \right) \,\dif\xi.
\end{align}
Taking $\overline{C}_0(\vr)=2\overline{C}_1(\vr)\overline{C}_2(\vr)$ and adding \eqref{rewrite} to \eqref{decayeq9}, from \eqref{X:sim} we get
\begin{align}\label{decay:key}
\frac{\dif}{\dif t} X(t)+ \frac{1}{1+t}X(t)\leq~& \frac{C(\vr)}{1+t}\int_{S(t)}\left( |\reallywidehat{\rho-\vr}(\xi,t)|^2+|\reallywidehat{u}(\xi,t)|^2+|\reallywidehat{\theta-1}(\xi,t)|^2+|\reallywidehat{ \dot{u}}(\xi,t)|^2+|\reallywidehat{ \dot{\theta}}(\xi,t)|^2 \right) \,\dif\xi\nonumber\\
\leq~& C(\vr)(1+t)^{-1-r_1}, \qquad\text{for}\quad t\geq\overline{C}_0(\vr).
\end{align}
Multiplying by the integrating factor $(1+t)$ gives
\begin{align}
\frac{\dif}{\dif t}\left[(1+t)X(t)\right]\leq C(\vr)(1+t)^{-r_1},\qquad\text{for}\quad t\geq\overline{C}_0(\vr),
\end{align}
which leads to
\begin{equation}\label{informdecay}
X(t)\leq C(\vr)(1+t)^{-r_1},
\end{equation}
% In particular, we have 
% \begin{equation}\label{decayeq10-0}
% \|\theta-1\|_{L^1}  \leq C(1+t)^{-r_m},
% \end{equation}
and 
\begin{equation}\label{decayeq10}
\|\rho-\vr\|_{L^2}+\|u\|_{L^2}+ \|\theta-1\|_{L^2}+\|\nabla \rho\|_{L^2}+\|\nabla u\|_{L^2}\leq C(\vr)(1+t)^{-r_1/2}.
\end{equation}

\textbf{Step 2}: We want to improve the decay estimate if $\beta(p_0)>\frac 14$. By definition, $r_1=\frac 12$. Thanks to \eqref{fl2-0} and \eqref{decayeq10}, we improve the estimate for the low frequency part as follows
\begin{align}
&\int_{S(t)}\left(|\reallywidehat{\rho-\vr}(\xi,t)|^2+|\reallywidehat{\rho u}(\xi,t)|^2+ |\reallywidehat{\rho(\theta-1)}(\xi,t)|^2\right) \,\dif\xi\nonumber\\
\leq~&  C(\vr)\left((1+t)^{-2\beta({p_0})} + (1+t)^{-1}+ (1+t)^{-\frac32}\log(1+t)\right).
\end{align}
Now following the similar argument used in the previous step, we conclude that
\begin{align}
&\int_{S(t)} \left(|\reallywidehat{\rho-\vr}(\xi,t)|^2+|\reallywidehat{u}(\xi,t)|^2+|\reallywidehat{\theta-1}(\xi,t)|^2+|\reallywidehat{\dot{u}}(\xi,t)|^2+|\reallywidehat{\dot{\theta}}(\xi,t)|^2 \right) \,\dif\xi \nonumber\\
\leq~&  C(\vr)\left((1+t)^{-2\beta({p_0})} + (1+t)^{-1}+ (1+t)^{-\frac32}\log(1+t)\right), 
\end{align}
which implies that
\begin{equation}
\frac{\dif}{\dif t} X(t)+ \frac{1}{1+t} X(t)\leq C(\vr)(1+t)^{-1}\left((1+t)^{-2\beta({p_0})} + (1+t)^{-1}+ (1+t)^{-\frac32}\log(1+t)\right).
\end{equation}
We obtain that 
\begin{equation}\label{decayeq11}
X(t)\leq C(\vr)\max\left\{(1+t)^{-2\beta({p_0})},  (1+t)^{-1}\right\}=C(\vr)(1+t)^{-r_2},
\end{equation}
% In particular, 
% $$\|\theta-1\|_{L^1}\leq C\max\left\{(1+t)^{-2\beta({p_0})}, (1+t)^{-\frac 32}\log(1+t) \right\},$$
with $r_2=\min\{2\beta(p_0), 1\}$, and 
$$\|\rho-\vr\|_{L^2}+\|u\|_{L^2}+ \|\theta-1\|_{L^2}+\|\nabla \rho\|_{L^2}+\|\nabla u\|_{L^2} \leq C(\vr)(1+t)^{-r_2/2}.$$

\textbf{Step 3}: Finally we deal with the case that $\beta(p_0)>\frac 12$. By \eqref{decayeq11}, we have 
% \begin{equation}\label{decayeq10-2}
% \|\theta-1\|_{L^1}  \leq C(1+t)^{-1},
% \end{equation}
% and 
\begin{equation}\label{decayeq10-3}
\|\rho-\vr\|_{L^2}+\|u\|_{L^2}+ \|\theta-1\|_{L^2}+\|\nabla \rho\|_{L^2}+\|\nabla u\|_{L^2}  \leq C(\vr)(1+t)^{-\frac 12}.
\end{equation}
We may repeat the same process in the above to get that
\begin{align}
&\int_{S(t)} \left(|\reallywidehat{\rho-\vr}(\xi,t)|^2+|\reallywidehat{u}(\xi,t)|^2+|\reallywidehat{\theta-1}(\xi,t)|^2+|\reallywidehat{\dot{u}}(\xi,t)|^2+|\reallywidehat{\dot{\theta}}(\xi,t)|^2 \right) \,\dif\xi \nonumber\\
\leq~&  C(\vr)\left((1+t)^{-2\beta({p_0})} +(1+t)^{-\frac 32}\right), 
\end{align}
which implies that
\begin{align}
\frac{\dif}{\dif t}X(t)+\frac{1}{1+t} X(t)\leq C(\vr)(1+t)^{-1}(1+t)^{-2\beta({p_0})}.
\end{align}
It is enough to derive \eqref{decayest2}. We ends the proof to the lemma.
\end{proof}

\bigskip

\noindent {\bf Acknowledgments}\\
The research of this work was supported in part by the National Natural Science Foundation
of China under grants 12371221, 12301277, 12161141004 and 11831011. This work was also partially supported by
the Fundamental Research Funds for the Central Universities and Shanghai Frontiers Science
Center of Modern Analysis.

\bigskip 
 
\noindent{\bf Data Availability Statements}\\ 
Data sharing not applicable to this article as no datasets were generated or analysed during the current study.

\bigskip

\noindent{\bf Conflict of interests}\\
The authors declare that they have no competing interests.

\bigskip

\noindent{\bf Authors' contributions}\\
The authors have made the same contribution. All authors read and approved the final manuscript.

\bigskip

\bibliographystyle{plain}
%\bibliography{references}

\end{document}